\documentclass{article}%
\usepackage{amsmath}
\usepackage{amssymb}
\usepackage{amsmath}
\usepackage{amsfonts}
\usepackage{graphicx}
\usepackage[T1]{fontenc}
\usepackage{amsmath, amsthm, amsfonts, amssymb}
\usepackage{dsfont}
\usepackage{stackrel,stmaryrd}
\usepackage{hyperref}
\usepackage{mathrsfs}
\usepackage{dsfont}
\usepackage{hyperref}
\usepackage{url}
\usepackage{abstract} 
\usepackage{stmaryrd}
\usepackage{authblk}  

\usepackage{vmargin}
\setmarginsrb{2.5cm}{2.5cm}{2.5cm}{2.5cm}{0cm}{0cm}{1.25cm}{1.25cm}

\providecommand{\U}[1]{\protect\rule{.1in}{.1in}}
\newtheorem{theorem}{Theorem}

\newtheorem{corollary}[theorem]{Corollary}

\newtheorem{definition}[theorem]{Definition}

\newtheorem{lemma}[theorem]{Lemma}

\newtheorem{proposition}[theorem]{Proposition}
\newtheorem{remark}[theorem]{Remark}

\begin{document}
	\title{Reflected generalized BDSDEs driven by non-homogeneous Lévy processes and obstacle problems for stochastic integro-PDEs with nonlinear Neumann boundary conditions}
	
	\author[1]{Badr Elmansouri\footnote{Corresponding author.}\footnote{Email: \href{mailto:b.elmansouri@uca.ac.ma}{b.elmansouri@uca.ac.ma}}}
	
	\author[2]{Mohammed Elhachemy\footnote{Email: \href{mailto:mohammed.elhachemy@edu.uiz.ac.ma}{mohammed.elhachemy@edu.uiz.ac.ma}}}
	
	\author[3]{Mohamed Marzougue\footnote{Email: \href{mailto:m.marzougue@uae.ac.ma}{m.marzougue@uae.ac.ma}}}
	
	\author[4]{Mohamed El Jamali\footnote{Email: \href{mailto:m.eljamali@insea.ac.ma}{m.eljamali@insea.ac.ma}}}

	
	
	\affil[1]{Cadi Ayyad University (UCA), National School of Applied Sciences of Marrakech (ENSA-M), BP 575, Avenue Abdelkrim Khattabi, 40000, Guéliz, Marrakech, Morocco.}
	
	\affil[2]{Faculty of Sciences Agadir, Ibn Zohr University, Hay Dakhla, BP8106, Agadir, Morocco.}
	
	\affil[3]{LAR2A Laboratory, Faculty of Sciences Tetouan, Abdelmalek Essaadi University, 93000, Tetouan, Morocco.}
	
	\affil[4]{National Institute of Statistics and Applied Economics, Rabat, B.P. 6217, Morocco.}

	\date{}
	\maketitle
	\begin{abstract}
		We consider reflected generalized backward doubly stochastic differential equations driven by a non-homogeneous Lévy process. Under  stochastic conditions on the coefficients, we prove the existence and uniqueness of a solution. Furthermore, we apply these results to obtain a probabilistic representation for the viscosity solutions of an obstacle problem governed by stochastic integro-partial differential equations with a nonlinear Neumann boundary condition.
		\vspace{0.3cm}
		
		
	\end{abstract}

	\section{Introduction}
	Backward stochastic differential equations (BSDEs) were first introduced as adjoint equations in stochastic control and studied in depth by Bismut \cite{bismut1973conjugate}. Later, Pardoux and Peng \cite{pardoux1990adapted} established existence and uniqueness for nonlinear BSDEs. Since then, BSDEs have been widely applied in finance \cite{el1997backward}, insurance \cite{brachetta2024optimal}, risk management \cite{elliott2015backward}, stochastic optimal control and differential games \cite{HamadeneLepeltierPeng1997,hamadene1995backward}, and in the analysis of partial differential equations (PDEs) \cite{Pardoux1999BSDEsWeakConvHomog,PardouxPeng1992BSDEsQuasilinearPDEs,peng1991probabilistic,peng1992generalized,Peng1992NonlinearFeynmanKac}, where BSDEs are naturally linked to viscosity solutions of certain quasi-linear parabolic PDEs via probabilistic representations that extend the classical Feynman--Kac formula for linear parabolic PDEs.
	
	Motivated by this latter connection, Pardoux and Zhang \cite{pardoux1998generalized} introduced a new class of BSDEs that includes an additional finite-variation term, given by an integral with respect to a continuous increasing process interpreted as the boundary local time of a diffusion. These are called \emph{generalized BSDEs}. This class has proved to be a powerful tool for providing probabilistic representations of solutions to systems of parabolic and elliptic semilinear PDEs. A solution on a time horizon $T\in(0,+\infty)$ with terminal value $\xi$ and drivers $f(\omega,t,y,z)$ and $h(\omega,t,y)$ is a pair of processes $(Y_t,Z_t)_{t\leq T}$ satisfying
	\begin{equation}\label{intro1}
		Y_t=\xi+\int_{t}^{T}f(s,Y_s,Z_s)ds+\int_{t}^{T}h(s,Y_s)d\kappa_s-\int_t^T Z_s dW_s,
	\end{equation}
	a.s. for all $t\in[0,T]$, where $W$ is a Brownian motion, $(\kappa_t)_{t\geq0}$ is a continuous increasing process with $\kappa_0=0$ (the local time), and $(Y,Z,\kappa)$ are adapted to the natural filtration of $W$. This framework was later extended by Pardoux \cite{Pardoux1997GeneralizedDiscontinuousBSDEs}, who considered a generalized BSDE driven by a Brownian motion together with independent Poisson jumps, and by Elmansouri and El Otmani \cite{elmansouri2023generalized,elmansourielotmani2024} to more general filtrations.
	
	The generalized BSDE \eqref{intro1} was further extended to the reflected case by Ren and Xia \cite{ren2006generalized}. Here, the state process $(Y_t)_{t\leq T}$ is constrained to stay above a given barrier (obstacle) $S:=(S_t)_{t\leq T}$, a continuous adapted process, following the formulation initiated by El Karoui et al. \cite{el1997reflected} for classical reflected BSDEs (the case $d\kappa=0$). The solution is now a triple $(Y_t,Z_t,K_t)_{t\leq T}$ of adapted processes solving
	\begin{equation}
		\left\{
		\begin{split}
			&~Y_t=\xi+\int_{t}^{T}f(s,Y_s,Z_s)ds+\int_{t}^{T}h(s,Y_s)d\kappa_s+\left(K_T-K_t\right)-\int_{t}^{T}ZdW_s,\quad t \in [0,T],\\
			&~Y_t \geq S_t,~t \in [0,T], \quad \text{ and } \int_{0}^{T} \left(Y_{s}-S_{s}\right)dK_s=0.
		\end{split}
		\right.
		\label{intro2}
	\end{equation}
	The main motivation for studying the reflected generalized BSDE \eqref{intro2} lies in its applications to optimal stopping and in providing a probabilistic representation of the viscosity solution to obstacle problems for PDEs with nonlinear Neumann boundary conditions. This line of research has attracted considerable attention; in particular, Ren and El Otmani \cite{ren2010generalized} treated the reflected generalized BSDE driven by a L\'{e}vy process, and Elhachemy and El Otmani \cite{elhachemyjiea} considered the Brownian--Poisson setting, generalizing \eqref{intro2} by incorporating jumps together with stochastically monotone generators. In both works \cite{elhachemyjiea,ren2010generalized}, existence and uniqueness were applied to derive probabilistic formulas for the viscosity solution of obstacle problems for partial differential--integral equations with nonlinear Neumann boundary conditions (see, e.g., \cite{barles1997backward,hamadene2016viscosity}).
	
	Beyond the (deterministic) PDE setting, an important direction is the probabilistic representation of solutions to \emph{stochastic} PDEs (SPDEs), which have become a major focus of research and a key modeling tool in the applied sciences. The linear case was initiated by Pardoux \cite{Pardoux1978CRAS,Pardoux1979Stochastics}, who extended the Feynman--Kac formula and established basic existence and uniqueness results for SPDEs (see also \cite{ocone1993stochastic,Rozovskii1990StochasticEvolutionSystems}). Quasilinear parabolic SPDEs were studied by Pardoux and Peng \cite{pardoux1994backward} through the theory of BSDEs. In 1994, Pardoux and Peng \cite{pardoux1994backward} introduced backward doubly stochastic differential equations (BDSDEs) as a generalization of \eqref{intro1}, driven by both forward and backward It\^o integrals. A solution is a pair $(Y_t,Z_t)_{t\leq T}$ such that
	\begin{equation}
		Y_t=\xi+\int_{t}^{T}f(s,Y_s,Z_s)ds+\int_{t}^{T}g(s,Y_s,Z_s)\overleftarrow{dB}_s-\int_{t}^{T}Z_s dW_s,\quad t \in [0,T].
		\label{intro3}
	\end{equation}
	Here, $(W_t)_{t\leq T}$ and $(B_t)_{t\leq T}$ are independent Brownian motions; the integral with respect to $dW$ is the classical forward It\^o integral, while the integral with respect to $\overleftarrow{dB}$ is a backward It\^o integral. Within the framework of Nualart and Pardoux \cite{nualart1988stochastic}, these integrals are instances of the It\^o Skorohod integral. In \cite{pardoux1994backward}, existence and uniqueness for \eqref{intro3} are established under Lipschitz conditions on $f$ and $g$ together with appropriate square-integrability assumptions, and \eqref{intro3} is related to a probabilistic representation for a system of quasilinear SPDEs, thus extending \cite{Pardoux1978CRAS,Pardoux1979Stochastics}.
	
	Motivated by Pardoux \& Zhang \cite{pardoux1998generalized} (for \eqref{intro2}) and Pardoux \& Peng \cite{pardoux1994backward} (for \eqref{intro3}) and their applications to PDEs with nonlinear Neumann boundary conditions and to SPDEs, Boufoussi et al. \cite{boufoussi} studied generalized BDSDEs. They proved existence and uniqueness under Lipschitz drivers and provided a probabilistic representation of stochastic viscosity solutions for semilinear SPDEs with a Neumann boundary condition.
	
	A central ingredient in the BSDE theory is the martingale representation theorem. This result is classical in filtrations generated by a Brownian motion, a Poisson point process, or a Poisson random measure. For (homogeneous) L\'{e}vy processes, Nualart and Schoutens \cite{nualart2000chaotic} established a martingale representation via the chaos expansion; building on this, \cite{bj} proved existence and uniqueness for BSDEs driven by Teugels martingales associated with L\'{e}vy processes possessing moments of all orders. El Otmani \cite{el2006generalized} studied generalized BSDEs driven by Teugels martingales together with an independent Brownian motion, while Hu and Ren \cite{hu2009stochastic} considered generalized BDSDEs driven by L\'{e}vy processes and provided a probabilistic interpretation for solutions to a class of stochastic integro-partial differential equations (SIPDEs) with a nonlinear Neumann boundary condition. Aman \cite{aman2012reflected} treated the reflected case for such generalized BDSDEs and gave a probabilistic representation for reflected SIPDEs with a nonlinear Neumann boundary condition. More recently, El Jamali and El Otmani \cite{jamali2019predictable} extended the predictable representation theorem to non-homogeneous L\'{e}vy processes, using techniques akin to those of Nualart and Schoutens \cite{nualart2000chaotic} in the homogeneous case. As an application, they studied BSDEs driven by a non-homogeneous L\'{e}vy process with a finite set of jump sizes and a driver with a stochastic Lipschitz coefficient $f$, of the form
	\begin{equation}\label{intro5}
		Y_t=\xi+\int_{t}^{T}f(s,Y_s,Z_s)ds-\sum_{k=1}^{d}\int_{t}^{T}Z^{(k)}_s dH^{(k)}_s,\quad t \in [0,T],
	\end{equation}
	where $(H^{(k)})_{k\geq1}$ is a family of martingales, called the orthonormalized power-jump processes, associated with the components of the non-homogeneous L\'{e}vy process; see also their connection with IPDEs. For the generalized BSDE version of \eqref{intro5}, El Jamali \cite{ElJamali2022_GBSDES_TimeInhomogeneousLevy} proved existence and uniqueness under stochastic Lipschitz assumptions on the drivers. More recently, the doubly stochastic case of \eqref{intro5} was studied by Marzougue and Elmansouri \cite{MarzougueElmansouri_JIEA_inpress}, who considered the BDSDE
	\[
	Y_t=\xi+\int_{t}^{T}f(s,Y_s,Z_s)ds+\int_{t}^{T}g(s,Y_s)\overleftarrow{dB}_s-\sum_{k=1}^{d}\int_{t}^{T}Z^{(k)}_s dH^{(k)}_s,\quad t \in [0,T],
	\]
	and established existence and uniqueness for the pair $(Y_t,Z_t)_{t\leq T}$ under stochastic Lipschitz conditions on $f$ and $g$. They also proved a comparison theorem and applied it to obtain existence and uniqueness in the general case where $f$ is continuous with stochastic linear growth, while $g$ remains stochastically Lipschitz.
	
	Motivated by the above contributions, we study in this paper a reflected generalized backward doubly SDE driven by a non-homogeneous L\'{e}vy process (RGBDSDE-NL) of the form
	\begin{equation}
		\left\{
		\begin{split}
			\text{(i).}&~Y_t=\xi+\int_{t}^{T}f(s,Y_s,Z_s)ds+\int_{t}^{T}h(s,Y_s)d\kappa_s+\int_{t}^{T}g(s,Y_s,Z_s)\overleftarrow{dB}_s+\left(K_T-K_t\right)\\
			&\hspace*{1cm}-\sum_{k=1}^{d}\int_{t}^{T}Z^{(k)}_s dH^{(k)}_s,\quad t \in [0,T],\\
			\text{(ii).}&~Y_t \geq S_t,~t \in [0,T],\\
			\text{(iii).}&\mbox{ Skorokhod condition: }\int_{0}^{T} \left(Y_{s}-S_{s}\right)dK_s=0.
		\end{split}
		\right.
		\label{basic equation}
	\end{equation}
	Here,
	\begin{itemize}
		\item the barrier $S:=(S_t)_{t\leq T}$ is a continuous process with $S_t$ being $\mathcal{F}_t$-measurable for each $t\in[0,T]$, where
		$$
		\mathcal{F}_t\triangleq\mathcal{F}_{t}^L \vee \mathcal{F}_{t,T}^B,\quad t \in [0,T],
		$$
		$(B_t)_{t\leq T}$ is a one-dimensional Brownian motion, $(L_{t})_{t\leq T}$ is a non-homogeneous L\'{e}vy process independent of $(B_t)_{t\leq T}$, and $\mathcal{F}_{s,t}^B= \sigma\{B_r-B_s;\; s\leq r\leq t\}$ while $\mathcal{F}_{t}^L:=\mathcal{F}_{0,t}^L$ is augmented with null sets.
		\item the drivers $f$ and $h$ satisfy stochastic monotonicity conditions with respect to the state variable $y$;
		\item the drivers $f$ and $g$ are stochastically Lipschitz with respect to $z$ and $y$, respectively, and $g$ also satisfies a Lipschitz condition in $z$ as in \cite{pardoux1994backward};
		\item $(\kappa_t)_{t\leq T}$ is a continuous increasing process interpreted as the boundary local time of a Markov process driven by a non-homogeneous L\'{e}vy noise.
	\end{itemize}
	
	Under these general assumptions on the data, we establish existence and uniqueness for the RGBDSDE-NL \eqref{basic equation}. This extends and unifies the works \cite{aman2012reflected,boufoussi,pardoux1994backward,ElJamali2022_GBSDES_TimeInhomogeneousLevy,hu2009stochastic} by encompassing previously studied deterministic Lipschitz and monotonicity conditions on the drivers. A second motivation of this work is the connection with SIPDEs. In particular, using the RGBDSDE-NL \eqref{basic equation}, we derive a probabilistic representation of the stochastic viscosity solution to an obstacle problem for stochastic integro-partial differential equations with nonlinear Neumann boundary conditions (SIPDE-NBC) of parabolic type, under suitable assumptions. We note that \cite{aman2012reflected,hu2009stochastic} applied related results to SIPDE-NBCs with a (homogeneous) L\'{e}vy process without providing full details. In the non-homogeneous L\'{e}vy setting with stochastic monotonicity and Lipschitz conditions, the problem is new and technically challenging.
	
	The paper is organized as follows. In Section \ref{Sec1}, we specify the stochastic basis, introduce notation, and prove auxiliary results used throughout the paper. In Section \ref{Sec 3}, we state our assumptions on the data $(\xi,f,h,\kappa,g,S)$ of the RGBDSDE-NL \eqref{basic equation}, present a modified It\^o formula for processes involving forward–backward integrals, and establish existence and uniqueness under our general stochastic conditions. Finally, in Section \ref{Sec4}, we provide a probabilistic interpretation for solutions to a class of stochastic integro-partial differential equations with nonlinear Neumann boundary conditions.
	
	\section{Preliminaries}
	\label{Sec1}
	Let $\mathbb{F}:=\left(\mathcal{F}_t\right)_{t \leq T}$ be a filtration defined on the probability space $(\Omega,\mathcal{F},\mathbb{P})$ that satisfies the usual conditions of right-continuity and completeness. By $\mathcal{B}(\mathbb{R})$, we mean the Borel $\sigma$-algebra of the real line $\mathbb{R}=(-\infty,+\infty)$. All stochastic processes are defined on the time interval $[0,T]$. Additionally, to simplify the notation, we omit the dependence on $\omega$ for any given process or random function. By convention, all brackets and stochastic integrals are assumed to be zero at time zero.
	\begin{definition}
		An $\mathbb{F}$-adapted $\mathbb{R}$-valued process $L:=(L_{t})_{t\leq T}$ is a non-homogeneous L\'{e}vy process, if the following conditions hold:
		\begin{itemize}
			\item $L$ has independent increments, i.e., $L_t-L_s$ is independent of $\mathcal{F}_s$ for any $0 \leq s \leq t \leq T$. 
			
			\item For any $t\in [0,T]$, the law of $L_t$ is characterized by the characteristic function:
			\begin{equation}
				\mathbb{E}\left[e^{iuL_t}\right]=
				\exp\left\{\int_0^t\left(iub_s-\frac{c_s}{2}u^2+\int_{\mathbb{R} }\left(e^{iuz}-1-iuz\mathds{1}_{\{|z|\leq1\}}\right)F_s(dz)\right)ds\right\},
				\label{characteristic}
			\end{equation}
			where $b:[0,T]\to\mathbb{R}$ and $c:[0,T]\to\mathbb{R}^{\ast,+}$ are Borel-measurable functions, and for each $s\in[0,T]$, $F_s$ is a measure on $\mathbb{R}$ that integrates $(1\wedge |z|^2)$ and satisfies $F_s(\{0\})=0$. Moreover, for every $\mathsf{A}\in\mathcal{B}(\mathbb{R})$, the map $s\mapsto F_s(\mathsf{A})$ is Borel-measurable on $[0,T]$. 
		\end{itemize}
		The triplet $(b,c,F):=(b_t,c_t,F_t)_{t \leq T}$ is called the characteristics of $L$.
		\label{Definition of NL}
	\end{definition}
	\begin{remark}
		\begin{itemize}
			\item In Definition \ref{Definition of NL}, if the characteristics $(b_t, c_t, F_t)$ are assumed to be independent of the time variable $t$, we obtain the classical definition of a homogeneous Lévy process.
			
			\item From (\ref{characteristic}), we can easily derive that $L$ is a stochastically continuous process with $L_0 = 0$ a.s. Using this, along with the independent increments property, we conclude that every non-homogeneous Lévy process is an additive process.
		\end{itemize}
		\label{Rmq1}
	\end{remark}

	In Definition \ref{Definition of NL}, we did not impose regularity on the paths of $L$. However, since $L$ is an additive process (see Remark \ref{Rmq1}), it admits an RCLL modification (see \cite[Ch 2. Theorem 11.5]{Sato}). From now on, we consider this modification, ensuring that the left limits and jumps of $L$, denoted by $L_{t-} = \lim_{s \nearrow t} L_s$ and $\Delta L_t = L_t - L_{t-}$, are well-defined for all $t \in [0,T]$.

	In order to obtain some necessary properties of the non-homogeneous Lévy process $L$, we will make the following assumptions for the remainder of the paper:
	\begin{itemize}
		\item[\textsf{(NL1)}] The characteristics $(b_t,c_t,F)_{t \leq T}$ of $L$ satisfy: $$\int_0^T\left(|b_s|+|c_s|+\int_\mathbb{R}\left(1\wedge|z|^2\right)F_s(dz)\right)ds<+\infty$$ 
		
		\item[\textsf{(NL2)}] There are strictly positive constants $\epsilon$, $\mathfrak{C}$ such that 
		$$\int_0^T\int_{\{|z|>1\}}e^{uz}F_s(dz)ds<+\infty,\quad \forall u \in \left[-(1+\epsilon)\mathfrak{C},(1+\epsilon)\mathfrak{C}\right].$$
	\end{itemize}

	The main needed properties of a non-homogeneous Lévy process under (NL1) and (NL2) are summarized in the following lemma:
	\begin{lemma}
		\begin{itemize}
			\item Fix $t \in [0,T]$. Under assumption \textsf{(NL1)}, the distribution of $L_t$ is infinitely divisible with the Lévy–Khintchine triplet $(\mathsf{b},\mathsf{c},\mathsf{F}):=(\mathsf{b}_t,\mathsf{c}_t,\mathsf{F}_t)_{t \leq T}$ given by
			$$
			\mathsf{b}_t = \int_{0}^{t} b_s \, ds, \quad \mathsf{c}_t= \int_{0}^{t} c_s \, ds, \quad \mathsf{F}_t(dz) = \int_{0}^{t} F_s(dz) \, ds.
			$$
			
			\item Under assumption \textsf{(NL2)}, we have 
			$$
			\mathbb{E}\left[e^{uL_t}\right] < +\infty,\quad \quad \forall u \in \left[-(1+\epsilon)\mathfrak{C},(1+\epsilon)\mathfrak{C}\right].
			$$
			As a consequence, the random variables $L_t$ have moments of all orders. In particular, the expected value of $L_t$ is finite.

			\item Under assumptions \textsf{(NL1)} and \textsf{(NL2)}, the process $L$ is a special semimartingale with the canonical representation 
			\begin{equation}
				L_t = \int_{0}^{t} \mathsf{b}_s \, ds + \int_{0}^{t} \sqrt{c_s} \, dW_s + \int_{0}^{t} \int_{\mathbb{R}} z \, \tilde{\mu}(ds, dz).
				\label{Decomposition}
			\end{equation}
			Here, 
			\begin{itemize}
				\item $\mathsf{b}_s := b_s + \int_{\mathbb{R}} z \mathds{1}_{\{\left|z\right| > 1\}} F_s(dz)$, with $b_s$ given in (\ref{characteristic}).
				
				\item $\sqrt{c_s}$ is a measurable version of the square root of $c_s$.
				
				\item $W := (W_t)_{t \leq T}$ is a standard one-dimensional Brownian motion.
				
				\item $\tilde{\mu} := \mu - \nu$ with $\mu$ is the random measure associated with the jumps of $L$ and $\nu$ is the compensator of $\mu$, defined by
				\begin{equation}
					\nu\left(\left[0,t\right] \times A\right) := \int_{0}^{t} \int_{A} F_s(dz) \, ds, \quad A \in \mathcal{B}(\mathbb{R}).
					\label{compensator}
				\end{equation}
			\end{itemize}
		\end{itemize}
		\label{Lemma1}
	\end{lemma}
	
	Throughout this paper, $(B_t)_{t\leq T}$ will denote a standard one-dimensional Brownian motion and $(L_{t})_{t\leq T}$ a non-homogeneous L\'{e}vy process independent of $(B_t)_{t\leq T}$. Let $\mathcal{N}$ denotes the class of $\mathbb{P}$-null sets of $\mathcal{F}$. For each $t\in  [0,T]$, we define
	$$\mathcal{F}_t\triangleq\mathcal{F}_{t}^L \vee \mathcal{F}_{t,T}^B,\quad t \in [0,T],$$
	where for any process $(\eta_t)_{t\leq T}$; $\mathcal{F}_{s,t}^\eta= \sigma\{\eta_r-\eta_s;\; s\leq r\leq t\}\vee\mathcal{N}$, $\mathcal{F}_{t}^\eta:=\mathcal{F}_{0,t}^\eta$. 
	
	Note that the collection $(\mathcal{F}_{t})_{t\leq T}$ is neither increasing nor decreasing, so it does not constitute a filtration. 
	
	We define the power jumps of the non-homogeneous L\'evy process $L$ by
	$$
	L_{t}^{(1)}=L_{t}\quad \mbox{and}\quad L_{t}^{(i)}=\sum_{0< s\leq t}(\Delta L_{s})^{i},\ i\geq 2.
	$$
	
	Let $m_t^{(1)}=\mathbb{E}[L_{t}]$ and $m_t^{(i)}=\mathbb{E}[L^i_{t}]=\int_0^t\int_{-\infty }^{+\infty }x^{i}\nu (dx,ds)$ for $i\geq 2$ and $t \in [0,T]$, with $\nu$ defined by (\ref{compensator}). Note that $m^{(1)}_t$ is well-defined due to Lemma \ref{Lemma1}.
	
	Let us set $X_{t}^{(i)}=L_{t}^{(i)}-m_t^{(i)},\ i\geq1$, the so-called {\em Teugels martingales}. For $i,j\geq 1$, the sharp bracket is given by $[X^{(i)},X^{(j)}]_t=\int_0^tc_sds\mathds{1}_{\{i=j=1\}}+L^{(i+j)}_t$ and the angle bracket is given by $\langle X^{(i)},X^{(j)}\rangle_t=\int_0^tc_sds\mathds{1}_{\{i=j=1\}}+m^{(i+j)}_t$.\\
	With the non-homogeneous L\'evy process $(L_{t})_{t\leq T}$, we associate the family of processes $(H^{(i)})_{i\geq1}$ defined by
	$H_{t}^{(i)}=\sum_{k=1}^{i}\alpha_{i,k}X_{t}^{(k)}$ where the coefficients $\alpha_{i,j}$ correspond to the orthonormalization of the polynomials $1, z, z^{2}, \cdots$  with respect to the measure
	$$\pi ([0,t],dz)=\int_0^tc_s\delta_0(dz)ds+\int_0^tz^2F_s(dz)ds.$$
	More precisely, the polynomials $q_n$ defined by $q_n(z)=\sum_{k=1}^n\alpha_{n,k}z^{k-1}$ are orthonormal with respect to the measure $\pi$, i.e.
	$$ \int_0^t\int_{-\infty }^{+\infty } q_n(z)q_m(z)\pi(ds,dz)=0~ \mbox{ if } n \neq m.$$
	
	We set
	$$
	p_n(z)=zq_{n}(z)=\alpha_{n,n}z^n+\alpha_{n,n-1}z^{n-1}+...+\alpha_{n,1}z.
	$$
	
	The martingales $H^{(i)}$, called the {\em orthonormalized $i$th-power-jump processes}, are strongly orthogonal, meaning that $\langle H^{(i)}, H^{(j)} \rangle = 0$ for any $i \neq j$. Indeed, we have
	\begin{eqnarray*}
		\langle H^{(i)},H^{(j)}\rangle_t&=&\sum_{k=1}^{i}\sum_{k'=1}^{j}\alpha_{i,k}\alpha_{j,k'} \left(\int_0^t\int_{-\infty }^{+\infty }z^{k+k'}F_s(dz)ds+\int_0^tc_sds\mathds{1}_{\{k=k'=1\}}\right)\\
		&=& \int_0^t\int_{-\infty }^{+\infty } \sum_{k=1}^{i}\sum_{k'=1}^{j} \alpha_{i,k}\alpha_{j,k'} z^{k+k'-2}\left\{z^2F_s(dz)\right\}ds+\alpha_{i,1}\alpha_{j,1} \int_{0}^{t} c_s ds \\
		&=&\int_0^t\int_{-\infty }^{+\infty } q_i(z)q_j(z)\pi(dz,ds).
	\end{eqnarray*}
	On the other hand, we have
	\begin{equation*}
		\begin{split}
			[H^{(i)},X^{(j)}]_t
			&=\sum_{k=1}^i \alpha_{i,k}\left(\int_0^tc_sds\mathds{1}_{\{k=j=1\}}+L^{(k+j)}_t\right)\\
			&=\alpha_{i,i}L_t^{(i+j)}+\alpha_{i,i-1}L_t^{(i-1+j)}+...+\alpha_{i,1}L_t^{(1+j)}+\alpha_{i,1}\int_0^tc_sds\mathds{1}_{\{j=1\}}
		\end{split}
	\end{equation*}
	
	Furthermore, for $i,j\geq 1$
	$$
	[H^{(i)},H^{(j)}]_t=\sum_{k=1}^{i}\sum_{k'=1}^{j}\left(\alpha_{i,k}\alpha_{j,k'}L_t^{(k+k')}+\alpha_{i,1}^2\int_0^tc_sds\mathds{1}_{\{k=k'=1\}}\right).
	$$
	
	Throughout this paper, we restrict our analysis to the case where the non-homogeneous Lévy process $L$ admits a finite number, $d-1$, of jump sizes for some $d \geq 1$. We set (see, e.g., \cite[Section 2.3, Example 3]{jamali2019predictable})
	$$
	d\langle H^{(k)},H^{(k')}\rangle_t=|\gamma^{(k)}_t| |\gamma^{(k')}_t| \mathsf{D}_{k,k'}  dt \quad \text{for all} \quad t \in [0,T] \quad \text{and} \quad{ k,k' \in \{1,\dots,d\}},
	$$
	where $(\gamma_t)_{t \leq T}=\left(\gamma^{(1)}_t,\gamma^{(2)}_t,\cdots,\gamma^{(d)}_t\right)_{t \leq T}$ are deterministic continuous functions and $\mathsf{D}_{k,j}$ is the Kronecker delta function of $k$ and $j$.

	\section{RGBDSDE-NL with stochastic monotone coefficients}\label{Sec 3}
	\subsection{Assumptions and technical results} 
	
	Let us now state the assumption imposed on the data of the RGBDSDE-NL \eqref{basic equation}.
	\subsubsection*{\bfseries Assumptions on the data $(\xi,f,h,\kappa,g,S)$}
	First, we define the space $\ell^2=\big\{\mathsf{z}=\{\mathsf{z}_k\}_{1 \leq k\leq d} \in  \mathbb{R}^d;\ \|\mathsf{z}\|_{\ell^2}^2=\sum_{k=1}^{d}|\mathsf{z}_k|^2<+\infty\big\}$.
	\subparagraph*{\bf - Measurability of the data and trajectory properties of the process $(\kappa_t)_{t \leq T}$:}
	\begin{itemize}
		\item $\xi$ is an $\mathcal{F}_T$-measurable random variable.
		\item The process $(\kappa_t)_{t \leq T}$ is a continuous non-decreasing process such that $\kappa_0=0$ and $\kappa_t$ is $\mathcal{F}_t$-measurable for any $t \in [0,T]$;
		\item $f$, $h$ and $g$ are jointly measurable;
		\item $\forall t \in [0,T]$, $y\in \mathbb{R}$, $z \in \ell^2$, the processes $f\left(\cdot,t,y,z\right): \Omega  \rightarrow \mathbb{R}$, $g\left(\cdot,t,y,z\right): \Omega  \rightarrow \mathbb{R}$, and $g\left(\cdot,t,y\right): \Omega \rightarrow \mathbb{R}$ are $\mathcal{F}_t$-measurable.
	\end{itemize}
	\subparagraph*{\bf - Stochastic Monotonicity of $f$ and $h$ in $y$:}
	There exist two processes $\lambda:\Omega \times [0,T] \rightarrow \mathbb{R}$
	and $\varrho:\Omega \times [0,T] \rightarrow \mathbb{R}^{\ast}_-$ such that
	\begin{itemize}
		\item[(i)] $\lambda_t$ and $\varrho_t$ are $\mathcal{F}_t$-measurable for any $t \in [0,T]$.
		\item[(ii)]  For all $y$, $y^{\prime}\in \mathbb{R}$, $z\in \ell^2$, $d\mathbb{P} \otimes dt$-a.e.,
		$$\left(y-y^{\prime}\right)\left(f\left(t,y,z\right)-f\left(t,y^{\prime},z\right)\right)\leq \lambda_t \left|y-y^{\prime} \right|^2 .$$
		\item[(iii)] For all $y$, $y^{\prime} \in \mathbb{R}$, $d\mathbb{P} \otimes d\kappa_t$-a.e.,
		$$\left(y-y^{\prime}\right)\left(h\left(t,y\right)-h\left(t,y^{\prime}\right)\right)\leq \varrho_t \left|y-y^{\prime} \right|^2. $$
	\end{itemize}
	\subparagraph*{\bf - Stochastic Lipschitz condition on $f$ in $z$:}
	There exists a stochastic process $\eta :\Omega \times [0,T] \rightarrow \mathbb{R}^{\ast}_+$ such that
	\begin{itemize}
		\item[(iv)] $\eta_t$ is $\mathcal{F}_t$-measurable for any $t \in [0,T]$,
		\item[(v)]  For all $y \in \mathbb{R}$, $z$, $z^{\prime} \in \ell^2$, $d\mathbb{P} \otimes dt$-a.e.,
		$$\left| f\left(t,y,z\right)-f\left(t,y,z^{\prime}\right)\right| \leq \eta_t \left\|\gamma_t \left(z-z^{\prime}\right) \right\|_{\ell^2}.$$
	\end{itemize}
	\subparagraph*{\bf - Stochastic Lipschitz condition on $g$ in $y$ and a Lipschitz condition in $z$:}
	There exists a stochastic process $\rho:\Omega \times [0,T] \rightarrow \mathbb{R}^{\ast}_+$ and a constant $\alpha \in (0,\frac{1}{2})$ such that
	\begin{itemize}
		\item[(vi)]  $\rho_t$ is $\mathcal{F}_t$-measurable for any $t \in [0,T]$.
		\item[(vii)]  For all $y$, $y' \in \mathbb{R}$, $z$, $z^{\prime} \in \ell^2$,  $d\mathbb{P} \otimes dt$-a.e.,
		$$\left| g\left(t,y,z\right)-g\left(t,y',z^{\prime}\right)\right|^2 \leq \rho_t \left|y-y^{\prime} \right|^2 +\alpha  \left\|\gamma_{t} \left(z-z^{\prime}\right) \right\|^2_{\ell^2} .$$
	\end{itemize}
	\subparagraph*{\bf - Linear growth of $f$ and $h$:}
	For some constant $\zeta >0$, a $\left[ 0,\infty \right)$-valued process $(\phi_t)_{t \leq T}$, and some $\left[ 1,\infty \right)$-valued processes $\left\lbrace  \varphi_t,\psi_t; 0 \leq t \leq T  \right\rbrace $, such that for all $(t,y)\in [0,T] \times \mathbb{R}$, we have
	\begin{itemize}
		\item[(viii)] $\varphi_t$, $\phi_t$, and $\psi_t$ are $\mathcal{F}_t$-measurable for any $t \in [0,T]$;
		\item[(ix)]  $\left| f(t,y,0)  \right| \leq \varphi_t +\phi_t  \left| y \right|$ and
		$\left| h(t,y) \right| \leq \psi_t + \zeta \left| y \right|$.
	\end{itemize}
	\subparagraph*{\bf - Continuity condition on $f$ and $g$:}
	For all $z\in \ell^2$:
	\begin{itemize}
		\item[(x)]   $d\mathbb{P} \otimes dt$-a.e., the mapping $y \mapsto f\left(t,y,z\right) : \mathbb{R} \rightarrow \mathbb{R}$ is continuous.
		\item[(xi)]  $d\mathbb{P} \otimes d  \kappa_t$-a.e., the mapping $y \mapsto h\left(t,y\right) : \mathbb{R} \rightarrow \mathbb{R}$ is continuous.
	\end{itemize}
	\subparagraph*{\bf -Obstacle and integrability assumption:}
	We consider a stochastic continuous process $S:=(S_t)_{t \leq T}$ satisfying 
	\begin{itemize}
		\item $S_t$ is $\mathcal{F}_t$-measurable for any $t \in [0,T]$.
		\item $\xi \leq S_T$.
		
		Throughout the current paper, the process $S$ will be referred to as a barrier or an obstacle.
	\end{itemize}

	Let $(V_t)_{t \leq T}$ and $(Q_t)_{t \leq T}$ be the two continuous non-deceasing stochastic processes defined by
	\begin{equation*}
		V_t:=\int_{0}^{t} a_s^2 ds,\quad \mathcal{Q}_t:=V_t+\kappa_t,
	\end{equation*}
	with $a^2_s:=\left| \lambda_s\right|+\phi_s+\phi^2_s+\rho_s+\eta^2_s$. We assume that, there exists some $\epsilon >0$, such that $a^2_s\geq \epsilon$. Since $\varrho_s<0$, we can also consider $\left|\varrho_s\right| \geq \epsilon$.
	
	Let us set
	$$
	\phi^{\theta,\mu}_t=\theta V_t+\mu \kappa_t
	\mbox{ and }
	\Phi^{\theta,\mu}:=e^{\phi^{\theta,\mu}}
	\mbox{ for } \theta,\mu > 0.
	$$
	We assume that, for any $\theta$, $\mu >0$
	\begin{itemize}
		\item[(xii)]  $\displaystyle \mathbb{E}\left[\Phi^{\theta,\mu}_T|\xi|^2\right]  < + \infty.$\\
		\item[(xiii)] $\displaystyle \mathbb{E}  \left[\int_0^T {\Phi^{\theta,\mu}_t}\left(  \left|\frac{\varphi_t}{a_t}\right|^2  +  \left| g(t,0,0)\right|^2 \right)  d t\right]
		+\mathbb{E} \left[\int_0^T {\Phi^{\theta,\mu}_t} \psi_t^2  d\kappa_t\right]  < +\infty$.\\
		\item[(xiv)]  $\mathbb{E}\left[\sup_{t \in [0,T]}\big|\Phi^{\theta,\mu}_t  \left(S^{+}_{t}\right) \big|^2 \right] < +\infty	$.
	\end{itemize}
	In the rest of this paper, the previous assumptions will be denoted by \textbf{(H-M)}. 
	
	For any $\theta, \mu > 0$, we consider the following spaces:
	\begin{itemize}
		\item $\mathcal{S}^2_{\theta,\mu}$: The space of RCLL  processes $Y: \Omega \times[0, T] \longrightarrow \mathbb{R}$ such that
		\begin{itemize}
			\item $Y_t$ is $\mathcal{F}_t$-measurable for any $t \in [0,T]$.
			
			\item $\displaystyle{
				\left\|Y  \right\|^2_{\mathcal{S}^2_{\theta,\mu}}  = \mathbb{E}\left[\sup_{0\leq t\leq T} \Phi^{\theta,\mu}_t|Y_t|^2 \right] < + \infty.}$
			
			(with the convention $\mathcal{S}^2_\mu:=\mathcal{S}^2_{0,\mu}$ and  $\mathcal{S}^2:=\mathcal{S}^2_{0,0}$).
		\end{itemize}
		
		\item $\mathcal{A}^2$: The space of continuous non-decreasing  processes $K: \Omega \times[0, T] \longrightarrow \mathbb{R}$ such that $K_0=0$ and
		\begin{itemize}
			\item $K_t$ is $\mathcal{F}_t$-measurable for any $t \in [0,T]$.
			
			\item $\displaystyle{
				\left\|K  \right\|^2_{\mathcal{A}^2}  = \mathbb{E}\left|K_T\right|^2  < + \infty.}$
			
			(with the convention $\mathcal{S}^2:=\mathcal{S}^2_{0,0}$).
		\end{itemize}
		
		\item $\mathcal{H}^{2,\mathcal{Q}}_{\theta,\mu}$: The space of jointly measurable stochastic processes $Y: \Omega\times [0, T]\longrightarrow \mathbb{R}$ such that
		\begin{itemize}
			\item $Y_t$ is $\mathcal{F}_t$-measurable for any $t \in [0,T]$.
			
			\item $\displaystyle{
				\left\|Y \right\|^2_{\mathcal{H}^{2,\mathcal{Q}}_{\theta,\mu}}  = \mathbb{E}\int_0^T \Phi^{\theta,\mu}_t|Y_t|^2dQ_t < + \infty.}$
			
			(with the convention $\mathcal{H}^{2,\mathcal{Q}}_\mu := \mathcal{H}^{2,\mathcal{Q}}_{0,\mu}$, $\mathcal{H}^{2,\mathcal{Q}} := \mathcal{H}^{2,\mathcal{Q}}_{0,0}$, and $\mathcal{H}^{2,\kappa}_\mu$ is the same space as $\mathcal{H}^{2,\mathcal{Q}}_{0,\mu}$, with integrability taken with respect to $\kappa$).
		\end{itemize}
		
		\item $\mathcal{H}^{2,\ell^2}_{\theta,\mu}$: The space of jointly measurable $\ell^2$-valued processes $(Z_{t})_{t\leq T}=\left(\{Z^{k}_{t}\}_{1 \leq k\leq d}\right)_{t\leq T}$ such that
		\begin{itemize}
			\item $Z$ is $\mathcal{F}_t$-measurable for any $t \in [0,T]$.
			
			\item $\displaystyle{
				\|Z\|^2_{\mathcal{H}^{2,\ell^2}_\beta}=\mathbb{E}\int_{0}^{T} \Phi^{\theta,\mu}_t\|\gamma_t Z_{t}\|_{\ell^2}^2dt
				=\sum_{k=1}^{d}\mathbb{E}\int_{0}^{T}\Phi^{\theta,\mu}_t|Z_{t}^{(k)}|^2|\gamma_{t}^{(k)}|^2dt<+\infty.}$
			
			(with the convention $\mathcal{H}^{2,\ell^2}_\mu:=\mathcal{H}^{2,\ell^2}_{\mu,0}$ and  $\mathcal{H}^{2,\ell^2}:=\mathcal{H}^{2,\ell^2}_{0,0}$).
		\end{itemize}
	\end{itemize}
	
	Notice that the  space $\mathcal{C}^2_{\theta,\mu}=\mathcal{S}^2_{\theta,\mu}\cap\mathcal{H}^{2,\mathcal{Q}}_{\theta,\mu}$ endowed  with  the  norm
	$  \left\|Y\right\|^2_{\mathcal{C}^2_{\theta,\mu}}=\left\|Y \right\|_{\mathcal{S}^{2}_{\theta,\mu}}^2+\left\|Y \right\|_{\mathcal{H}^{2,\mathcal{Q}}_{\theta,\mu}}^2$
	is a Banach space as is the space
	$\mathcal{B}^2_{\theta,\mu}=\mathcal{C}^2_{\theta,\mu}\times\mathcal{H}^{2,\ell^2}_{\theta,\mu} \times \mathcal{A}^2$
	with the norm
	$  \left\|(Y,Z,K)\right\|^2_{\mathcal{B}^2_{\theta,\mu}}=\left\|Y \right\|_{\mathcal{S}^{2}_{\theta,\mu}}^2+\left\|Y \right\|_{\mathcal{H}^{2,\mathcal{Q}}_{\theta,\mu}}^2 + \left\|Z \right\|_{\mathcal{H}^{2,\ell^2}_{\theta,\mu}}^2+ \left\|K \right\|_{\mathcal{A}^2}^{2}.$
	\begin{definition}
		Let $\theta, \mu > 0$. A solution to the RGBDSDE-NL (\ref{basic equation}) associated with the parameters $(\xi, f, h, \kappa, g, S)$ is a triplet of processes $(Y, Z, K)$ that satisfies (\ref{basic equation}) and belongs to the space $\mathcal{B}^2_{\theta,\mu}$.
	\end{definition}
	\begin{remark}\label{finitness Remark}
		Note that, since $\left\lbrace  \varphi_t,\psi_t;\ 0 \leq t \leq T  \right\rbrace $ are $[1,\infty)$-valued processes, then
		$$
		\mathbb{E}\int_{0}^{T} {\Phi^{\theta,\mu}_s}dQ_s< \infty.
		$$
	\end{remark}
	\begin{remark}
		The Skorokhod condition (\ref{basic equation})-(iii) is equivalent to either $\int_{0}^{T} \left(Y_{s-} - S_{s}\right) dK_s = 0$ or\\ $\int_{0}^{T} \left(Y_{s-} - S_{s-}\right) dK_s = 0$. Indeed, since $S$ is continuous, we have $S_s = S_{s-}$. Moreover, as the reflection process $K$ is also continuous, the distinction between $Y_s$ and $Y_{s-}$ is irrelevant, because the set $\{s \in [0,T] : Y_{s-} \neq Y_s\}$ is finite (the number of jumps is $d-1$) and thus has $dK_s$-measure zero.
	\end{remark}
	
			In the following, let $\mathfrak{c}>0$ denote a finite constant that may take different values from line to line. Additionally, we will use the notation $\mathfrak{c}_{\nu}$ to indicate that the constant $\mathfrak{c}$ depends on a specific set of parameters $\nu$. Finally, it should be noted that references to assumptions \textbf{(H-M)}-(i)-(ix)-(x) are equivalent to referencing \textbf{(H-M)}-(i), \textbf{(H-M)}-(ix) and \textbf{(H-M)}-(x) separately.\\
			
			The following extension of the well-known Itô formula is crucial for the remainder of the paper, particularly in deriving useful a priori estimates between the solutions of the RGBDSDE-NL (\ref{basic equation}). The proof is analogous to the one provided in \cite[Lemma 1.3]{pardoux1994backward}, and is based on the techniques outlined in \cite[Ch. II, Theorem 32]{protter2005stochastic}.
			\begin{lemma}\label{Used Lemma}
				Let $\mathcal{Y} \in \mathcal{S}^2$, $\mathsf{B}, \mathsf{C}, \mathsf{G} \in \mathcal{H}^{2,\mathcal{Q}}$, $\mathcal{K} \in \mathcal{A}^2$, and $\mathsf{Z} \in \mathcal{H}^{2,\ell}$ be such that
				\begin{equation}\label{D1}
					\mathcal{Y}_t =\mathcal{Y}_0 + \int_{0}^{t} \mathsf{B}_s \, ds + \int_{0}^{t} \mathsf{C}_s \, d\kappa_s + \int_{0}^{t} \mathsf{G}_s \, \overleftarrow{dB}_s +\mathcal{K}_t+ \sum_{k=1}^{d} \int_0^t \mathsf{Z}_{s}^{(k)} \, dH_{s}^{(k)}.
				\end{equation}
				
				Then, for any $\Psi \in \mathcal{C}^{1,1,2}\left(\mathbb{R}^3\right)$ and any 2-dimensional continuous non-decreasing process $\left(\mathsf{A}_t\right)_{t \leq T} = \left(\mathsf{V}_t,\mathsf{K}_t\right)_{t \leq T}$ such that $\mathsf{V}_0 = \mathsf{K}_0 = 0$ and $\mathsf{V}_t, \mathsf{K}_t \in \mathcal{F}_t$ for all $t \in [0,T]$, we have
				$$
				\begin{aligned}
					\Psi\left(\mathsf{A}_t,\mathcal{Y}_t\right)
					& =  \Psi\left(0,\mathcal{Y}_0\right)+\int_0^t \frac{\partial \Psi}{\partial x}\left(\mathsf{A}_s,\mathcal{Y}_{s}\right)d\mathsf{V}_s+\int_0^t \frac{\partial \Psi}{\partial y}\left(\mathsf{A}_s,\mathcal{Y}_{s}\right)d\mathsf{K}_s + \int_0^t \frac{\partial \Psi}{\partial z}\left(\mathsf{A}_s,\mathcal{Y}_{s}\right) \mathsf{B}_s \, ds \\
					&+ \int_0^t \frac{\partial \Psi}{\partial z}\left(\mathsf{A}_s,\mathcal{Y}_{s}\right) \mathsf{C}_s \, d\kappa_s +\int_0^t \frac{\partial \Psi}{\partial z}\left(\mathsf{A}_s,\mathcal{Y}_{s}\right) \mathsf{G}_s \, \overleftarrow{dB}_s+\int_0^t \frac{\partial \Psi}{\partial z}\left(\mathsf{A}_s,\mathcal{Y}_{s}\right)  \, d\mathcal{K}_s
					\\
					&+ \sum_{k=1}^{d} \int_0^t \frac{\partial \Psi}{\partial z}\left(\mathsf{A}_s,\mathcal{Y}_{s-}\right) \mathsf{Z}_{s}^{(k)} \, dH_{s}^{(k)}- \frac{1}{2} \int_0^t \frac{\partial^2 \Psi}{\partial z^2}\left(\mathsf{A}_s,\mathcal{Y}_{s-}\right) \left|\mathsf{G}_s\right|^2 \, ds\\
					& + \frac{1}{2} \sum_{k,k'=1}^{d} \int_{0}^{t} \frac{\partial^2 \Psi}{\partial z^2}\left(\mathsf{A}_s,\mathcal{Y}_{s-}\right) \mathsf{Z}_{s}^{(k)} \mathsf{Z}_{s}^{(k')} \, d\big[H^{(k)}, H^{(k')}\big]_{s}\\
					&+\sum_{0< s\leq t}\left\{\Psi(\mathsf{A}_s,\mathcal{Y}_s)-\Psi(\mathsf{A}_s,\mathcal{Y}_{s-})-\frac{\partial \Psi}{\partial z}\left(\mathsf{A}_s,\mathcal{Y}_{s-}\right)\Delta \mathcal{Y}_{s}-\frac{1}{2} \frac{\partial^2 \Psi}{\partial z^2}\left(\mathsf{A}_s,\mathcal{Y}_{s-}\right)\left(\Delta \mathcal{Y}_{s}\right)^2  \right\},
				\end{aligned}
				$$
				with $\Delta \mathcal{Y}_{s}=\sum_{k=1}^{d} \mathsf{Z}_{s}^{(k)} \, \Delta H_{s}^{(k)}$.
			\end{lemma}
			\begin{corollary}\label{Used Coro}
				Let $\mathcal Y$ be a stochastic process of the form \eqref{D1}, and let $\theta,\mu>0$. Then, for every $t\in[0,T]$, almost surely, we have
				\begin{equation*}
					\begin{split}
						\Phi^{\theta,\mu}_t\left|\mathcal{Y}_t\right|^2
						&=\left|\mathcal{Y}_0\right|^2+\theta\int_{0}^{t} \Phi^{\theta,\mu}_s\left|\mathcal{Y}_0\right|^2 dV_s+\mu\int_{0}^{t} \Phi^{\theta,\mu}_s\left|\mathcal{Y}_0\right|^2 d\kappa_s+2\int_{0}^{t} \Phi^{\theta,\mu}_s \mathcal{Y}_s \mathsf{B}_sds\\
						&+2\int_0^t \Phi^{\theta,\mu}_s \mathcal{Y}_s \mathsf{C}_s  d\kappa_s +2\int_0^t \Phi^{\theta,\mu}_s \mathcal{Y}_s \mathsf{G}_s  \overleftarrow{dB}_s+2\sum_{k=1}^{d} \int_0^t \Phi^{\theta,\mu}_s\mathcal{Y}_{s-} \mathsf{Z}_{s}^{(k)} dH_{s}^{(k)} \\
						& - \int_0^t  \Phi^{\theta,\mu}_s \left|\mathsf{G}_s\right|^2 ds+ \sum_{k,k'=1}^{d} \int_{0}^{t} \mathsf{Z}_{s}^{(k)} \mathsf{Z}_{s}^{(k')} \, d\big[H^{(k)}, H^{(k')}\big]_{s}+2\int_0^t \Phi^{\theta,\mu}_s \mathcal{Y}_s d\mathcal{K}_s
					\end{split}
				\end{equation*}
				Moreover, if $\left(\mathcal{Y},\mathsf{Z},\mathsf{B},\mathsf{C},\mathsf{G}\right) \in\left( \mathcal{S}^2_{\theta,\mu} \cap\mathcal{H}^{2,\mathcal{Q}}_{\theta,\mu}\right) \times \mathcal{H}^{2,\ell^2}_{\theta,\mu} \times \left(\mathcal{H}^{2,\mathcal{Q}}_{\theta,\mu}\right)^3$, then using $d\langle H^{(k)}, H^{(k')}\rangle_t = |\gamma^{(k)}_t| |\gamma^{(k')}_t| \mathsf{D}_{k,k'} \, dt$, we have
				\begin{equation*}
					\begin{split}
						\mathbb{E}\Phi^{\theta,\mu}_t\left|\mathcal{Y}_t\right|^2
						&=\mathbb{E}\left|\mathcal{Y}_0\right|^2+\theta\mathbb{E}\int_{0}^{t} \Phi^{\theta,\mu}_s\left|\mathcal{Y}_s\right|^2 dV_s+\mu\mathbb{E}\int_{0}^{t} \Phi^{\theta,\mu}_s\left|\mathcal{Y}_s\right|^2 d\kappa_s+2\mathbb{E}\int_{0}^{t} \Phi^{\theta,\mu}_s \mathcal{Y}_s \mathsf{B}_sds\\
						&+2\mathbb{E}\int_0^t \Phi^{\theta,\mu}_s \mathcal{Y}_s \mathsf{C}_s  d\kappa_s 
						- \mathbb{E}\int_0^t  \Phi^{\theta,\mu}_s \left|\mathsf{G}_s\right|^2 ds+ \sum_{k=1}^{d}\mathbb{E} \int_{0}^{t} \Phi^{\theta,\mu}_s \left|\mathsf{Z}_{s}^{(k)} \gamma^{(k)}_s \right|^2 ds+\mathbb{E}\int_{0}^{t}\Phi^{\theta,\mu}_s\mathcal{Y}_s d\mathcal{K}_s
					\end{split}
				\end{equation*}
			\end{corollary}
			\begin{proof}
				For the first part, it suffices to apply Lemma \ref{Used Lemma} with $\Psi(x,y,z)=e^{\theta x+\mu y}\left|z\right|^2$ and $\left(\mathsf{A}_t\right)_{t \leq T}=(V_t,\kappa_t)_{t \leq T}$. Following this, in the expectation calculation stage, we apply the Burkholder-Davis-Gundy (B-D-G) inequality (see, e.g., \cite[Ch IV. Theorem 48]{protter2005stochastic}), which implies
				\begin{equation*}
					\begin{split}
						\mathbb{E}\sup_{0 \leq t \leq T}\left|\int_{t}^{T}\Phi^{\theta,\mu}_s\mathcal{Y}_s \mathsf{G}_s \overleftarrow{dB}_s\right|
						&\leq \mathfrak{c} \mathbb{E}\sqrt{\int_{0}^{T} \Phi^{2\theta,2\mu}_s\mathcal{Y}_s^2 \mathsf{G}_s^2 ds}\\
						& \leq \frac{\mathfrak{c}}{2}\left(\mathbb{E}\sup_{0 \leq t \leq T}\Phi^{\theta,\mu}_t\left|\mathcal{Y}_t\right|^2+\mathbb{E}\int_{0}^{T}\Phi^{\theta,\mu}_s\left|\mathsf{G}_s\right|^2 ds\right)< +\infty.
					\end{split}
				\end{equation*}
				Hence, $\mathbb{E}\int_{0}^{t}\Phi^{\theta,\mu}_s\mathcal{Y}_s \mathsf{G}_s \overleftarrow{dB}_s = 0$ for any $t \in [0,T]$. Similarly, we can show that $\mathbb{E}\sum_{k=1}^{d}\int_{0}^{t}\Phi^{\theta,\mu}_s\mathcal{Y}_{s-} \mathsf{Z}_{s}^{(k)} dH_{s}^{(k)} = 0$ for any $t \in [0,T]$, using the following:
				\begin{equation*}
					\begin{split}
						\mathbb{E}\sup_{0 \leq t \leq T}\left|\sum_{k=1}^{d}\int_{t}^{T}\Phi^{\theta,\mu}_s\mathcal{Y}_{s-} \mathsf{Z}^{(k)}_s d H^{(k)}_s\right|
						&\leq \mathfrak{c} \mathbb{E}\sqrt{\sum_{k=1}^{d}\int_{0}^{T}\Phi^{2\theta,2\mu}_s \mathcal{Y}_s^2  \left|\mathsf{Z}^{(k)}_s\right|^2 \left|\gamma^{(k)}_s\right|^2 ds}\\ 
						&\leq  \frac{\mathfrak{c}}{2}\left(\mathbb{E}\sup_{0 \leq t \leq T}\Phi^{\theta,\mu}_t\left|\mathcal{Y}_t\right|^2+\mathbb{E}\sum_{k=1}^{d}\int_{0}^{T}\Phi^{\theta,\mu}_s \left|\mathsf{Z}^{(k)}_s\right|^2 \left|\gamma^{(k)}_s\right|^2 ds\right)
						< +\infty.
					\end{split}
				\end{equation*}
			\end{proof}

			\subsection{A priori estimates and uniqueness}
			In this part, we provide some a priori estimates for the solutions of the RGBDSDE-NL \eqref{basic equation}, which will be useful for further applications.
			\begin{proposition}\label{Propo 1}
				Assume that \textbf{(H-M)} holds. Let $\theta, \mu > 0$. Let $(Y, Z, K)$ be a solution to the RGBDSDE-NL (\ref{basic equation}) with parameters $(\xi, f, h, g, \kappa, S)$. Then, there exists a constant $\theta_0 > 4 + \frac{2}{1 - 2\alpha}$ such that for any $\theta > \theta_0$ and $\mu > 0$, we have
				\begin{equation*}
					\begin{split}
						&\mathbb{E}\sup_{0 \leq t \leq T}\Phi^{\theta,\mu}_t\left|Y_t\right|^2+\mathbb{E}\int_{0}^{T}\Phi^{\theta,\mu}_s\left|Y_s\right|^2d \mathcal{Q}_s+\mathbb{E}\int_{0}^{T}\Phi^{\theta,\mu}_s\left\|Z_s \gamma_s \right\|^2_{\ell^2}ds+\mathbb{E}\left[\left|K_T\right|^2\right]\\
						& \leq \mathfrak{c}_{\theta,\mu,\alpha,\zeta}\left(\mathbb{E}\Phi^{\theta,\mu}_T\left|\xi\right|^2+\mathbb{E}\int_{0}^{T}\Phi^{\theta,\mu}_s\left|\frac{\varphi_s}{a_s}\right|^2 ds+\mathbb{E}\int_{0}^{T}\Phi^{\theta,\mu}_s\left|\psi_s\right|^2 d\kappa_s \right.\\
						&\qquad\left.+\mathbb{E}\int_{0}^{T}\Phi^{\theta,\mu}_s\left|g(s,0,0)\right|^2 ds+\mathbb{E}\sup_{0 \leq t \leq T}\left|\Phi^{\theta,\mu}_t\left(S_t\right)^{+}\right|^2\right)
					\end{split}
				\end{equation*}
			\end{proposition}
			\begin{proof}
				Using Corollary \ref{Used Coro}, we have
				\begin{equation}\label{Baic coro app}
					\begin{split}
						&\Phi^{\theta,\mu}_t\left|Y_t\right|^2+\theta \int_{t}^{T}\Phi^{\theta,\mu}_s\left|Y_s\right|^2 dV_s+\mu \int_{t}^{T}\Phi^{\theta,\mu}_s \left|Y_s\right|^2 d \kappa_s+\int_{t}^{T}\Phi^{\theta,\mu}_s\left\|Z_s \gamma_s \right\|^2_{\ell^2}ds\\
						&=\Phi^{\theta,\mu}_T\left|\xi\right|^2+2\int_{t}^{T} \Phi^{\theta,\mu}_s Y_s f(s,Y_s,Z_s)ds+2\int_{t}^{T} \Phi^{\theta,\mu}_s Y_s h(s,Y_s) d \kappa_s\\
						&\quad+2\int_{t}^{T} \Phi^{\theta,\mu}_s Y_s g(s,Y_s,Z_s)\overleftarrow{dB}_s+\int_{t}^{T}\Phi^{\theta,\mu}_s\left|g(s,Y_s,Z_s)\right|^2 ds-2\sum_{k=1}^d \int_{t}^{T}\Phi^{\theta,\mu}_s Y_{s-} Z^{(k)}_sdH^{(k)}_s\\
						&\quad-\sum_{k,k'=1}^{d} \int_{0}^{t}\Phi^{\theta,\mu}_s Z_{s}^{(k)}Z_{s}^{(k')} \, \left(d\big[H^{(k)}, H^{(k')}\big]_{s}-d\big\langle H^{(k)}, H^{(k')}\big\rangle _{s}\right)+2\int_{t}^{T} \Phi^{\theta,\mu}_s Y_s dK_s
					\end{split}
				\end{equation}	
				Using \textbf{(H-M)}-(ii)-(iii)-(v)-(vii) and the basic inequality $2ab \leq \varepsilon a^2+\frac{1}{\varepsilon}b^2$ for any $\varepsilon>0$, we derive
				\begin{equation*}
					\begin{split}
						2 Y_s f(s,Y_s,Z_s)ds & \leq 2 \lambda_s \left|Y_s\right|^2ds+2\left|Y_s\right|\left(\varphi_s+\eta_s \left\|\gamma_s Z_s\right\|_{\ell^2}\right)ds\\
						& \leq \left\{2 \lambda_s+\frac{1}{\varepsilon^{\prime}}\eta_s^2+\varepsilon a^2_s\right\}\left|Y_s\right|^2ds+\frac{1}{\varepsilon}\left|\frac{\varphi_s}{a_s}\right|^2ds+\varepsilon^{\prime}\left\|\gamma_s Z_s\right\|^2_{\ell^2}ds\\
						& \leq \left(2+\varepsilon+\frac{1}{\varepsilon^{\prime}}\right)\left|Y_s\right|^2dV_s+\frac{1}{\varepsilon}\left|\frac{\varphi_s}{a_s}\right|^2ds+\varepsilon^{\prime}\left\|\gamma_s Z_s\right\|^2_{\ell^2}ds
					\end{split}
				\end{equation*}
				\begin{equation*}
					\begin{split}
						2 Y_s h(s,Y_s)d \kappa_s & \leq 2 \varrho_s \left|Y_s\right|^2d\kappa_s+2\left|Y_s\right|\psi_s d\kappa_s\\
						& \leq \varepsilon^{\prime\prime} \left|Y_s\right|^2d\kappa_s+\frac{1}{\varepsilon^{\prime\prime}} \varphi_s^2 d\kappa_s
					\end{split}
				\end{equation*}
				\begin{equation*}
					\begin{split}
						\left|g(s,Y_s,Z_s)\right|^2 ds &\leq 2 \left(\rho_s\left|Y_s\right|^2+\alpha\left\|\gamma_s Z_s\right\|^2_{\ell^2}ds\right)+2 \left|g(s,0,0)\right|^2 ds\\
						& \leq 2\left|Y_s\right|^2dV_s+2\alpha \left\|\gamma_s Z_s\right\|^2_{\ell^2}ds+2 \left|g(s,0,0)\right|^2 ds
					\end{split}
				\end{equation*}
				Then, plugging those in (\ref{Baic coro app}), using the Skorokhod condition and taking expectation, we get
				\begin{equation*}
					\begin{split}
						&\mathbb{E}\left[\Phi^{\theta,\mu}_t\left|Y_t\right|^2\right]+\left(\theta-4-\varepsilon-\frac{1}{\varepsilon^{\prime}}\right)\mathbb{E} \int_{t}^{T}\Phi^{\theta,\mu}_s\left|Y_s\right|^2 dV_s\\
						&\qquad+\left(\mu-\varepsilon^{\prime \prime}\right)\mathbb{E} \int_{t}^{T}\Phi^{\theta,\mu}_s \left|Y_s\right|^2 d \kappa_s+(1-2\alpha-\varepsilon^{\prime})\mathbb{E}\int_{t}^{T}\Phi^{\theta,\mu}_s\left\|Z_s \gamma_s \right\|^2_{\ell^2}ds\\
						&\leq \mathbb{E}\left[\Phi^{\theta,\mu}_T\left|\xi\right|^2\right]+\frac{1}{\varepsilon}\mathbb{E}\int_{t}^{T}\Phi^{\theta,\mu}_s\left|\frac{\varphi_s}{a_s}\right|^2 ds+\frac{1}{\varepsilon^{\prime\prime}}\int_{t}^{T} \Phi^{\theta,\mu}_s \left|\psi_s\right|^2 d \kappa_s\\
						&\qquad+2\mathbb{E}\int_{t}^{T}\Phi^{\theta,\mu}_s \left|g(s,0,0)\right|^2 ds+2\mathbb{E}\int_{t}^{T} \Phi^{\theta,\mu}_s S_s dK_s
					\end{split}
				\end{equation*}	
				Set $\theta_0=4+\epsilon+\frac{1}{\varepsilon^{\prime}}$.  Choosing $\varepsilon=\frac{1}{1-2\alpha}$, $\theta>\theta_0$ , $\varepsilon^{\prime}<1-2\alpha$, and $\varepsilon^{\prime \prime}=\frac{\mu}{2}$, one has for any $\varepsilon^{\prime \prime \prime}>0$
				\begin{equation}\label{Baic coro appv2}
					\begin{split}
						&\mathbb{E}\Phi^{\theta,\mu}_t\left|Y_t\right|^2+\mathbb{E} \int_{t}^{T}\Phi^{\theta,\mu}_s\left|Y_s\right|^2 d \mathcal{Q}_s+\mathbb{E}\int_{t}^{T}\Phi^{\theta,\mu}_s\left\|Z_s \gamma_s \right\|^2_{\ell^2}ds\\
						&\leq \mathfrak{c}_{\theta,\mu,\alpha} \left(\mathbb{E}\Phi^{\theta,\mu}_T\left|\xi\right|^2+\mathbb{E}\int_{t}^{T}\Phi^{\theta,\mu}_s\left|\frac{\varphi_s}{a_s}\right|^2 ds+\int_{t}^{T} \Phi^{\theta,\mu}_s \left|\psi_s\right|^2 d \kappa_s+\mathbb{E}\int_{t}^{T}\Phi^{\theta,\mu}_s \left|g(s,0,0)\right|^2 ds\right)\\
						&\qquad+\varepsilon^{\prime\prime\prime}\mathfrak{c}_{\theta,\mu,\alpha}^2	\mathbb{E}\sup_{t\leq u\leq T}\left|\Phi^{\theta,\mu}_u(S_u)^{+}\right|^2+\frac{1}{\varepsilon^{\prime\prime\prime}}\mathbb{E}\left|K_T-K_t\right|^2
					\end{split}
				\end{equation}	
				Coming back to (\ref{basic equation})-(i), by squaring, we have
				\begin{equation}\label{Sbagha}
					\begin{split}
						\mathbb{E}\left|K_T-K_t\right|^2
						& \leq 6\left(\mathbb{E}\left|Y_t\right|^2+\mathbb{E}\left|\xi\right|^2+\mathbb{E}\left(\int_{t}^{T}f(s,Y_s,Z_s)ds\right)^2+\mathbb{E}\left(\int_{t}^{T}h(s,Y_s)d\kappa_s\right)^2\right.\\
						&\left. +\mathbb{E}\left(\int_{t}^{t}g(s,Y_s,Z_s)\overleftarrow{dB}_s\right)^2+\mathbb{E}\left(\sum_{k=1}^{d}\int_{t}^{T}Z^{(k)}_s dH^{(k)}_s\right)^2 \right)\\
						&\leq 6\left(\mathbb{E}\Phi^{\theta,\mu}_t\left|Y_t\right|^2+\mathbb{E}\Phi^{\theta,\mu}_T\left|\xi\right|^2+\mathbb{E}\left(\int_{t}^{T}f(s,Y_s,Z_s)ds\right)^2+\mathbb{E}\left(\int_{t}^{T}h(s,Y_s)d\kappa_s\right)^2\right.\\
						&\qquad\left. +\mathbb{E}\int_{t}^{T}\Phi^{\theta,\mu}_s \left|g(s,Y_s,Z_s)\right|^2ds+\mathbb{E}\int_{t}^{T}\Phi^{\theta,\mu}_s\left\|Z_s \gamma_s \right\|^2_{\ell^2}ds \right)
					\end{split}
				\end{equation}
				Using \textbf{(H-M)}-(vii)-(ix) and Hölder's inequality we have
				\begin{equation*}
					\begin{split}
						\mathbb{E}\left(\int_{t}^{T}f(s,Y_s,Z_s)ds\right)^2
						& \leq \mathbb{E}\left(\int_{t}^{T}\Phi^{-\theta,\mu}_s dV_s\int_{t}^{T}\Phi^{\theta,\mu}_s\left|\frac{f(s,Y_s,Z_s)}{a_s}\right|^2ds\right)\\
						& \leq \frac{3}{\theta}\mathbb{E}\int_{t}^{T}\Phi^{\theta,\mu}_s\frac{\varphi^2_s+\phi_s^2 \left|Y_s\right|^2+\eta_s^2 \left\|\gamma_s Z_s\right\|^2_{\ell^2}}{a^2_s}ds\\
						& \leq \frac{3}{\theta}\mathbb{E}\int_{t}^{T}\Phi^{\theta,\mu}_s\left|Y_s\right|^2 dV_s+\frac{3}{\theta}\mathbb{E}\int_{t}^{T}\Phi^{\theta,\mu}_s\left|\frac{\varphi_s}{a_s}\right|^2 ds+\frac{3}{\theta}\mathbb{E}\int_{t}^{T}\Phi^{\theta,\mu}_s\left\| Z_s \gamma_s\right\|^2_{\ell^2}ds
					\end{split}
				\end{equation*}
				and
				\begin{equation*}
					\begin{split}
						\mathbb{E}\left(\int_{t}^{T}h(s,Y_s)d\kappa_s\right)^2 & \leq \mathbb{E}\left(\int_{t}^{T}\Phi^{\theta,-\mu}_s d\kappa_s\int_{t}^{T}\Phi^{\theta,\mu}_s\left|h(s,Y_s)\right|^2d\kappa_s\right)\\
						& \leq \frac{2\zeta^2}{\mu}\mathbb{E}\int_{t}^{T}\Phi^{\theta,\mu}_s\left|Y_s\right|^2 d\kappa_s+ \frac{2}{\mu}\mathbb{E}\int_{t}^{T}\Phi^{\theta,\mu}_s \psi^2_sd\kappa_s
					\end{split}
				\end{equation*}
				Next, using ,we have
				\begin{equation*}
					\begin{split}
						\mathbb{E}\int_{t}^{T}\Phi^{\theta,\mu}_s \left|g(s,Y_s,Z_s)\right|^2ds
						& \leq 2\mathbb{E}\int_{t}^{T}\Phi^{\theta,\mu}_s \left(\left|Y_s\right|^2 dV_s+\alpha\left\| Z_s \gamma_s\right\|^2_{\ell^2}ds\right)+2 \mathbb{E}\int_{t}^{T}\Phi^{\theta,\mu}_s \left|g(s,0,0)\right|^2ds\\
						& \leq 2\mathbb{E}\int_{t}^{T}\Phi^{\theta,\mu}_s \left|Y_s\right|^2 dV_s+\mathbb{E}\int_{t}^{T}\Phi^{\theta,\mu}_s\left\| Z_s \gamma_s\right\|^2_{\ell^2}ds+2 \mathbb{E}\int_{t}^{T}\Phi^{\theta,\mu}_s \left|g(s,0,0)\right|^2ds
					\end{split}
				\end{equation*}
				Therefore
				\begin{equation}\label{nmagh}
					\begin{split}
						\mathbb{E}\left|K_T-K_t\right|^2& \leq 6\mathbb{E}\left|Y_t\right|^2+6\mathbb{E}\left|\xi\right|^2+\left(\frac{18}{\theta}+\frac{12 \zeta^2}{\mu}+12\right)\mathbb{E}\int_{t}^{T}\Phi^{\theta,\mu}\left|Y_s\right|^2 d\mathcal{Q}_s\\
						&\qquad+\left(\frac{18}{\theta}+6\right)\mathbb{E}\int_{t}^{T}\Phi^{\theta,\mu}_s\left\| Z_s \gamma_s\right\|^2_{\ell^2}ds+\frac{18}{\theta}\mathbb{E}\int_{t}^{T}\Phi^{\theta,\mu}_s\left|\frac{\varphi_s}{a_s}\right|^2 ds\\
						&\qquad +12 \mathbb{E}\int_{t}^{T}\Phi^{\theta,\mu}_s \left|g(s,0,0)\right|^2ds 
					\end{split}
				\end{equation}
				Plugging this into (\ref{Baic coro appv2}), choosing  $\varepsilon^{\prime\prime\prime}>\frac{18}{\theta}+\frac{12 \zeta^2}{\mu}+12$, we have
				\begin{equation}\label{Baic coro appv2 2}
					\begin{split}
						&\sup_{0 \leq t \leq T}\mathbb{E}\Phi^{\theta,\mu}_t\left|Y_t\right|^2+\mathbb{E} \int_{0}^{T}\Phi^{\theta,\mu}_s\left|Y_s\right|^2 d \mathcal{Q}_s+\mathbb{E}\int_{0}^{T}\Phi^{\theta,\mu}_s\left\|Z_s \gamma_s \right\|^2_{\ell^2}ds\\
						&\leq \mathfrak{c}_{\theta,\mu,\alpha,\zeta} \left(\mathbb{E}\Phi^{\theta,\mu}_T\left|\xi\right|^2+\mathbb{E}\int_{0}^{T}\Phi^{\theta,\mu}_s\left|\frac{\varphi_s}{a_s}\right|^2 ds+\int_{0}^{T} \Phi^{\theta,\mu}_s \left|\psi_s\right|^2 d \kappa_s\right.\\
						&\qquad\left.+\mathbb{E}\int_{0}^{T}\Phi^{\theta,\mu}_s \left|g(s,0,0)\right|^2 ds+	\mathbb{E}\sup_{0\leq t\leq T}\left|\Phi^{\theta,\mu}_t(S_t)^{+}\right|^2\right)
					\end{split}
				\end{equation}	
				Then, from (\ref{nmagh}), we deduce
				\begin{equation*}
					\begin{split}
						\mathbb{E}\left|K_T\right|^2
						&\leq \mathfrak{c}_{\theta,\mu,\alpha,\zeta} \left(\mathbb{E}\Phi^{\theta,\mu}_T\left|\xi\right|^2+\mathbb{E}\int_{0}^{T}\Phi^{\theta,\mu}_s\left|\frac{\varphi_s}{a_s}\right|^2 ds+\int_{0}^{T} \Phi^{\theta,\mu}_s \left|\psi_s\right|^2 d \kappa_s\right.\\
						&\qquad\left.+\mathbb{E}\int_{0}^{T}\Phi^{\theta,\mu}_s \left|g(s,0,0)\right|^2 ds+	\mathbb{E}\sup_{0\leq t\leq T}\left|\Phi^{\theta,\mu}_t(S_t)^{+}\right|^2\right)
					\end{split}
				\end{equation*}

				Now, applying the B-D–G inequality similar to those used in Corollary \ref{Used Coro}, we have
				\begin{equation*}
					\begin{split}
						2\mathbb{E}\sup_{0 \leq t \leq T}\left|\int_{t}^{T}\Phi^{\theta,\mu}_sY_s g(s,Y_s,Z_s) \overleftarrow{dB}_s\right|
						&\leq2 \mathfrak{c} \mathbb{E}\sqrt{\int_{0}^{T} \Phi^{2\theta,2\mu}_sY_s^2 \left|g(s,Y_s,Z_s)\right|^2 ds}\\
						& \leq \frac{1}{4}\mathbb{E}\sup_{0 \leq t \leq T}\Phi^{\theta,\mu}_t\left|\mathcal{Y}_t\right|^2+4\mathfrak{c}^2\mathbb{E}\int_{0}^{T}\Phi^{\theta,\mu}_s\left|g(s,Y_s,Z_s)\right|^2 ds\\
						& \leq \frac{1}{4}\mathbb{E}\sup_{0 \leq t \leq T}\Phi^{\theta,\mu}_t\left|\mathcal{Y}_t\right|^2+8\mathfrak{c}^2\mathbb{E}\int_{0}^{T}\Phi^{\theta,\mu}_s\left|Y_s\right|^2 dV_s+\mathfrak{c}^2\mathbb{E}\int_{0}^{T}\Phi^{\theta,\mu}_s\left\|Z_s \gamma_s \right\|^2_{\ell^2}ds
					\end{split}
				\end{equation*}
				and
				\begin{equation*}
					\begin{split}
						2\mathbb{E}\sup_{0 \leq t \leq T}\left|\sum_{k=1}^{d}\int_{t}^{T}\Phi^{\theta,\mu}_sY_{s-} Z^{(k)}_s d H^{(k)}_s\right|
						&\leq 2\mathfrak{c} \mathbb{E}\sqrt{\sum_{k=1}^{d}\int_{0}^{T}\Phi^{2\theta,2\mu}_s Y_s^2  \left|Z^{(k)}_s\right|^2 \left|\gamma^{(k)}_s\right|^2 ds}\\ 
						&\leq  \frac{1}{4}\mathbb{E}\sup_{0 \leq t \leq T}\Phi^{\theta,\mu}_t\left|Y_t\right|^2+4\mathfrak{c}^2\mathbb{E}\int_{0}^{T}\Phi^{\theta,\mu}_s\left\|Z_s \gamma_s \right\|^2_{\ell^2}ds.
					\end{split}
				\end{equation*}
				Finally, the B-D-G inequality along with Moreover, using again the B-D-G inequality along with the Kunita-Watanabe inequality (see, e.g., \cite[Ch II. Theorem 25]{protter2005stochastic}), we get
				\begin{equation*}
					\begin{split}
						&\mathbb{E}\sup_{0\leq t\leq T}\left|\sum_{k,k'=1}^d\int_t^T \Phi^{\theta,\mu}_s Z^{(k)}_s Z^{(k')}_s\left(d[H^{(k)},H^{(k')}]_s-d\langle H^{(k)},H^{(k')}\rangle_s\right)\right|\\
						&\leq \mathfrak{c}\mathbb{E}\sqrt{\sum_{k,k'=1}^d \sum_{0 < s \leq T} \Phi^{\theta,\mu}_s \left|Z^{(k)}_s\right|^2 \left|Z^{(k')}_s\right|^2 \left(\Delta H^{(k)}_s\right)^2 \left(\Delta H^{(k')}_s\right)^2}.\\
						&\leq \mathfrak{c}\mathbb{E}\sum_{k,k'=1}^d \sum_{0 < s \leq T} \Phi^{\theta,\mu}_s\left|Z^{(k)}_s\right| \left|Z^{(k')}_s\right| \left|\Delta H^{(k)}_s\right| \left|\Delta H^{(k')}_s\right|.\\
						&\leq \mathfrak{c}\mathbb{E}\sum_{k,k'=1}^d\int_{0}^{T} \Phi^{\theta,\mu}_s\left|Z^{(k)}_s\right| \left|Z^{(k')}_s\right| d\left\| \left[H^{(k)}, H^{(k')}\right]\right\|_s.\\
						&\leq \mathfrak{c}\mathbb{E}\sum_{k,k'=1}^d\left(\int_{0}^{T} \Phi^{\theta,\mu}_s \left|Z^{(k)}_s\right|^2 d \left[H^{(k)}, H^{(k)}\right]_s\right)^{\frac{1}{2}} \left(\int_{0}^{T} \Phi^{\theta,\mu}_s\left|Z^{(k')}_s\right|^2 d \left[H^{(k')}, H^{(k')}\right]_s\right)^{\frac{1}{2}}.\\
						&\leq \mathfrak{c}\mathbb{E}\sum_{k=1}^d\int_{0}^{T} \Phi^{\theta,\mu}_s \left|Z^{(k)}_s\right|^2 d \left[H^{(k)}, H^{(k)}\right]_s+\mathfrak{c}\mathbb{E}\sum_{k'=1}^d \int_{0}^{T} \Phi^{\theta,\mu}_s \left|Z^{(k')}_s\right|^2 d \left[H^{(k')}, H^{(k')}\right]_s.\\
						&\leq \mathfrak{c}\mathbb{E}\int_0^T \Phi^{\theta,\mu}_s\|\gamma_sZ_s\|_{\ell^2}^2ds,
					\end{split}
				\end{equation*}
				where $\left\| \left[H^{(k)}, H^{(k')}\right]\right\|_s$ is the total variation of $ \left[H^{(k)}, H^{(k')}\right]_s$.

				Returning to (\ref{Baic coro app}), taking supremum and using the above three inequalities, along with (\ref{Baic coro appv2 2}), we derive the following
				\begin{equation*}
					\begin{split}
						\mathbb{E}\sup_{0 \leq t \leq T}\Phi^{\theta,\mu}_t\left|Y_t\right|^2
						&\leq \mathfrak{c}_{\theta,\mu,\alpha,\zeta} \left(\mathbb{E}\Phi^{\theta,\mu}_T\left|\xi\right|^2+\mathbb{E}\int_{0}^{T}\Phi^{\theta,\mu}_s\left|\frac{\varphi_s}{a_s}\right|^2 ds+\int_{0}^{T} \Phi^{\theta,\mu}_s \left|\psi_s\right|^2 d \kappa_s\right.\\
						&\qquad\left.+\mathbb{E}\int_{0}^{T}\Phi^{\theta,\mu}_s \left|g(s,0,0)\right|^2 ds+	\mathbb{E}\sup_{0\leq t\leq T}\left|\Phi^{\theta,\mu}_t(S_t)^{+}\right|^2\right)
					\end{split}
				\end{equation*}	
				Completing the proof.
			\end{proof}
			\begin{proposition}\label{Propo 2}
				Assume that \textbf{(H-M)} hold. Then, the RGBDSDE-NL (\ref{basic equation}) has at most one solution $(Y, Z, K)$.	
			\end{proposition}
			\begin{proof}
				Let $(Y, Z, K)$ and $(Y', Z', K')$ be two solutions of the RGBDSDE-NL (\ref{basic equation}) associated with the data $(\xi, f, h, \kappa, g, S)$. Denote $\bar{\mathfrak{S}}=\mathfrak{S}-\mathfrak{S}'$ for $\mathfrak{S} \in \left\{Y,Z,K\right\}$.\\
				Using Corollary \ref{Used Coro} and the minimality condition $\bar{Y}_s d\bar{K}_s \leq 0$, we obtain the following result for any $\theta, \mu > 0$:
				\begin{equation*}\label{}
					\begin{split}
						&\Phi^{\theta,\mu}_t\left|\bar{Y}_t\right|^2+\theta \int_{t}^{T}\Phi^{\theta,\mu}_s\left|\bar{Y}_s\right|^2 dV_s+\mu \int_{t}^{T}\Phi^{\theta,\mu}_s \left|\bar{Y}_s\right|^2 d \kappa_s+\int_{t}^{T}\Phi^{\theta,\mu}_s\left\|\bar{Z}_s \gamma_s \right\|^2_{\ell^2}ds\\
						& \leq 2\int_{t}^{T} \Phi^{\theta,\mu}_s Y_s f(s,Y_s,Z_s)ds+2\int_{t}^{T} \Phi^{\theta,\mu}_s Y_s \left(h(s,Y_s)-h(s,Y'_s)\right) d \kappa_s\\
						&\quad+2\int_{t}^{T} \Phi^{\theta,\mu}_s \bar{Y}_s \left(g(s,Y_s,Z_s)-g(s,Y'_s,Z'_s)\right)\overleftarrow{dB}_s+\int_{t}^{T}\Phi^{\theta,\mu}_s\left|g(s,Y_s,Z_s)-g\left(s,Y'_s,Z'_s\right)\right|^2 ds\\
						&\quad-2\sum_{k=1}^d \int_{t}^{T}\Phi^{\theta,\mu}_s \bar{Y}_{s-} \bar{Z}^{(k)}_sdH^{(k)}_s-\sum_{k,k'=1}^{d} \int_{0}^{t}\Phi^{\theta,\mu}_s \bar{Z}_{s}^{(k)}\bar{Z}_{s}^{(k')} \, \left(d\big[H^{(k)}, H^{(k')}\big]_{s}-d\big\langle H^{(k)}, H^{(k')}\big\rangle _{s}\right).
					\end{split}
				\end{equation*}	
				Then, by using \textbf{(H-M)}-(ii)-(iii)-(v)-(vii) and taking the expectation, we obtain:
				\begin{equation*}
					\begin{split}
						&\mathbb{E}\Phi^{\theta,\mu}_t\left|\bar{Y}_t\right|^2+\theta\mathbb{E} \int_{t}^{T}\Phi^{\theta,\mu}_s\left|\bar{Y}_s\right|^2 dV_s+\mu \mathbb{E}\int_{t}^{T}\Phi^{\theta,\mu}_s \left|\bar{Y}_s\right|^2 d \kappa_s+\mathbb{E}\int_{t}^{T}\Phi^{\theta,\mu}_s\left\|\bar{Z}_s \gamma_s \right\|^2_{\ell^2}ds\\
						&\leq 2\mathbb{E}\int_{t}^{T} \Phi^{\theta,\mu}_s \bar{Y}_s \left(f(s,Y_s,Z_s)-f(s,Y'_s,Z'_s)\right)ds+2\mathbb{E}\int_{t}^{T} \Phi^{\theta,\mu}_s \bar{Y}_s\left( h(s,Y_s)- h(s,Y'_s)\right) d \kappa_s\\
						&\quad+\mathbb{E}\int_{t}^{T}\Phi^{\theta,\mu}_s\left|g(s,Y_s,Z_s)-g(s,Y'_s,Z'_s)\right|^2 ds\\
						& \leq 2\mathbb{E} \int_{t}^{T}\Phi^{\theta,\mu}_s\left|\bar{Y}_s\right|^2 dV_s+\left(\alpha+\frac{1}{2}\right)\mathbb{E}\int_{t}^{T}\Phi^{\theta,\mu}_s\left\|\bar{Z}_s \gamma_s \right\|^2_{\ell^2}ds
					\end{split}
				\end{equation*}	
				Then, for $\theta > 2$ and $\mu > 0$, and using the fact that $\alpha + \frac{1}{2} < 1$, we have:
				\begin{equation}\label{Y=Y'}
					\mathbb{E}\Phi^{\theta,\mu}_t\left|\bar{Y}_t\right|^2+\mathbb{E} \int_{t}^{T}\Phi^{\theta,\mu}_s\left|\bar{Y}_s\right|^2 dV_s+ \mathbb{E}\int_{t}^{T}\Phi^{\theta,\mu}_s \left|\bar{Y}_s\right|^2 d \kappa_s+\mathbb{E}\int_{t}^{T}\Phi^{\theta,\mu}_s\left\|\bar{Z}_s \gamma_s \right\|^2_{\ell^2}ds \leq 0.
				\end{equation}
				Therefore, we have $Z = Z'$. On the other hand, by applying the B-D-G inequality and performing similar calculations as those in Proposition \ref{Propo 1}, together with (\ref{Y=Y'}), we derive that $\left\|\bar{Y}\right\|^2_{\mathcal{S}^2_{\theta,\mu}} = 0$ for any $(\theta, \mu) \in (2, +\infty) \times (0, +\infty)$. Thus, $Y = Y'$. Finally, from (\ref{basic equation})-(i), we have
				$$
				K_t = Y_0 - Y_t - \int_{0}^{t} f(s,Y_s,Z_s) \, ds - \int_{0}^{t} h(s,Y_s) \, d\kappa_s - \int_{0}^{t} g(s,Y_s,Z_s) \overleftarrow{dB}_s + \sum_{k=1}^{d} \int_{0}^{t} Z^{(k)}_s \, dH^{(k)}_s.
				$$
				Hence, $K = K'$, completing the proof.
			\end{proof}
			
			\subsection{Existence}
			The proof of existence will be divided into two parts:
			\begin{itemize}
				\item In the first part, we deal with a special case where the generator $f$ does not depend on the variable $z$ and $g$ does not depend on the variables $(y, z)$, using Yosida approximation of monotone functions.
				
				\item In the second part, which addresses the general case, we resolve the problem using the Picard iteration method.
			\end{itemize}
			
			\subsubsection{Case where generators $f$ do not depend on $z$ and $g$ does not depend on $(y,z)$}
			We consider the following version of the RGBDSDE-NL (\ref{basic equation}):
			\begin{equation}
				\left\{
				\begin{split}
					\text{(i)}&~Y_t=\xi+\int_{t}^{T}\mathfrak{f}(s,Y_s)ds+\int_{t}^{T}h(s,Y_s)d\kappa_s+\int_{t}^{T}\mathfrak{g}(s)\overleftarrow{dB}_s+\left(K_T-K_t\right)-\sum_{k=1}^{d}\int_{t}^{T}Z^{(k)}_s dH^{(k)}_s,\\
					\text{(ii)}&~Y_t \geq S_t, ~t \in [0,T] ~\mbox{ and }~\int_{0}^{T} \left(Y_{s}-S_{s}\right)dK_s=0.
				\end{split}
				\right.
				\label{basic equation 1}
			\end{equation}
			\begin{theorem}\label{thm1}
				Suppose that \textbf{(H-M)} holds. Then the RGBDSDE-NL (\ref{basic equation 1}) has a unique solution $(Y, Z, K) \in \mathcal{B}^2_{\theta,\mu}$ for any $\theta, \mu > 0$.
			\end{theorem}
			
			Before proving Theorem \ref{thm1}, we record the following remark, which is identical to \cite[Remark 2.3]{elhachemyjiea}:
			\begin{remark}\label{Modified data}
				Let $(Y,Z,K)$ be a solution of the RGBDSDE-NL (\ref{basic equation}) associated with $(\xi, f, h, \kappa, g, S)$. For any $\mu > 0$, we define
				\begin{equation*}
					\begin{split}
						\hat{Y}_t:=e^{\int_{0}^{t}\lambda_s ds+\mu \kappa_t} Y_t,\quad
						\hat{Z}_t:=e^{\int_{0}^{t}\lambda_s ds+\mu \kappa_t} Z_t, \quad
						d \hat{K}_t:=e^{\int_{0}^{t}\lambda_s ds+\mu \kappa_t} dK_t.
					\end{split}
				\end{equation*}
				Using Corollary \ref{Used Coro} with $\Psi(x,y,z) = e^{x + \mu y} z$, $\mathsf{A}_t = \left(\int_{0}^{t} \lambda_s \, ds, \kappa_t\right)$, and $\mathcal{Y} = Y$, we have
				\begin{equation*}
					\begin{split}
						\hat{Y}_t=&\hat{\xi}-\int_{t}^{T} \hat{Y}_s\left(\lambda_sds+\mu d\kappa_s\right)+\int_{t}^{T}\hat{f}\left(s,\hat{Y}_s,\hat{Z}_s\right) ds+\int_{t}^{T}\hat{h}\left(s,\hat{Y}_s\right) d\kappa_s\\
						&+\int_{t}^{T}\hat{g}\left(s,\hat{Y}_s,\hat{Z}_s\right) \overleftarrow{dB}_s+\int_{t}^{T}d\hat{K}_s-\sum_{k=1}^{d} \int_{0}^{t} \hat{Z}^{(k)}_s \, dH^{(k)}_s,
					\end{split}
				\end{equation*}
				$\hat{Y}_t \geq \hat{S}_t$ for any $t \in [0,T]$ and $\int_{0}^{T}\left(\hat{Y}_s - \hat{S}_s\right)d\hat{K}_s=0$, with 
				\begin{equation*}
					\left\{
					\begin{split}
						\hat{\xi}&=e^{\int_{0}^{T}\lambda_s ds+\mu \kappa_T}\xi,\\
						\hat{f}(t,y,z)&=e^{\int_{0}^{t}\lambda_s ds+\mu \kappa_t} f\left(t,e^{-\int_{0}^{t}\lambda_u du-\mu \kappa_t}y,e^{-\int_{0}^{t}\lambda_u du-\mu \kappa_t}z\right)-\lambda_t y,\\
						\hat{h}(t,y)&=e^{\int_{0}^{t}\lambda_s ds+\mu \kappa_t} h\left(s,e^{-\int_{0}^{t}\lambda_u du-\mu \kappa_t}y\right)-\mu y,\\
						\hat{g}(t,y,z)&=e^{\int_{0}^{t}\lambda_s ds+\mu \kappa_t} g\left(t,e^{-\int_{0}^{t}\lambda_u du-\mu \kappa_t}y,e^{-\int_{0}^{s}\lambda_u du-\mu \kappa_t}z\right),\\
						\hat{S}_t&=e^{\int_{0}^{t}\lambda_s ds+\mu \kappa_t} S_t.
					\end{split}
					\right.
				\end{equation*}
				Therefore, whenever $(Y,Z,K)$ is a solution of the RGBDSDE-NL (\ref{basic equation}) associated with $(\xi, f, h, \kappa, g, S)$, the process $(\hat{Y},\hat{Z},\hat{K})$ satisfies an analogous RGBDSDE-NL associated with $(\hat{\xi}, \hat{f}, \hat{h}, \kappa, \hat{g}, \hat{S})$. Hence, if $f$ satisfies \textbf{(H-M)}-(ii), we can assume that $\lambda_t = 0$ for any $t \in [0,T]$. Additionally, by choosing $\mu > \varrho_t$ for any $t \in [0,T]$, we can observe that $\varrho_s < 0$ as used in \textbf{(H-M)}-(iii) is not a severe restriction. However, the driver $\hat{g}$ retains the stochastic Lipschitz condition described in \textbf{(H-M)}-(vii).
			\end{remark}
			\begin{proof}[Proof of Theorem \ref{thm1}] 
				\proof[Uniqueness] The uniqueness of the solution is established in Proposition \ref{Propo 2}.
				
				\proof[Existence] For the existence, the proof is divided into two main parts.
				
				\paragraph*{Part 1:} In this part, we suppose that there exists a constant $M > 0$ such that 
				\begin{equation}	\label{boudness condition}
					\sup_{0 \leq t \leq T}\left| \varphi_t\right|^2+\sup_{0 \leq t \leq T}\left| \psi_t\right|^2 \leq M.
				\end{equation}
				
				Since $f$ is independent of $z$ and $g$ is independent of $(y,z)$, we set $\mathfrak{f}(t,y):=f(t,y,z)$ and $\mathfrak{g}(t):=g(t,y,z)$ for all $(t,y,z)\in[0,T]\times\mathbb{R}\times\ell^2$. In view of Remark \ref{Modified data}, assumptions \textbf{(H-M)} (ii) and (iii) can be rewritten as follows:
				\begin{description}
					\item[(H-M)] 
					\begin{itemize}
						\item[(ii')] For all $y$, $y' \in \mathbb{R}$,  $d\mathbb{P} \otimes dt$-a.e.,
						$$\left(y-y^{\prime}\right)\left(\mathfrak{f}\left(t,y\right)-\mathfrak{f}\left(t,y^{\prime}\right)\right)\leq 0.$$
						
						\item[(iii')] For all $y$, $y^{\prime} \in \mathbb{R}$, $d\mathbb{P} \otimes d\kappa_t$-a.e.,
						$$\left(y-y^{\prime}\right)\left(h\left(t,y\right)-h\left(t,y^{\prime}\right)\right)\leq 0 .$$
					\end{itemize}
				\end{description}
				\subparagraph*{Step 1: Yosida approximation of the RGBDSDE-NL (\ref{basic equation}).}
				\emph{}\\
				From \textbf{(H-M)}-(ii')-(iii') and \cite[Annex B, p. 524]{pardoux2014stochastic}, it follows that for every, $\left(\omega,t,y\right) \in \Omega \times [0,T] \times \mathbb{R}$ and $\delta >0$, there exists a unique $J_{\delta}^{\mathfrak{f}}=J_{\delta}^{\mathfrak{f}}(\omega,t,y)$, $J_{\delta}^{h}=J_{\delta}^{h}(\omega,t,y) \in \mathbb{R}$ such that
				$$
				J_{\delta}^{\mathfrak{f}}-\delta \mathfrak{f}(\omega,t,J^\mathfrak{f}_{\delta})=y,\quad J_{\delta}^{h}-\delta h(\omega,t,J^h_{\delta})=y.
				$$
				The Yosida approximation of $\mathfrak{f}$ and $h$ is defined respectively by $\mathfrak{f}_{\delta}=\mathfrak{f}_{\delta}\left(\omega,t,y\right)$ and $h_{\delta}=h_{\delta}\left(\omega,t,y\right) \in \mathbb{R}$ such that
				\begin{equation}
					\begin{split}
						&\mathfrak{f}_{\delta}(\omega,t,y):=\dfrac{1}{\delta}\left(J_{\delta}^{\mathfrak{f}}(\omega,t,y)-y\right)=\mathfrak{f}\left(t,y+\delta \mathfrak{f}_{\delta}(\omega,t,y) \right),\\
						&h_{\delta}(\omega,t,y):=\dfrac{1}{\delta}\left(J_{\delta}^{h}(\omega,t,y)-y\right)=h\left(t,y+\delta h_{\delta}(\omega,t,y) \right).
					\end{split}
					\label{Yusida system}
				\end{equation}
				Note that $(\mathfrak{f}_{\delta},h_{\delta})$ is the unique triplet satisfying the system (\ref{Yusida system}).\\
				From \cite[Annex B, Proposition 6.7]{pardoux2014stochastic}, recall that for every \( y \in \mathbb{R} \), the processes \( \mathfrak{f}_{\delta}(\cdot, \cdot, y) \) and \( h_{\delta}(\cdot, \cdot, y): \Omega \times [0,T] \rightarrow \mathbb{R} \) are jointly measurable and \( \mathcal{F}_t \)-measurable for each \( t \in [0,T] \). Furthermore, it holds that:
				\begin{itemize}
					\item[\textbf{(Y)}]  $\forall \delta,\mathsf{r} >0$, $\forall t \in [0,T]$, $\forall y,y^{\prime} \in \mathbb{R}$, $\mathbb{P}$-a.s.
					\begin{itemize}	
						\item[(i)]  $\left(y-y^{\prime}\right)\left(\mathfrak{f}_{\delta}(t,y)-\mathfrak{f}_{\delta}(t,y^{\prime})\right) \leq 0 $.
						\item[(ii)]  $\left(y-y^{\prime}\right)\left(h_{\delta}(t,y)-h_{\delta}(t,y^{\prime})\right) \leq 0$.
						\item[(iii)]  $\left| \mathfrak{f}_{\delta}(t,y)-\mathfrak{f}_{\delta}\left( t,y^{\prime}\right)  \right|+\left| h_{\delta}(t,y)-h_{\delta}(t,y^{\prime} ) \right|  \leq \dfrac{2}{\delta}\left| y-y^{\prime}\right|$.\\		
						\item[(iv)]  $\left| \mathfrak{f}_{\delta}(t,y) \right|\leq \left| \mathfrak{f}(t,y)\right|$ and $\lim\limits_{\delta \rightarrow 0}\mathfrak{f}_{\delta}(t,y)=\mathfrak{f}(t,y)$,\\
						$\left| h_{\delta}(t,y) \right|\leq \left| h(t,y)\right|$ and
						$\lim\limits_{\delta \rightarrow 0}h_{\delta}(t,y)=h(t,y)$.
						\item[(v)] $\left(y-y^{\prime}\right)\left(\mathfrak{f}_{\delta}(t,y)-\mathfrak{f}_{\mathsf{r}}(t,y^{\prime})\right) \leq \left(\delta+\mathsf{r}\right)\mathfrak{f}_{\delta}(t,y)\mathfrak{f}_{\mathsf{r}}(t,y^{\prime})$,\\
						$\left(y-y^{\prime}\right)\left(h_{\delta}(t,y)-h_{\mathsf{r}}(t,y^{\prime})\right) \leq \left(\delta+\mathsf{r} \right)h_{\delta}(t,y)h_{\mathsf{r}}(t,y^{\prime})$.
					\end{itemize}
				\end{itemize}
				
				Let $0 < \delta \leq 1$. From Theorem \ref{Thm A3}, there exists a unique solution $\left(Y^{\delta}, Z^{\delta}, K^{\delta}\right) \in \left(\mathcal{S}^2_\mu \cap \mathcal{H}^{2,\mathcal{Q}}_\mu\right) \times \mathcal{H}^{2,\ell^2}_\mu \times \mathcal{A}^2$ of the RGBDSDE-NL (\ref{basic equation 1}) associated with $\left({\xi}, \mathfrak{f}_\delta, h_\delta, \kappa, \mathfrak{g}, S\right)$. In other words, the approximating equation:
				\begin{equation}
					\left\{
					\begin{split}
						\text{(i)}&~Y^\delta_t={\xi}+\int_{t}^{T}{\mathfrak{f}}_\delta(s,Y^\delta_s)ds+\int_{t}^{T}{h}_\delta(s,Y^\delta_s)d\kappa_s+\int_{t}^{T}\mathfrak{g}(s)\overleftarrow{dB}_s+\left(K^\delta_T-K^\delta_t\right)-\sum_{k=1}^{d}\int_{t}^{T}Z^{\delta,(k)}_s dH^{(k)}_s,\\
						\text{(ii)}&~Y^\delta_t \geq S_t, ~t \in [0,T] ~\mbox{ and }~\int_{0}^{T} \left(Y^\delta_{s}-S_{s}\right)dK^\delta_s=0,
					\end{split}
					\right.
					\label{Approximating}
				\end{equation}
				has a unique solution $\left(Y^{\delta}, Z^{\delta}, K^{\delta}\right) \in \left(\mathcal{S}^2_\mu \cap \mathcal{H}^{2,\mathcal{Q}}_\mu\right) \times \mathcal{H}^{2,\ell^2}_\mu \times \mathcal{A}^2$.
				\subparagraph*{Step 2: Uniform estimation of $\left\{\left(Y^\delta,Z^\delta,K^\delta\right)\right\}_{0<\delta \leq 1}$.}\emph{}\\
				We aim to show the following lemma:
				\begin{lemma}\label{UI Lemma}
					For any $0 < \delta \leq 1$ and for every $\theta, \mu > 0$, the process $\left(Y^\delta, Z^\delta, K^\delta\right)$ satisfies
					\begin{equation*}
						\begin{split}
							&\mathbb{E}\sup_{0 \leq t \leq T}\Phi^{\theta,\mu}_t\left|Y^{\delta}_t\right|^2+\mathbb{E}\int_{0}^{T}\Phi^{\theta,\mu}_s\left|Y^{\delta}_s\right|^2d \mathcal{Q}_s+\mathbb{E}\int_{0}^{T}\Phi^{\theta,\mu}_s\left\|Z^{\delta}_s \gamma_s \right\|^2_{\ell^2}ds+\mathbb{E}\left|K^{\delta}_T\right|^2\\
							& \leq \mathfrak{c}_{\theta,\mu,\zeta}\left(\mathbb{E}\Phi^{\theta,\mu}_T\left|\xi\right|^2+\mathbb{E}\int_{0}^{T}\Phi^{\theta,\mu}_s\left|\frac{\varphi_s}{a_s}\right|^2 ds+\mathbb{E}\int_{0}^{T}\Phi^{\theta,\mu}_s\left|\psi_s\right|^2 d\kappa_s \right.\\
							&\qquad\left.+\mathbb{E}\int_{0}^{T}\Phi^{\theta,\mu}_s\left|\mathfrak{g}(s)\right|^2 ds+\mathbb{E}\sup_{0 \leq t \leq T}\left|\Phi^{\theta,\mu}_t\left(S_t\right)^{+}\right|^2\right)
						\end{split}
					\end{equation*}
				\end{lemma} 
				
				\begin{proof}
					Using Corollary \ref{Used Coro}, we derive that 
					\begin{equation}\label{Baic coro approx}
						\begin{split}
							&\Phi^{\theta,\mu}_t\left|Y^\delta_t\right|^2+\theta \int_{t}^{T}\Phi^{\theta,\mu}_s\left|Y_s\right|^2 dV_s+\mu \int_{t}^{T}\Phi^{\theta,\mu}_s \left|Y^\delta_s\right|^2 d \kappa_s+\int_{t}^{T}\Phi^{\theta,\mu}_s\left\|Z^\delta_s \gamma_s \right\|^2_{\ell^2}ds\\
							&=\Phi^{\theta,\mu}_T\left|{\xi}\right|^2+2\int_{t}^{T} \Phi^{\theta,\mu}_s Y^\delta_s {\mathfrak{f}}_\delta(s,Y^\delta_s)ds+2\int_{t}^{T} \Phi^{\theta,\mu}_s Y^\delta_s {h}_\delta(s,Y^\delta_s) d \kappa_s\\
							&\quad+2\int_{t}^{T} \Phi^{\theta,\mu}_s Y^\delta_s \mathfrak{g}(s)\overleftarrow{dB}_s+\int_{t}^{T}\Phi^{\theta,\mu}_s\left|\mathfrak{g}(s)\right|^2 ds-2\sum_{k=1}^d \int_{t}^{T}\Phi^{\theta,\mu}_s Y^\delta_{s-} Z^{\delta,(k)}_sdH^{(k)}_s\\
							&\quad-\sum_{k,k'=1}^{d} \int_{0}^{t}\Phi^{\theta,\mu}_s Z_{s}^{\delta,(k)}Z_{s}^{\delta,(k')} \, \left(d\big[H^{(k)}, H^{(k')}\big]_{s}-d\big\langle H^{(k)}, H^{(k')}\big\rangle _{s}\right)+2\int_{t}^{T} \Phi^{\theta,\mu}_s Y^\delta_s dK^\delta_s.
						\end{split}
					\end{equation}	
					Using \textbf{(Y)}-(i)-(ii)-(iv) and the basic inequality $2ab \leq \varepsilon a^2 + \frac{1}{\varepsilon}b^2$ for any $\varepsilon > 0$, we derive
					\begin{equation*}
						\begin{split}
							2Y^\delta_s {\mathfrak{f}}_\delta(s,Y^\delta_s)ds \leq 2Y^\delta_s {\mathfrak{f}}_\delta(s,0)ds
							& \leq 2 \left|Y^\delta_s \right| \left|{\mathfrak{f}}(s,0)\right|ds\leq \varepsilon' \left|Y^\delta_s\right|^2 dV_s+\frac{1}{\varepsilon'}\left|\frac{\varphi_s}{a_s}\right|^2 ds,
						\end{split}
					\end{equation*}
					and
					\begin{equation*}
						\begin{split}
							2Y^\delta_s h_\delta(s,Y^\delta_s)d\kappa_s \leq 2Y^\delta_s h_\delta(s,0)d\kappa_s
							& \leq 2 \left|Y^\delta_s \right| \left|h(s,0)\right|d\kappa_s\leq \varepsilon" \left|Y^\delta_s\right|^2 d\kappa_s+\frac{1}{\varepsilon"}\left|\psi_s\right|^2 d\kappa_s,
						\end{split}
					\end{equation*}
					for any $\varepsilon', \varepsilon>0$.\\
					Then, plugging these inequalities into (\ref{Baic coro approx}), along with the Skorokhod condition and taking the expectation, we get
					\begin{equation}\label{Baic coro approx-}
						\begin{split}
							&\mathbb{E}\Phi^{\theta,\mu}_t\left|Y^\delta_t\right|^2+\left(\theta-\varepsilon'\right)\mathbb{E} \int_{t}^{T}\Phi^{\theta,\mu}_s\left|Y_s\right|^2 dV_s+(\mu-\varepsilon")\mathbb{E} \int_{t}^{T}\Phi^{\theta,\mu}_s \left|Y^\delta_s\right|^2 d \kappa_s+\mathbb{E}\int_{t}^{T}\Phi^{\theta,\mu}_s\left\|Z^\delta_s \gamma_s \right\|^2_{\ell^2}ds\\
							&\leq \mathbb{E}\Phi^{\theta,\mu}_T\left|{\xi}\right|^2+\mathbb{E}\int_{t}^{T}\Phi^{\theta,\mu}_s\left|\mathfrak{g}(s)\right|^2 ds+2\int_{t}^{T} \Phi^{\theta,\mu}_s S_s dK^\delta_s.
						\end{split}
					\end{equation}	
					Choosing $\varepsilon' < \theta$ and $\varepsilon'' < \mu$, we derive, for any $\varepsilon''' > 0$ 
					\begin{equation}\label{Baic coro appv2-}
						\begin{split}
							&\mathbb{E}\Phi^{\theta,\mu}_t\left|Y^\delta_t\right|^2+\mathbb{E} \int_{t}^{T}\Phi^{\theta,\mu}_s\left|Y^\delta_s\right|^2 d \mathcal{Q}_s+\mathbb{E}\int_{t}^{T}\Phi^{\theta,\mu}_s\left\|Z^\delta_s \gamma_s \right\|^2_{\ell^2}ds\\
							&\leq \mathfrak{c}_{\theta,\mu} \left(\mathbb{E}\Phi^{\theta,\mu}_T\left|\xi\right|^2+\mathbb{E}\int_{t}^{T}\Phi^{\theta,\mu}_s\left|\frac{\varphi_s}{a_s}\right|^2 ds+\int_{t}^{T} \Phi^{\theta,\mu}_s \left|\psi_s\right|^2 d \kappa_s+\mathbb{E}\int_{t}^{T}\Phi^{\theta,\mu}_s \left|\mathfrak{g}(s)\right|^2 ds\right)\\
							&\qquad+\varepsilon^{\prime\prime\prime}\mathfrak{c}_{\theta,\mu}^2	\mathbb{E}\sup_{t\leq u\leq T}\left|\Phi^{\theta,\mu}_u(S_u)^{+}\right|^2+\frac{1}{\varepsilon^{\prime\prime\prime}}\mathbb{E}\left|K^\delta_T-K^\delta_t\right|^2.
						\end{split}
					\end{equation}	
					Returning to (\ref{basic equation})-(i), by squaring, we have
					\begin{equation}\label{Sbagha-}
						\begin{split}
							&\mathbb{E}\left|K^\delta_T-K^\delta_t\right|^2\\
							& \leq 6\left(\mathbb{E}\left|Y^\delta_t\right|^2+\mathbb{E}\left|\xi\right|^2+\mathbb{E}\left(\int_{t}^{T}\mathfrak{f}_\delta(s,Y^\delta_s)ds\right)^2+\mathbb{E}\left(\int_{t}^{T}h_\delta(s,Y^\delta_s)d\kappa_s\right)^2\right.\\
							&\left. +\mathbb{E}\left(\int_{t}^{t}\mathfrak{g}(s)\overleftarrow{dB}_s\right)^2+\mathbb{E}\left(\sum_{k=1}^{d}\int_{t}^{T}Z^{\delta,(k)}_s dH^{(k)}_s\right)^2 \right)\\
							&\leq 6\left(\mathbb{E}\Phi^{\theta,\mu}_t\left|Y^\delta_t\right|^2+\mathbb{E}\Phi^{\theta,\mu}_T\left|\xi\right|^2+\mathbb{E}\left(\int_{t}^{T}\mathfrak{f}_\delta(s,Y^\delta_s)ds\right)^2+\mathbb{E}\left(\int_{t}^{T}h_\delta(s,Y^\delta_s)d\kappa_s\right)^2\right.\\
							&\qquad\left. +\mathbb{E}\int_{t}^{T}\Phi^{\theta,\mu}_s \left|\mathfrak{g}(s)\right|^2ds+\mathbb{E}\int_{t}^{T}\Phi^{\theta,\mu}_s\left\|Z^\delta_s \gamma_s \right\|^2_{\ell^2}ds \right)
						\end{split}
					\end{equation}
					Using \textbf{(Y)}-(iv), \textbf{(H-M)}-(ix), and Hölder's inequality, we have
					\begin{equation}
						\begin{split}
							\mathbb{E}\left(\int_{t}^{T}\mathfrak{f}_\delta(s,Y^\delta_s)ds\right)^2
							& \leq \mathbb{E}\left(\int_{t}^{T}\Phi^{-\theta,\mu}_s dV_s\int_{t}^{T}\Phi^{\theta,\mu}_s\left|\frac{\mathfrak{f}_\delta(s,Y^\delta_s)}{a_s}\right|^2ds\right)\\
							& \leq \mathbb{E}\left(\int_{t}^{T}\Phi^{-\theta,\mu}_s dV_s\int_{t}^{T}\Phi^{\theta,\mu}_s\left|\frac{\mathfrak{f}(s,Y^\delta_s)}{a_s}\right|^2ds\right)\\
							& \leq \frac{3}{\theta}\mathbb{E}\int_{t}^{T}\Phi^{\theta,\mu}_s\frac{\varphi^2_s+\phi_s^2 \left|Y^\delta_s\right|^2}{a^2_s}ds\\
							& \leq \frac{3}{\theta}\mathbb{E}\int_{t}^{T}\Phi^{\theta,\mu}_s\left|Y^\delta_s\right|^2 dV_s+\frac{3}{\theta}\mathbb{E}\int_{t}^{T}\Phi^{\theta,\mu}_s\left|\frac{\varphi_s}{a_s}\right|^2 ds,
						\end{split}
					\end{equation}
					and
					\begin{equation}
						\begin{split}
							\mathbb{E}\left(\int_{t}^{T}h_\delta(s,Y^\delta_s)d\kappa_s\right)^2 & \leq \mathbb{E}\left(\int_{t}^{T}\Phi^{\theta,-\mu}_s d\kappa_s\int_{t}^{T}\Phi^{\theta,\mu}_s\left|h_\delta(s,Y^\delta_s)\right|^2d\kappa_s\right)\\
							& \leq \mathbb{E}\left(\int_{t}^{T}\Phi^{\theta,-\mu}_s d\kappa_s\int_{t}^{T}\Phi^{\theta,\mu}_s\left|h(s,Y^\delta_s)\right|^2d\kappa_s\right)\\
							& \leq \frac{2\zeta^2}{\mu}\mathbb{E}\int_{t}^{T}\Phi^{\theta,\mu}_s\left|Y^\delta_s\right|^2 d\kappa_s+ \frac{2}{\mu}\mathbb{E}\int_{t}^{T}\Phi^{\theta,\mu}_s \psi^2_sd\kappa_s.
						\end{split}
					\end{equation}
					Substituting this into (\ref{Baic coro appv2-}), and choosing $\varepsilon^{\prime\prime\prime} > \frac{18}{\theta} + \frac{12 \zeta^2}{\mu} + 12$, we have
					\begin{equation}\label{Baic coro appv2 2-}
						\begin{split}
							&\sup_{0 \leq t \leq T}\mathbb{E}\Phi^{\theta,\mu}_t\left|Y^\delta_t\right|^2+\mathbb{E} \int_{0}^{T}\Phi^{\theta,\mu}_s\left|Y^\delta_s\right|^2 d \mathcal{Q}_s+\mathbb{E}\int_{0}^{T}\Phi^{\theta,\mu}_s\left\|Z^\delta_s \gamma_s \right\|^2_{\ell^2}ds\\
							&\leq \mathfrak{c}_{\theta,\mu,\zeta} \left(\mathbb{E}\Phi^{\theta,\mu}_T\left|\xi\right|^2+\mathbb{E}\int_{0}^{T}\Phi^{\theta,\mu}_s\left|\frac{\varphi_s}{a_s}\right|^2 ds+\int_{0}^{T} \Phi^{\theta,\mu}_s \left|\psi_s\right|^2 d \kappa_s\right.\\
							&\qquad\left.+\mathbb{E}\int_{0}^{T}\Phi^{\theta,\mu}_s \left|\mathfrak{g}(s)\right|^2 ds+	\mathbb{E}\sup_{0\leq t\leq T}\left|\Phi^{\theta,\mu}_t(S_t)^{+}\right|^2\right).
						\end{split}
					\end{equation}	
					Then, from (\ref{Sbagha-}), we conclude
					\begin{equation*}
						\begin{split}
							\mathbb{E}\left|K^\delta_T\right|^2
							&\leq \mathfrak{c}_{\theta,\mu,\zeta} \left(\mathbb{E}\Phi^{\theta,\mu}_T\left|\xi\right|^2+\mathbb{E}\int_{0}^{T}\Phi^{\theta,\mu}_s\left|\frac{\varphi_s}{a_s}\right|^2 ds+\int_{0}^{T} \Phi^{\theta,\mu}_s \left|\psi_s\right|^2 d \kappa_s\right.\\
							&\qquad\left.+\mathbb{E}\int_{0}^{T}\Phi^{\theta,\mu}_s \left|\mathfrak{g}(s)\right|^2 ds+	\mathbb{E}\sup_{0\leq t\leq T}\left|\Phi^{\theta,\mu}_t(S_t)^{+}\right|^2\right).
						\end{split}
					\end{equation*}	
					The remainder of the proof relies on the uniform estimation (\ref{Baic coro appv2 2-}) and the B-D-G inequality. Using a similar computation as that in Proposition \ref{Propo 1}, we can easily derive that, for any $\theta, \mu > 0$,
					\begin{equation*}
						\begin{split}
							\mathbb{E}\sup_{0 \leq t \leq T}\Phi^{\theta,\mu}_t\left|Y^\delta_t\right|^2
							&\leq \mathfrak{c}_{\theta,\mu,\zeta} \left(\mathbb{E}\Phi^{\theta,\mu}_T\left|\xi\right|^2+\mathbb{E}\int_{0}^{T}\Phi^{\theta,\mu}_s\left|\frac{\varphi_s}{a_s}\right|^2 ds+\int_{0}^{T} \Phi^{\theta,\mu}_s \left|\psi_s\right|^2 d \kappa_s\right.\\
							&\qquad\left.+\mathbb{E}\int_{0}^{T}\Phi^{\theta,\mu}_s \left|\mathfrak{g}(s)\right|^2 ds+	\mathbb{E}\sup_{0\leq t\leq T}\left|\Phi^{\theta,\mu}_t(S_t)^{+}\right|^2\right).
						\end{split}
					\end{equation*}	
					Completing the proof of Lemma \ref{UI Lemma}.
				\end{proof}
				
				\subparagraph*{Step 3: The process $\left\{Y^\delta, Z^\delta\right\}_{0 < \delta \leq 1}$ is a Cauchy sequence in $\mathcal{C}^2_{\theta,\mu} \times \mathcal{H}^{2,\ell}_{\theta,\mu}$.}
				\emph{}\\
				Let $0 < \delta, \delta' \leq 1$. Set $\mathcal{R}^{\delta, \delta'} = \mathcal{R}^\delta - \mathcal{R}^{\delta'}$ with $\mathcal{R} \in \{Y, Z, K\}$. Using Corollary \ref{Used Coro} with expectation, we arrive at
				\begin{equation}\label{Baic coro approx--}
					\begin{split}
						&\mathbb{E}\Phi^{\theta,\mu}_t\left|Y^{\delta,\delta'}_t\right|^2+\theta\mathbb{E} \int_{t}^{T}\Phi^{\theta,\mu}_s\left|Y^{\delta,\delta'}_s\right|^2 dV_s+\mu\mathbb{E} \int_{t}^{T}\Phi^{\theta,\mu}_s \left|Y^{\delta,\delta'}_s\right|^2 d \kappa_s+\mathbb{E}\int_{t}^{T}\Phi^{\theta,\mu}_s\left\|Z^\delta_s \gamma_s \right\|^2_{\ell^2}ds\\
						&=2\mathbb{E}\int_{t}^{T} \Phi^{\theta,\mu}_s \left(Y^{\delta}_s-Y^{\delta'}_s\right) \left({\mathfrak{f}}_\delta(s,Y^\delta_s)-{\mathfrak{f}}_{\delta'}(s,Y^{\delta'}_s)\right)ds\\
						&\quad+2\mathbb{E}\int_{t}^{T} \Phi^{\theta,\mu}_s \left(Y^\delta_s-Y^{\delta'}_s\right) \left({h}_\delta(s,Y^\delta_s)-{h}_{\delta'}(s,Y^{\delta'}_s)\right) d \kappa_s+2\mathbb{E}\int_{t}^{T} \Phi^{\theta,\mu}_s\left(Y^\delta_s-Y^{\delta'}_s\right) dK^{\delta,\delta'}_s.
					\end{split}
				\end{equation}	
				Using the Skorokhod condition (\ref{Approximating})-(ii), we obtain
				$$
				\left(Y^\delta_s-Y^{\delta'}_s\right) dK^{\delta,\delta'}_s=-\left(Y^\delta_s-S_s\right) dK^{\delta'}_s+\left(S_s-Y^{\delta'}_s\right) dK^{\delta}_s\leq 0.
				$$
				Next, from \textbf{(Y)}-(iv)-(v), we can deduce
				\begin{equation*}
					\begin{split}
						&2\mathbb{E}\int_{t}^{T} \Phi^{\theta,\mu}_s \left(Y^{\delta}_s-Y^{\delta'}_s\right) \left({\mathfrak{f}}_\delta(s,Y^\delta_s)-{\mathfrak{f}}_{\delta'}(s,Y^{\delta'}_s)\right)ds\\
						&\leq2\left(\delta+\delta'\right)\mathbb{E}\int_{t}^{T} \Phi^{\theta,\mu}_s  \left|\mathfrak{f}_\delta(s,Y^\delta_s)\right|\left|\mathfrak{f}_{\delta'}(s,Y^{\delta'}_s)\right|ds\\
						&\leq2\left(\delta+\delta'\right)\mathbb{E}\int_{t}^{T} \Phi^{\theta,\mu}_s  \left|\mathfrak{f}(s,Y^\delta_s)\right|\left|\mathfrak{f}(s,Y^{\delta'}_s)\right|ds\\
						&\leq2\left(\delta+\delta'\right)\mathbb{E}\int_{t}^{T} \Phi^{\theta,\mu}_s  \left(\varphi_s+\phi_s \left|Y^\delta_s\right|\right)\left(\varphi_s+\phi_s \left|Y^{\delta'}_s\right|\right)ds\\
						&=2\left(\delta+\delta'\right)\mathbb{E}\int_{t}^{T} \Phi^{\theta,\mu}_s  \left(\varphi^2_s+\varphi_s \phi_s \left|Y^\delta_s\right|+\varphi_s \phi_s \left|Y^{\delta'}_s\right|+\phi_s^2\left|Y^{\delta'}_s\right|\left|Y^{\delta'}_s\right|\right)ds\\
						&\leq \left(\delta+\delta'\right)\mathbb{E}\int_{t}^{T} \Phi^{\theta,\mu}_s\left\{2 M^2ds+\left(2M+1\right)\left\{\left|Y^{\delta'}_s\right|^2+\left|Y^{\delta'}_s\right|^2\right\}dV_s\right\}\\
						&\leq \frac{2M^2\left(\delta+\delta'\right)}{\epsilon}\mathbb{E}\int_{t}^{T} \Phi^{\theta,\mu}_sd\mathcal{Q}_s+\left(2M+1\right)\left(\delta+\delta'\right)\mathbb{E}\int_{t}^{T} \Phi^{\theta,\mu}_s\left\{\left|Y^{\delta'}_s\right|^2+\left|Y^{\delta'}_s\right|^2\right\}dV_s\\
						&\leq\left(\delta+\delta'\right) \mathfrak{c}_{M,\epsilon}\mathbb{E}\int_{t}^{T} \left(\Phi^{\theta,\mu}_sd\mathcal{Q}_s+\left\{\left|Y^{\delta'}_s\right|^2+\left|Y^{\delta'}_s\right|^2\right\}dV_s\right).
					\end{split}
				\end{equation*}
				Using Lemma \ref{UI Lemma} and Remark \ref{finitness Remark}, we conclude that there exists a constant $\mathfrak{c}_{M,\epsilon}$ (independent of $\delta$ and $\delta'$) such that for any $t \in [0,T]$, we have
				\begin{equation*}
					\begin{split}
						&2\mathbb{E}\int_{t}^{T} \Phi^{\theta,\mu}_s \left(Y^{\delta}_s-Y^{\delta'}_s\right) \left({\mathfrak{f}}_\delta(s,Y^\delta_s)-{\mathfrak{f}}_{\delta'}(s,Y^{\delta'}_s)\right)ds
						\leq\left(\delta+\delta'\right) \mathfrak{c}_{M}.
					\end{split}
				\end{equation*}
				Similarly, we can show that
				\begin{equation*}
					\begin{split}
						&2\mathbb{E}\int_{t}^{T} \Phi^{\theta,\mu}_s \left(Y^{\delta}_s-Y^{\delta'}_s\right) \left(h_\delta(s,Y^\delta_s)-h_{\delta'}(s,Y^{\delta'}_s)\right)d\kappa_s
						\leq\left(\delta+\delta'\right) \mathfrak{c}_{M,\zeta}.
					\end{split}
				\end{equation*}
				By the dominated convergence theorem, it follows that the right-hand side of (\ref{Baic coro approx--}) goes $0$ when $\delta, \delta'\rightarrow 0^{+}$. This implies that, for any $\theta, \mu >0$
				\begin{equation}\label{CV}
					\lim\limits_{\delta, \delta'\rightarrow 0^{+}}\left(\big\|Y^{\delta}-Y^{\delta'}\big\|^2_{\mathcal{H}^{2,\mathcal{Q}}_{\theta,\mu}}+\big\|Z^{\delta}-Z^{\delta'}\big\|^2_{\mathcal{H}^{2,\ell^2}_{\theta,\mu}}\right)=0.
				\end{equation}
				Now, using the B-D-G inequality and following a similar computation as in Proposition \ref{Propo 1} (or \ref{Propo 2}), along with the convergence result (\ref{CV}), we can derive that
				$$
				\lim\limits_{\delta, \delta'\rightarrow 0^{+}}	\mathbb{E}\sup_{0 \leq t \leq T}\Phi^{\theta,\mu}_t\left|Y^\delta_t-Y^{\delta'}_t\right|^2=0.
				$$
				Then, $\left\{Y^\delta, Z^\delta\right\}_{0 < \delta \leq 1}$ converges in $\mathcal{C}^2_{\theta,\mu} \times \mathcal{H}^{2,\ell^2}_{\theta,\mu}$, i.e., there exists a process $(Y, Z)$ such that $Y^\delta \rightarrow Y$ in $\mathcal{C}^2_{\theta,\mu}$ and $Z^\delta \rightarrow Z$ in $\mathcal{H}^{2,\ell^2}_{\theta,\mu}$ as $\delta \rightarrow 0^{+}$.
				
				\subparagraph*{Step 4: Convergence of the process $\{K^{\delta}\}_{0 < \delta \leq 1}$ in $\mathcal{A}^2$ and the integrals $\{\int_{0}^{\cdot}\mathfrak{f}(s,Y^{\delta}_s)ds,\int_{0}^{\cdot}h(s,Y^{\delta}_s)d\kappa_s\}_{0 < \delta \leq 1}$ in $\mathbb{L}^2$.}\emph{}\\
				From the RGBDSDE-NL (\ref{Approximating}), for any $0 \leq \delta \leq 1$, we have
				$$
				K^\delta_t=Y^\delta_0-Y^\delta_t-\int_{0}^{t}{\mathfrak{f}}_\delta(s,Y^\delta_s)ds-\int_{0}^{t}{h}_\delta(s,Y^\delta_s)d\kappa_s-\int_{0}^{t}\mathfrak{g}(s)\overleftarrow{dB}_s+\sum_{k=1}^{d}\int_{0}^{t}Z^{\delta,(k)}_s dH^{(k)}_s.
				$$
				By employing a similar argument, we also obtain
				\begin{equation}\label{Skoro}
					\lim\limits_{\delta, \delta'\rightarrow 0^{+}}	\mathbb{E}\sup_{0 \leq t \leq T}\left|K^\delta_t-K^{\delta'}_t\right|^2=0.
				\end{equation}
				In conclusion $\left\{Y^\delta, Z^\delta, K^\delta\right\}$ is a Cauchy sequence in $\mathcal{B}^2_{\theta,\mu}$.
				
				Now, from the definition of Yosida approximation (\ref{Yusida system}), we have
				$$
				\mathfrak{f}_{\delta}\left(s,Y^\delta_s\right)=\mathfrak{f}\left(s,Y^\delta_s+\delta\mathfrak{f}_{\delta}\left(s,Y^\delta_s\right)\right)~\mbox{ and }~h_{\delta}\left(s,Y^\delta_s\right)=h\left(s,Y^\delta_s+\delta h_{\delta}\left(s,Y^\delta_s\right)\right)
				$$
				Using Lemma \ref{UI Lemma}, along with conditions \textbf{(Y)}-(iv) and \textbf{(H-M)}-(ix), we have
				\begin{equation}\label{G1}
					\begin{split}
						\mathbb{E}\left(\int_{0}^{T}\mathfrak{f}_{\delta}\left(s,Y^\delta_s\right)ds\right)^2&\leq \frac{2}{\theta}\mathbb{E}\int_{0}^{T}\Phi^{\theta,\mu}\left\{\left|\frac{\varphi_s}{a_s}\right|^2 ds+\left|Y^\delta_s\right|^2 dV_s\right\}\leq \mathfrak{c}_\theta,
					\end{split}
				\end{equation}
				and
				\begin{equation}\label{G2}
					\begin{split}
						\mathbb{E}\left(\int_{0}^{T}h_{\delta}\left(s,Y^\delta_s\right)d\kappa_s\right)^2&\leq \frac{2}{\mu}\mathbb{E}\int_{0}^{T}\Phi^{\theta,\mu}\left\{\psi^2_s ds+\left|Y^\delta_s\right|^2 d\kappa_s\right\}\leq \mathfrak{c}_{\mu,\zeta}.
					\end{split}
				\end{equation}
				Then, for any $t \in [0,T]$, we have $\lim_{\delta \rightarrow 0^{+}} \int_{0}^{t} \delta \mathfrak{f}_{\delta}(s, Y^\delta_s) \, ds = 0$ in $\mathbb{L}^2$ and $\lim_{\delta \rightarrow 0^{+}} \int_{0}^{t} \delta h_{\delta}(s, Y^\delta_s) \, d\kappa_s = 0$ in $\mathbb{L}^2$. Therefore, by applying the partial converse of the Dominated Convergence Theorem, we deduce the existence of two subsequences $\left(\delta_k \mathfrak{f}_{\delta_k}(s, Y^{\delta_k}_s)\right)_{k \in \mathbb{N}}$ and $\left(\delta_k h_{\delta_k}(s, Y^{\delta_k}_s)\right)_{k \in \mathbb{N}}$ such that $\lim_{k \rightarrow +\infty} \delta_k \mathfrak{f}_{\delta_k}(s, Y^{\delta_k}_s) = 0$ $d\mathbb{P} \otimes ds$-a.e. and $\lim_{k \rightarrow +\infty} \delta_k h_{\delta_k}(s, Y^{\delta_k}_s) = 0$ $d\mathbb{P} \otimes d \kappa_s$-a.e. For simplicity, we continue to denote these two subsequences by the original notation, thus we have $\lim_{\delta \rightarrow 0^{+}} \delta \mathfrak{f}_{\delta}(s, Y^{\delta}_s) = 0$ $d\mathbb{P} \otimes ds$-a.e. and $\lim_{\delta \rightarrow 0^{+}} \delta h_{\delta}(s, Y^{\delta}_s) = 0$ $d\mathbb{P} \otimes d \kappa_s$-a.e. Moreover, given that $\left\{Y^\delta\right\}_{0 < \delta \leq 1}$ converges in $\mathcal{S}^2_{\theta,\mu}$ to $Y$, we have $\lim_{\delta \rightarrow 0^{+}} Y^\delta_s = Y_s$ a.s. for any $t \in [0,T]$. This, together with the continuity of the functions $y \mapsto \mathfrak{f}(s,y)$ and $y \mapsto h(s,y)$, allows us to deduce that
				$$
				\lim\limits_{\delta \rightarrow 0^{+}}\mathfrak{f}_{\delta}\left(s,Y^\delta_s\right)=\mathfrak{f}\left(s,Y_s\right)~d\mathbb{P} \otimes ds\mbox{-a.e.} \mbox{ and }\lim\limits_{\delta \rightarrow 0^{+}}h_{\delta}\left(s,Y^\delta_s\right)=h\left(s,Y_s\right)~d\mathbb{P} \otimes d\kappa_s\mbox{-a.e.}
				$$ 
				Combining this with (\ref{G1}) and (\ref{G2}), and applying the Lebesgue Dominated Convergence Theorem, we obtain $\lim_{\delta \rightarrow 0^{+}} \int_{\cdot}^{T} \mathfrak{f}_{\delta}(s, Y^\delta_s) \, ds = \int_{\cdot}^{T} \mathfrak{f}(s, Y_s) \, ds$ and $\lim_{\delta \rightarrow 0^{+}} \int_{\cdot}^{T} h_{\delta}(s, Y^\delta_s) \, d\kappa_s = \int_{t}^{T} h(s, Y_s)$ in $\mathcal{S}^2$.
				
				\subparagraph*{Step 5: The limiting process $(Y, Z, K)$ satisfies the RGBDSDE-NL (\ref{basic equation 1}).}
				\emph{}\\
				Using the B-D-G inequality and (\ref{CV}), we derive that $\lim_{\delta \rightarrow 0^{+}} \sum_{k=1}^{d} \int_{\cdot}^{T} Z^{\delta,(k)}_s \, dH^{(k)}_s = \sum_{k=1}^{d} \int_{\cdot}^{T} Z^{(k)}_s \, dH^{(k)}_s$ in $\mathcal{S}^2_{\theta,\mu}$. Then, taking the limit term by term in $\mathbb{L}^2$ in (\ref{Approximating}) along with the obtained convergence result, we get
				$$
				Y_t=\xi+\int_{t}^{T}\mathfrak{f}(s,Y_s)ds+\int_{t}^{T}h(s,Y_s)d\kappa_s+\int_{t}^{T}\mathfrak{g}(s)\overleftarrow{dB}_s+\left(K_T-K_t\right)-\sum_{k=1}^{d}\int_{t}^{T}Z^{(k)}_s dH^{(k)}_s.
				$$
				Additionally, from (\ref{Approximating}-(ii)), we have $Y_t \geq S_t$ for any $t \in [0,T]$. Finally, from (\ref{Skoro}), it follows that $\lim_{\delta \rightarrow 0^{+}} \|K^\delta - K\|^2_{\mathcal{A}^2} = 0$. Thus, the sequence $\{K^\delta\}_{0 < \delta \leq 1}$ converges to $K$ uniformly in $t$ in probability. Consequently, the measure $dK^\delta$ converges to $dK$ weakly in probability, leading to $\lim_{\delta \rightarrow 0^{+}} \int_{0}^{T} (Y^\delta_s - S_s) \, dK^\delta_s = \int_{0}^{T} (Y_s - S_s) \, dK_s = 0$ in probability. Passing again to a subsequence when necessary, we derive that $\int_{0}^{T} (Y_s - S_s) \, dK_s = 0$ almost surely. Note that Lemma \ref{UI Lemma}, Fatou's Lemma, and the convergence results allow us to derive that the limiting process verifies
				\begin{equation*}
					\begin{split}
						&\mathbb{E}\sup_{0 \leq t \leq T}\Phi^{\theta,\mu}_t\left|Y_t\right|^2+\mathbb{E}\int_{0}^{T}\Phi^{\theta,\mu}_s\left|Y_s\right|^2d \mathcal{Q}_s+\mathbb{E}\int_{0}^{T}\Phi^{\theta,\mu}_s\left\|Z_s \gamma_s \right\|^2_{\ell^2}ds+\mathbb{E}\left|K_T\right|^2\\
						& \leq \mathfrak{c}_{\theta,\mu,\zeta}\left(\mathbb{E}\Phi^{\theta,\mu}_T\left|\xi\right|^2+\mathbb{E}\int_{0}^{T}\Phi^{\theta,\mu}_s\left|\frac{\varphi_s}{a_s}\right|^2 ds+\mathbb{E}\int_{0}^{T}\Phi^{\theta,\mu}_s\left|\psi_s\right|^2 d\kappa_s \right.\\
						&\qquad\left.+\mathbb{E}\int_{0}^{T}\Phi^{\theta,\mu}_s\left|\mathfrak{g}(s)\right|^2 ds+\mathbb{E}\sup_{0 \leq t \leq T}\left|\Phi^{\theta,\mu}_t\left(S_t\right)^{+}\right|^2\right).
					\end{split}
				\end{equation*}
				
				\paragraph*{Step 2: General case}
				For any $\mathsf{p} \geq 1$, we consider
				\begin{equation}\label{New generators}
					\left\{
					\begin{split}
						\mathfrak{f}_\mathsf{p}(t,y)&=
						\left\{
						\begin{split}
							&\mathfrak{f}(t,y)-\mathfrak{f}(t,0)+\frac{\left(\left|\mathfrak{f}(t,0)\right|\wedge \mathsf{p}\right)}{\left|\mathfrak{f}(t,0)\right|}\mathfrak{f}(t,0)&\mbox{ if } \mathfrak{f}(t,0)\neq0.\\
							&\mathfrak{f}(t,y)&\mbox{ if } \mathfrak{f}(t,0)=0.
						\end{split}
						\right.\\
						h_\mathsf{p}(t,y)&=
						\left\{
						\begin{split}
							&h(t,y)-h(t,0)+\frac{\left(\left|h(t,0)\right|\wedge \mathsf{p}\right)}{\left|h(t,0)\right|}h(t,0)&\mbox{ if } h(t,0)\neq0.\\
							&h(t,y)&\mbox{ if } h(t,0)=0.
						\end{split}
						\right.
					\end{split}
					\right.
				\end{equation}
				Condition (\ref{boudness condition}) is satisfied:
				$$
				\sup_{0 \leq t \leq T}\left|\mathfrak{f}_\mathsf{p}(t,0)\right|^2+\sup_{0 \leq t \leq T}\left|h_\mathsf{p}(t,0)\right|^2 \leq \mathfrak{c}_p.
				$$
				By Part 1, there exists a unique process $\left(Y^{\mathsf{p}}, Z^{\mathsf{p}}, K^{\mathsf{p}}\right) \in \mathcal{B}^2_{\theta,\mu}$ such that
				\begin{equation}
					\left\{
					\begin{split}
						\text{(i)}&~Y^\mathsf{p}_t={\xi}+\int_{t}^{T}{\mathfrak{f}}_\mathsf{p}(s,Y_s)ds+\int_{t}^{T}{h}_\mathsf{p}(s,Y_s)d\kappa_s+\int_{t}^{T}\mathfrak{g}(s)\overleftarrow{dB}_s+\left(K^p_T-K^p_t\right)-\sum_{k=1}^{d}\int_{t}^{T}Z^{\mathsf{p},(k)}_s dH^{(k)}_s,\\
						\text{(ii)}&~Y^\mathsf{p}_t \geq S_t, ~t \in [0,T] ~\mbox{ and }~\int_{0}^{T} \left(Y^\mathsf{p}_{s}-S_{s}\right)dK^\mathsf{p}_s=0.
					\end{split}
					\right.
					\label{Approximating in p}
				\end{equation}
				From Definition (\ref{New generators}) and \textbf{(H-M)}-(xiii), we have
				\begin{equation}\label{CV1}
					\begin{split}
						\lim\limits_{\mathsf{p}\rightarrow+\infty}\left(\mathbb{E}\int_{0}^{T}\Phi^{\theta,\mu}_s\left|\frac{	\mathfrak{f}_\mathsf{p}(t,0)-	\mathfrak{f}(t,0)}{a_s}\right|^2ds+	\mathbb{E}\int_{0}^{T}\Phi^{\theta,\mu}_s\left|	h_\mathsf{p}(t,0)-h(t,0)\right|^2d\kappa_s\right)=0.
					\end{split}
				\end{equation}
				Now let $\mathsf{q} > \mathsf{p}$. By using Itô's formula, the Skorokhod condition in (\ref{Approximating in p}), the B-D-G inequality, and familiar computations, we end up with
				\begin{equation}\label{RH}
					\begin{split}
						&\mathbb{E}\sup_{0 \leq t \leq T}\Phi^{\theta,\mu}_t\left|Y^{\mathsf{q}}_t-Y^{\mathsf{p}}_t\right|^2+\mathbb{E}\int_{0}^{T}\Phi^{\theta,\mu}_s\left|Y^{\mathsf{q}}_s-Y^{\mathsf{p}}_s\right|^2d \mathcal{Q}_s+\mathbb{E}\int_{0}^{T}\Phi^{\theta,\mu}_s\left\|\left(Z^{\mathsf{q}}_s-Z^{\mathsf{p}}_s\right) \gamma_s \right\|^2_{\ell^2}ds\\
						& \leq \mathfrak{c}_{\theta,\mu,\zeta}\left(\mathbb{E}\int_{0}^{T}\Phi^{\theta,\mu}_s\left|\frac{\mathfrak{f}_\mathsf{q}(s,0)-\mathfrak{f}_\mathsf{p}(s,0)}{a_s}\right|^2 ds+\mathbb{E}\int_{0}^{T}\Phi^{\theta,\mu}_s\left|h_\mathsf{q}(s,0)-h_\mathsf{p}(s,0)\right|^2 d\kappa_s \right).
					\end{split}
				\end{equation}
				This, along with (\ref{CV1}), allows us to conclude that the right-hand side of (\ref{RH}) tends to $0$ as $\mathsf{p}, \mathsf{q} \rightarrow +\infty$. In addition, for any $\mathsf{p} \geq 1$,
				$$
				K^\mathsf{p}_t=Y^\mathsf{p}_0-Y^\mathsf{p}_t-\int_{0}^{t}{\mathfrak{f}}_\mathsf{p}(s,Y^\mathsf{p}_s)ds-\int_{0}^{t}{h}_\mathsf{p}(s,Y^\mathsf{p}_s)d\kappa_s-\int_{0}^{t}\mathfrak{g}(s)\overleftarrow{dB}_s+\sum_{k=1}^{d}\int_{0}^{t}Z^{\mathsf{p},(k)}_s dH^{(k)}_s
				$$
				Then, we also get
				\begin{equation*}
					\lim\limits_{\mathsf{p}, \mathsf{q}\rightarrow +\infty}	\mathbb{E}\sup_{0 \leq t \leq T}\left|K^\mathsf{p}_t-K^{\mathsf{q}}_t\right|^2=0.
				\end{equation*}
				Similarly to Part 1, we derive that $\left\{\left(Y^{\mathsf{p}}, Z^{\mathsf{p}}, K^{\mathsf{p}}\right)\right\}_{\mathsf{p} \geq 1}$ is a Cauchy sequence in $\mathcal{B}^2_{\beta,\mu}$, and the corresponding limit process $(Y, Z, K) \in \mathcal{B}^2_{\beta,\mu}$ verifies the RGBDSDE-NL (\ref{basic equation 1}).
			\end{proof}
			
			\subsubsection{General case}
			We now provide the main finding of this section.
			\begin{theorem}\label{thm}
				Suppose that \textbf{(H-M)} holds. Then, there exists a constant $\theta_0 > 0$ such that the RGBDSDE-NL (\ref{basic equation}) has a unique solution $\left(Y, Z, K\right) \in \mathcal{B}^2_{\theta,\mu}$ for any $\theta \geq \theta_0$ and $\mu > 0$.
			\end{theorem}
			\begin{proof}[Uniqueness] The uniqueness has previously been proven by Proposition \ref{Propo 2}.
				
				\proof[Existence] Using the Picard approximation sequence as our method for proving existence, we examine the sequence $\left\{Y^n, Z^n, K^n \right\}_{n \geq 0} \subset \mathcal{B}^2_{\theta,\mu}$ defined recursively as follows: $Y^0=Z^0=K^0=0$. Then, for any $n \geq 0$, we let $(Y^{n+1}, Z^{n+1}, K^{n+1})$ be the unique solution of the RGBDSDE-NL (\ref{basic equation}) associated with $(\xi, f(t,Y^{n+1}_t,Z^n_t),h(t,Y^{n+1}_t),\kappa_t,g(t,Y^n_t, Z^n_t),S_t)$. In other words, the process $(Y^{n+1}, Z^{n+1}, K^{n+1})$ satisfies
				\begin{equation}
					\left\lbrace 
					\begin{split}
						\text{(i) }&(Y^{n+1}, Z^{n+1}, K^{n+1}) \in \mathcal{B}^2_{\theta,\mu};\\
						\text{(ii) }&Y^{n+1}_{t}=\xi+\int_{t}^{T} f(s,Y^{n+1}_s,Z^n_s)ds+\int_{t}^{T}h(s,Y^{n+1}_s)d\kappa_s+\int_{t}^{T}g(s,Y^n_s,Z^n_s)\overleftarrow{dB}_s\\
						&\quad+\left(K^{n+1}_{T}-K^{n+1}_{t}\right)-\sum_{k=1}^{d}\int_{t}^{T}Z^{n+1,(k)}_s dH^{(k)}_s,~\forall t \in \left[0,T\right] ~\text{ a.s.};\\
						\text{(iii) }& Y^{n+1}_t \geq S_t,\text{ a.s. } \forall t \in [0,T]\mbox{ and } \int_{0}^{T} \left(Y^{n+1}_s-S_s\right)dK^{n+1}_s.
					\end{split}
					\right.
					\label{basic equation sequence}
				\end{equation}
				Using Theorem \ref{thm1}, we deduce that for every $n \geq 1$, the RGBDSDE-NL (\ref{basic equation sequence}) has a unique solution. To simplify notations, we set $\hat{\mathcal{R}}^{n+1}=\mathcal{R}^{n+1}-\mathcal{R}^n$ for $\mathcal{R}^n \in \left\{Y^n, Z^n, K^n\right\}$ and define $\mathfrak{f}_{n}(t)=f(t,Y^{n}_t,Z^{n-1}_t)-f(t,Y^{n-1}_t, Z^{n-2}_t)$, $\mathfrak{h}_{n}(t)=h(t,Y^n_t)-h(t,Y^{n-1}_t)$, and $\mathfrak{g}_{n}(t)=g(t,Y^n_t, Z^n_t)-g(t,Y^{n-1}_t, Z^{n-1}_t)$. From (\ref{basic equation sequence})-(ii), the state process $\hat{Y}^{n+1}$ satisfies the following BDSDE:
				\begin{equation}
					\begin{split}
						\hat{Y}^{n+1}_t=&\int_{t}^{T} \mathfrak{f}_{n+1}(s)ds+\int_{t}^{T}\mathfrak{h}_{n+1}(s)dA_s+\int_{t}^{T}\mathfrak{g}_n(s)\overleftarrow{dB}_s\\
						&+\left(\hat{K}^{n+1}_{T}-\hat{K}^{n+1}_{t}\right)-\sum_{k=1}^{d}\int_{t}^{T}Z^{n+1,(k)}_s dH^{(k)}_s,\quad t \in [0,T].
					\end{split}
					\label{Dynamic}
				\end{equation}
				To apply Corollary \ref{Used Coro} to the dynamic equation (\ref{Dynamic}), we require the following estimations concerning the newly introduced generators. Utilizing conditions \textbf{(H-M)}-(ii), (iii), (v), and (vii), we obtain the following results for any $\varepsilon_1, \varepsilon_2 > 0$:
				\begin{equation}
					\begin{split}
						2 \hat{Y}^{n+1}_s\mathfrak{f}_{n+1}(s)ds
						&\leq 2\left|\hat{Y}^{n+1}_s\right|\left(\lambda_s\left|\hat{Y}^{n+1}_s\right|^2+\eta_s\left\|\hat{Z}^{n}_s\gamma_s\right\|_{\ell^2}\right)ds\\
						&\leq \left(2+\frac{1}{\varepsilon_1}\right)\left|\hat{Y}^{n+1}_s\right|^2dV_s+\varepsilon_1\left\|\gamma_s\hat{Z}^{n}_s\right\|^2_{\ell^2}ds,
					\end{split}
					\label{Eq4}
				\end{equation}
				and
				\begin{equation}
					\begin{split}
						2 \hat{Y}^{n+1}_s\mathfrak{h}_{n+1}{\hat{\Theta}^n}(s)dA_s&\leq 2\varrho_s\left|\hat{Y}^{n+1}_s\right|^2d\kappa_s \leq \varepsilon_2   \left|\hat{Y}^{n}_s\right|^2d\kappa_s.
					\end{split}
					\label{Eq5}
				\end{equation}
				and
				\begin{equation}
					\begin{split}
						\left| \mathfrak{g}_n(s) \right|^2 ds &\leq \rho_s \left|\hat{Y}^n_s\right|^2ds+\alpha\left\|\gamma_s\hat{Z}^{n}_s\right\|^2_{\ell^2}ds\leq  \left|\hat{Y}^n_s\right|^2dV_s+\alpha\left\|\gamma_s\hat{Z}^{n}_s\right\|^2_{\ell^2}ds.
					\end{split}
					\label{Eq6}
				\end{equation}
				Next, we apply the formula from Corollary \ref{Used Coro} to ${\Phi}^{\theta,\mu}\left|\hat{Y}^{n+1}\right|^2$, where $\hat{Y}^{n+1}$ is defined by (\ref{Dynamic}). Utilizing the Skorokhod condition of the reflection processes $\hat{K}^{n+1}$, which implies that $\hat{Y}^{n+1}_{s}d\hat{K}^{n+1}_s \leq 0$, along with inequalities (\ref{Eq4}), (\ref{Eq5}), (\ref{Eq6}), and following similar arguments as in Proposition \ref{Propo 1}, we derive, for any $\theta$, $\mu$, $\varepsilon_1$, $\varepsilon_2 > 0$,   
				\begin{equation*}
					\begin{split}
						&\left(\theta-2-\frac{1}{\varepsilon_1}\right)\mathbb{E} \int_{t}^{T}{\Phi}^{\theta,\mu}_s\left|\hat{Y}^{n+1}_s\right|^2dV_s+\mu\mathbb{E} \int_{t}^{T}{\Phi}^{\theta,\mu}_s\left|\hat{Y}^{n+1}_s\right|^2d\kappa_s+\mathbb{E}\int_{t}^{T}{\Phi}^{\beta,\mu}_s\left\|\gamma_s\hat{Z}^{n+1}_s\right\|^2_{\ell^2}ds\\
						&\leq 2\mathbb{E}\int_{t}^{T}{\Phi}^{\beta,\mu}_s\left|\hat{Y}^n_s\right|^2d V_s+\varepsilon_2   \mathbb{E}\int_{t}^{T}{\Phi}^{\beta,\mu}_s\left|\hat{Y}^n_s\right|^2d \kappa_s+\left(\varepsilon_1+\alpha\right)\mathbb{E}\int_{t}^{T}{\Phi}^{\beta,\mu}_s\left\|\gamma_s\hat{Z}^{n}_s\right\|^2_{\ell^2}ds\\
						&\leq (\varepsilon_1+\alpha) \left(\frac{2}{\varepsilon_1+\alpha}\mathbb{E}\int_{t}^{T}\hat{\Phi}^{\beta,\mu}_s\left|\hat{Y}^n_s\right|^2d V_s+\frac{\varepsilon_2  }{\varepsilon_1+\alpha}\mathbb{E}\int_{t}^{T}\hat{\Phi}^{\beta,\mu}_s\left|\hat{Y}^n_s\right|^2d \kappa_s+\mathbb{E}\int_{t}^{T}\hat{\Phi}^{\beta,\mu}_s\left\|\gamma_s\hat{Z}^{n}_s\right\|^2_{\ell^2}ds\right).
					\end{split}
				\end{equation*}
				Fix $\varepsilon_1 > 0$, choose $\varepsilon_2 = (\varepsilon_1 + \alpha)\mu$, and define $\bar{\mathfrak{c}} = \frac{2}{\varepsilon_1 + \alpha}$, $\theta_0 = 2 + \frac{1}{\varepsilon_1} + \bar{\mathfrak{c}}$, and $\mu_0 = 1 + \frac{1}{\varepsilon_2}$. By choosing $\theta \geq \theta_0$, we obtain
				\begin{equation*}
					\begin{split}
						&\bar{\mathfrak{c}}\mathbb{E} \int_{t}^{T}{\Phi}^{\beta,\mu}_s\left|\hat{Y}^{n+1}_s\right|^2dV_s+\mu\mathbb{E} \int_{t}^{T}{\Phi}^{\beta,\mu}_s\left|\hat{Y}^{n+1}_s\right|^2d\kappa_s+\mathbb{E}\int_{t}^{T}{\Phi}^{\beta,\mu}_s\left\|\hat{Z}^{n+1}_s\right\|^2_{\ell^2}ds\\
						&\leq (\varepsilon_1+\alpha) \left(\bar{\mathfrak{c}}\mathbb{E}\int_{t}^{T}{\Phi}^{\beta,\mu}_s\left|\hat{Y}^n_s\right|^2d V_s+\mu\mathbb{E}\int_{t}^{T}{\Phi}^{\beta,\mu}_s\left|\hat{Y}^n_s\right|^2d \kappa_s+\mathbb{E}\int_{t}^{T}{\Phi}^{\beta,\mu}_s\left\|\hat{Z}^{n+1}_s\right\|^2_{\ell^2}ds\right).
					\end{split}
				\end{equation*}
				Using simple iterations, we establish that
				\begin{equation}
					\begin{split}
						&\bar{\mathfrak{c}}\mathbb{E} \int_{t}^{T}{\Phi}^{\beta,\mu}_s\left|\hat{Y}^{n+1}_s\right|^2dV_s+\mu\mathbb{E} \int_{t}^{T}{\Phi}^{\beta,\mu}_s\left|\hat{Y}^{n+1}_s\right|^2d\kappa_s+\mathbb{E}\int_{t}^{T}{\Phi}^{\beta,\mu}_s\left\|\hat{Z}^{n+1}_s\right\|^2_{\ell^2}ds\\
						&\leq (\varepsilon_1+\alpha)^n \left(\bar{\mathfrak{c}}\mathbb{E} \int_{t}^{T}{\Phi}^{\beta,\mu}_s\left|\hat{Y}^{1}_s\right|^2dV_s+\mu\mathbb{E} \int_{t}^{T}{\Phi}^{\beta,\mu}_s\left|\hat{Y}^{1}_s\right|^2d\kappa_s+\mathbb{E}\int_{t}^{T}{\Phi}^{\beta,\mu}_s\left\|\hat{Z}^{1}_s\right\|^2_{\ell^2}ds\right).
					\end{split}
				\end{equation}
				Choosing $\varepsilon_1 > 0$ such that $\varepsilon_1 < 1 - \alpha$, we deduce that $\{Y^n, Z^n\}_{n \geq 0}$ is a Cauchy sequence in the Banach space $\mathcal{H}^{2,\mathcal{Q}}_{\theta,\mu} \times \mathcal{H}^{2,\ell^2}_{\theta,\mu}$ for any $\theta \geq \theta_0$ and $\mu > 0$. It remains to show that the sequence $\{Y^n\}_{n \geq 0}$ is a Cauchy sequence in the Banach space $\mathcal{S}^2_{\theta,\mu}$. To this end, for any two integers $n, m \geq 0$, we set $\mathcal{R}^{n, m} = \mathcal{R}^n - \mathcal{R}^m$ for $\mathcal{R}^n \in \{Y^n, Z^n, K^n\}$. Then, we simply apply Corollary \ref{Used Coro} to $\Phi^{\theta,\mu}_{\tau} |Y^{n+1,m+1}_{\tau}|^2$ along with the B-D-G inequality, and we follow a similar argument as in Theorem \ref{thm1}, which yields
				\begin{equation}
					\begin{split}
						&\mathbb{E}\left[\sup_{0 \leq t \leq T} {\Phi}^{\beta,\mu}_\tau\left|Y^{n+1,m+1}_t\right|^2\right]\\
						& \leq \mathfrak{c}_{\theta,\mu,\alpha}\left( \left\|Y^{n+1,m+1}\right\|^2_{{\mathcal{H}}^{2,\mathcal{Q}}_{\theta,\mu}}+\left\|Z^{n+1,m+1}\right\|^2_{{\mathcal{H}}^{2,\ell^2}_{\theta,\mu}} + \left\|Y^{n,m}\right\|^2_{\mathcal{H}^{2,\mathcal{Q}}_{\beta,\mu}}+\left\|Z^{n,m}\right\|^2_{{\mathcal{H}}^{2,\ell^2}_{\theta,\mu}} \right).
					\end{split}
				\end{equation}
				Since $\{Y^n, Z^n, U^n\}_{n \geq 0}$ is a Cauchy sequence in the Banach space $\mathcal{H}^{2,\mathcal{Q}}_{\theta,\mu} \times \mathcal{H}^{2,\ell^2}_{\theta,\mu}$ for any $\theta \geq \theta_0$ and $\mu > 0$ (with $\theta_0$ defined above), we deduce that $\{Y^n\}_{n \geq 0}$ is also a Cauchy sequence in $\mathcal{S}^2_{\theta,\mu}$ for the same range of values of $\theta$ and $\mu$. Finally, by a familiar argument, we obtain
				$$
				\lim\limits_{n,m\rightarrow +\infty}	\mathbb{E}\sup_{0 \leq t \leq T}\left|{K}^{n+1,m+1}_t\right|^2=0.
				$$
				Therefore, the sequence $\{Y^n, Z^n, K^n\}_{n \geq 0}$ converges to a limit $(Y, Z, K)$ in the space $\mathcal{H}^{2,\mathcal{Q}}_{\theta,\mu} \times \mathcal{H}^{2,\ell^2}_{\theta,\mu} \times \mathcal{A}^2$ for any $\theta \geq \theta_0$ and $\mu > 0$. In other words, the triplet $(Y, Z, K)$ satisfies
				\begin{equation}
					\lim\limits_{n\rightarrow+\infty}\left\|\left(Y^n-Y, Z^n-Z, K^n-K\right)\right\|_{\mathcal{H}^{2,\mathcal{Q}}_{\theta,\mu} \times \mathcal{{H}}^{2,\ell^2}_{\theta,\mu}\times \mathcal{A}^2}=0,\quad \forall (\beta,\mu) \in [\theta_0,+\infty) \times (0,+\infty).
					\label{CV LPS}
				\end{equation}
				Now, let's focus on the convergence of the driver and martingale parts of (\ref{basic equation sequence})-(ii).\\ 
				To this end, based on (\ref{CV LPS}), we derive that $\{Y^n, Z^n\}_{n \geq 0}$ is bounded in $\mathcal{H}^{2,\mathcal{Q}}_{\theta,\mu} \times \mathcal{H}^{2,\ell^2}_{\theta,\mu}$ (another way is to do similar argumentation as in Lemma \ref{UI Lemma}). Then, on the one hand, using \textbf{(H-M)}-(v)-(x), we have $\lim \limits_{n\rightarrow+\infty} f(s,Y^{n+1}_s,Z^n_s)=f(s,Y_s,Z_s)$, $d\mathbb{P} \otimes dt$-a.e. Similarly, we derive that $\lim \limits_{n\rightarrow+\infty} h(s,Y^{n+1}_s)=h(s,Y_s)$, $d\mathbb{P} \otimes d\kappa_t$-a.e. On the other hand, from \textbf{(H-M)}-(ix), for any $n \geq 0$ and every $\theta,\mu>0$, we infer that
				\begin{equation*}
					\begin{split}
						\mathbb{E}\left(\int_{0}^{T}f(s,Y^{n+1}_s,Z^n_s) d s \right)^2+\mathbb{E}\left(\int_{0}^{T}h(s,Y^n_s) d \kappa_s \right)^2.
						\leq \mathfrak{c}_{\theta,\mu,\zeta}.
					\end{split}
					\label{AF0}
				\end{equation*} 
				Then, by the Lebesgue dominated convergence theorem, we can derive
				\begin{equation}\label{CVV}
					\mathbb{E}\left(\int_{0}^{T} \left(f(s,Y^{n+1}_s,Z^n_s)-f(s,Y_s,Z_s)\right) ds\right)^2+\mathbb{E}\left(\int_{0}^{T} \left(h(s,Y^{n+1}_s)-h(s,Y_s)\right) d\kappa_s\right)^2\xrightarrow[n\rightarrow+\infty]{}0.
				\end{equation}
				By the B-D-G inequality, \textbf{(H-M)}-(vii), and (\ref{CV LPS}), we have
				\begin{equation}
					\begin{split}
						&\mathbb{E}\left[\sup_{t \in [0,T]}\left|\int_{t}^{T} \left(g(s,Y^n_s, Z^n_s)-g(s,Y_s, Z_s)\right) \overleftarrow{dB}_s\right|^2\right]\\
						&\leq \mathfrak{c} \mathbb{E}\int_{0}^{T} {\Phi}^{\theta,\mu}_s \left|g(s,Y^n_s, Z^n_s)-g(s,Y_s, Z_s)\right|^2 ds\\
						&\leq \mathfrak{c}\left(\left\|Y^n-Y\right\|^{2}_{{\mathcal{H}}^{2,\mathcal{Q}}_{\theta,\mu}}+\alpha\left\{\left\|Z^n-Z\right\|^{2}_{{\mathcal{H}}^{2,\ell^2}_{\theta,\mu}}\right\}\right)\xrightarrow[n\rightarrow+\infty]{}0.
					\end{split}
					\label{AF1}
				\end{equation}
				Using once more (\ref{CV LPS}) and the B-D-G inequality, which yields to
				\begin{equation}
					\begin{split}
						\mathbb{E}\left[\sup_{t \in [0,T]}\left|\sum_{k=1}^{d}\int_{t}^{T}Z^{n+1,(k)}_s dH^{(k)}_s-\sum_{k=1}^{d}\int_{t}^{T}Z^{(k)}_s dH^{(k)}_s\right|^2\right] \leq \mathfrak{c}\left\|Z^{n+1}-Z\right\|^2_{{\mathcal{H}}^{2,\ell^2}_{\theta,\mu}}.
					\end{split}
					\label{AF2}
				\end{equation}
				Next, for any $n \geq 1$, we have
				\begin{equation*}
					\begin{split}
						K^{n+1}_t:=Y^n_0-Y^n_{t}-\int_{0}^{t} f(s,Y^{n+1}_s,Z^n_s)ds-\int_{0}^{t}h(s,Y^{n+1}_s)dA_s
						+\int_{0}^{t}g(s,Y^n_s,Z^n_s)\overleftarrow{dB}_s+\sum_{k=1}^{d}\int_{0}^{t}Z^{n+1,(k)}_s dH^{(k)}_s.
					\end{split}
				\end{equation*}
				From (\ref{CV LPS}), (\ref{CVV}), (\ref{AF1}), (\ref{AF2}), we deduce that
				$$
				K_t=Y_0-Y_t-\int_{0}^{t} f(s,Y_s,Z_s)ds-\int_{0}^{t}h(s,Y_s)d\kappa_s+\int_{0}^{t}g(s,Y_s,Z_s)\overleftarrow{dB}_s+\sum_{k=1}^{d}\int_{0}^{t}Z^{(k)}_sdH^{(k)}_s.
				$$
				Finally, using the same logic as in Part 1 of the proof of Theorem \ref{thm1}, we can easily see that the process $(Y, Z, K)$ solves the RGBDSDE-NL (\ref{basic equation}).
			\end{proof}
			
			\section{Application to an obstacle problem of  stochastic IPDE with nonlinear Neumann boundary conditions }
			\label{Sec4}
			In this section, with the assistance of the RGBDSDE-NL \eqref{basic equation}, we give a probabilistic formula for a stochastic viscosity solution for an obstacle problem of stochastic integral-partial differential equations  with nonlinear Neumann boundary conditions (SIPDE-NBC) of parabolic type under a special assumptions. 
			
			Our strategy involves first establishing the result for SIPDE-NBC. The main result then follows by adapting the arguments presented in \cite{aman2012reflected}.
			
			To facilitate a rigorous analysis, we assume without loss of generality that the inhomogeneous Lévy process $L$ is devoid of a Brownian component and possesses bounded jumps, i.e.
			\begin{equation}
				L_t = \int_{0}^{t} {b}_s \, ds + \int_{0}^{t} \int_{|z|< 1} z \, \tilde{\mu}(ds, dz), \ t\leq T.
				\label{Ddecomposition}
			\end{equation}
			
			Let $\mathbf{F}^{B} \triangleq\left\{\mathcal{F}_{t, T}^{B}\right\}_{0 \leq t \leq T}$. By $\mathcal{M}_{0,T}^{B}$ we will denote all the $\mathbf{F}^{B}$-stopping times $\tau$ such that $0 \leq \tau \leq T, \mathbb{P}$-a.s.. For generic Euclidean spaces $E$ and $E_{1}$ we introduce the following vector spaces of functions:
			
			\begin{itemize}
				\item $\mathcal{C}^{k_1, k_2}\left([0,T] \times E; E_{1}\right)$ stands for the space of all $E_{1}$-valued functions defined on $[0,T] \times E$ which are $k_1$ times continuously differentiable in $t$ and $k_2$ times continuously differentiable in $x$, and $\mathcal{C}^{k_1,k_2}_{b}\left([0,T] \times E; E_{1}\right)$ denotes the subspace of $\mathcal{C}^{k_1, k_2}\left([0,T] \times E; E_{1}\right)$ in which all functions have uniformly bounded partial derivatives.
				\item For any sub-$\sigma$-field $\mathcal{G} \subseteq \mathcal{F}_{T}^{B}, \mathcal{C}^{k_1,k_2}\left(\mathcal{G},[0,T] \times E; E_{1}\right)\left(\right.$or $\left.\mathcal{C}_{b}^{k_1, k_2}\left(\mathcal{G},[0, T] \times E; E_{1}\right)\right)$ denotes the space of all $\mathcal{C}^{k_1,k_2}\left([0, T] \times E ; E_{1}\right)\left(\right.$or $\mathcal{C}_{b}^{k_1, k_2}\left([0,T] \times E ; E_{1}\right)$-valued random variables that are $\mathcal{G} \otimes \mathcal{B}([0, T] \times E)$-measurable.
				\item $\mathcal{C}^{k_1, k_2}\left(\mathbf{F}^{B},[0, T] \times E ; E_{1}\right)$ $\left(\text{or} \ \mathcal{C}_{b}^{k_1, k_2}\left(\mathbf{F}^{B},[0,T] \times E; E_{1}\right)\right)$ is the space of all random fields $\alpha \in \mathcal{C}^{k_1, k_2}\left(\mathcal{F}_{T}^{B},[0, T] \times E ; E_{1}\right)$ $\left(\right.$ or $\left.\mathcal{C}_{b}^{k_1, k_2}\left(\mathcal{F}_{T}^{B},[0, T] \times E ; E_{1}\right)\right)$, such that for fixed $x \in E$, the mapping $(t, w) \rightarrow \alpha(t, \omega, x)$ is $\mathbf{F}^{B}$-progressively measurable.
				\item For any sub-$\sigma$-field $\mathcal{G} \subseteq \mathcal{F}_{T}^{B}$ and real number $p \geq 0, L^{p}(\mathcal{G} ; E)$ stands for all $E$-valued $\mathcal{G}$-measurable random variables $\xi$ such that $\mathbb{E}|\xi|^{p}<\infty$.
			\end{itemize}
			
			Furthermore, for $(t, x, y) \in[0,T] \times \mathbb{R}^{l} \times \mathbb{R}$, we write
			$$
			D_{x}=\left(\frac{\partial}{\partial x_{1}}, \ldots, \frac{\partial}{\partial x_{l}}\right), \quad D_{x x}=\left(\partial_{x_{i} x_{j}}^{2}\right)_{i, j=1}^{l}, \quad D_{y}=\frac{\partial}{\partial y}, \quad D_{t}=\frac{\partial}{\partial t}.
			$$
			
			The meaning of $D_{x y}$ and $D_{y y}$ is then self-explanatory.
			
			%
			%
			%
			
			
			\subsection{Reflected SDE driven by inhomogeneous Lèvy process}
			Let $G$ be an open connected bounded domain of $\mathbb{R}^l$ $(l\geq 1)$. We suppose that $G$ is a smooth domain, which is such that for a function $\Psi \in C^2_b (\mathbb{R}^l)$, $G$ and $\partial G$ are characterized by  $G=\left\lbrace x\in \mathbb{R}^l,  \Psi(x)>0 \right\rbrace$ and $\partial G = \left\lbrace x\in \mathbb{R}^l, \Psi(x)=0 \right\rbrace,$ and for any $x\in \partial G$, $\nabla \Phi(x)$
			is the unit normal vector pointing toward the interior of $G$. Furthermore, the interior sphere condition holds (see \cite{pardoux1998generalized}, pp.551), i.e.
			there exists  $m>0$ such that for any $x \in \partial G, x' \in \bar{G},$
			\begin{equation}\label{sphere condition}
				|x'-x|^2 + m\langle \nabla\Psi(x),x'-x\rangle \geq 0.
			\end{equation}
			
			For every $(t,x) \in [0,T]\times \bar{G}$, we consider the following reflected stochastic differential equation driven by the inhomogeneous Lèvy process $L$:
			\begin{equation}\label{ve2}
				\left\{
				\begin{array}{lll}
					X^{t,x} _s \in \bar{G} \text{ and }  X^{t,x} _{s\wedge t}=x, \text{ for all } s\geq 0,\\
					0=\kappa^{t,x}_r\leq \kappa^{t,x}_s \leq \kappa^{t,x}_v \text{ for all } 0\leq r \leq t \leq s \leq v,\\
					X^{t,x} _s = x +  \int_{t}^{s}  \sigma(X ^{t,x}_{r^-}) dL_r + \int_{t}^{s} \nabla \Psi(X^{t,x}_{r}) d\kappa^{t,x}_r, ~~0\leq t\leq s\leq T,\\
					{\kappa^{t,x} _s = \int_{t}^{s} \mathds{1}_{\left\lbrace X^{t,x}_r \in \partial G \right\rbrace }d\kappa^{t,x}_r}, \ \kappa^{t,x}_. \ {is ~~ increasing}.
				\end{array}
				\right.
			\end{equation}
			where $\sigma: \mathbb{R}^l \to \mathbb{R}^{l}$ is globally Lipschitz: $ \exists C > 0, \forall x,x' \in \mathbb{R}^l, \|\sigma(x)-\sigma(x')\| \leq C|x-x'|,$ and satisfies the linear growth condition: $ \|\sigma(x)\| \leq C(1+|x|)$.
			
			%
			First, we state some properties of the processes $(X^{t,x}_s,\kappa^{t,x}_s)_{s\geq 0}$ which can be found in \cite{fujiwara1985stochastic} or  \cite{rong2006theory}.
			\begin{proposition}\label{Proposition est X}
				For all $T\geq 0$, $\mu>0$ and $p\geq 2$, there exists a positive constant $C$ such that for all $ t , t' \in [0,T]$ and $x,x' \in \bar{G}$,
				\begin{itemize}
					\item [(i)]	$\displaystyle\mathbb{E} \left[ \underset{s\in [0,T]}{\sup} \left|X^{t,x}_s -X^{t',x'}_s\right|^p \right] + \mathbb{E} \left[\underset{s\in [0,T]}{\sup}  \left|\kappa^{t,x}_s -\kappa^{t',x'}_s\right|^p \right]   \leq C \left(\left|x-x'\right|^p+\left|t-t'\right|^{p/2}\right).$\\
					\item [(ii)]	$\displaystyle\mathbb{E} \left[\underset{s\in [0,T]}{\sup} \left|X^{t,x}_s\right|^p\right]+\mathbb{E}\left[\left|\kappa^{t,x}_T\right|^p\right] + \mathbb{E} \left[ e^{\mu \kappa^{t,x}_T}\right]\leq C.$
				\end{itemize}
			\end{proposition}
			Moreover, we have the following result of existence and uniqueness of the solution of \eqref{ve2}. Here, the argument follows the same procedure as in \cite{Lions_Sznitman}. For reader convenience we outline the main steps of the proof.
			\begin{proposition}\label{th.ex.X}
				There exists a unique solution $(X^{t,x}_s,\kappa^{t,x}_s)_{s\geq 0}$ for \eqref{ve2}.
			\end{proposition}
			%
			%
					%
					\begin{proof} The uniqueness result follows from Proposition \ref{Proposition est X}. For the existence, we proceed as follows.
						To simplify the notation in the proof of the above proposition, we will  denote $\mathfrak{R} = \mathfrak{R}^{t,x}$ for $\mathfrak{R} =X, \kappa$. 
						For each $n \in \mathbb{N}$, define the truncated process:
						\begin{equation}
							L_t^n = \int_0^t b_s ds + \int_0^t \int_{\frac{1}{n}<|z|<1}z\mu(ds,dz) - \int_0^t \int_{\frac{1}{n}<|z|<1}z F_s(dz)ds
						\end{equation}
						Since $L^n$ has finite variation, the Lipschitz and growth conditions hold on $\sigma$, and the domain $G$ satisfies the interior sphere condition \eqref{sphere condition}, then from \cite{Lions_Sznitman}, for each $n$, there exists a unique solution $(X^n, \kappa^n)$ to the truncated problem:
						\begin{equation}
							X_s^n = x + \int_t^s \sigma(X_{r-}^n)dL_r^n + \int_t^s \nabla\Psi(X_r^n)d\kappa_r^n
						\end{equation}
						Moreover, by Proposition \ref{Proposition est X}
						and Jakubowski's compactness criterion for processes with right-continuous with left limits paths (see, e.g., \cite{Jakubowski}), there exists a subsequence $(n_k)_{k\ge 0}$ and a limit process $(X,\kappa,L)$ such that:
						\[ (X^{n_k}, \kappa^{n_k}, L^{n_k}) \xrightarrow[k\to\infty]{} (X, \kappa, L). \]
						Finally, the limit $(X,\kappa)$ satisfies the original reflected SDE \eqref{ve2}. The convergence holds in the Skorokhod topology, and the stochastic integrals converge by continuity.
					\end{proof}

					\subsection{Markovian generalized doubly SDE} 
					For all $(t,x)\in [0,T]\times \bar{G}$, let $(X^{t,x}_s,\kappa^{t,x}_s)_{s\geq 0}$ represents the solution to the reflected SDE \eqref{ve2}. 
						Assume that the measurable functions $f,g,h$ and $H$ satisfy:
						\begin{itemize}
							\item [$(\mathcal{H}.1)$]  $f : \Omega_{2} \times [0,T]  \times \bar{G} \times \mathbb{R} \times \mathbb{R}^{d} \to \mathbb{R}$ satisfies for some constant $C>0$, $p\geq 2$ and some 
							$\mathcal{F}_t$-measurable processes $\eta, \phi :\Omega_2 \times [0,T] \rightarrow \mathbb{R}^{\ast}_+$:
							\begin{enumerate}
								\item[(i)] $\lvert f(\omega_{2},s,x,0,0)\rvert \leq C\left(1+\left|x\right|^p\right);$
								\item [(ii)] $\lvert f(\omega_{2},s,X^{t,x}_s,y,0)\rvert \leq \lvert f(\omega_{2},s,X^{t,x}_s,0,0)\rvert + \phi_s |y|,$
								\item[(iii)]  $(y-y')\big(f(\omega_{2},s,X^{t,x}_s,y,z)-f(\omega_{2},s,X^{t,x}_s,y',z)\big) \leq0,$
								\item[(iv)]  $\lvert f(\omega_{2},s,X^{t,x}_s,y,z)-f(\omega_{2},s,X^{t,x}_s,y,z')\rvert \leq \eta_s \|\gamma_s(z-z')\|_{\ell^2}$, 
								\item[(v)] There exist a positive constant $C_a$ and $\varepsilon>0$  such that $a_t^2=\eta_t^2+\phi_t^2>\varepsilon$ and $\int_{0}^{T}a^2_s ds \leq C_a<\infty$.
							\end{enumerate}
						\end{itemize}
						
						\begin{itemize}
							\item [$(\mathcal{H}.2)$]  $h : \Omega_{2} \times [0,T]\times \bar{G} \times \mathbb{R}\to\mathbb{R}$ satisfies for some constant $C>0$ and $p\geq 2$: 
							\begin{enumerate}
								\item[(i)] $\lvert h(s,x,0)\rvert \leq  C\left(1+\left|x\right|^p\right),$
								\item[(ii)] $\lvert h(s,X^{t,x}_s,y)\rvert\leq \lvert h(s,X^{t,x}_s,0)\rvert + \zeta |y|,$
								\item[(iii)] $(y-y') \big(h(s,X^{t,x}_s,y)-h(s,X^{t,x}_s,y')\big) \leq0.$
							\end{enumerate}
						\end{itemize}
						\begin{itemize}
							\item [$(\mathcal{H}.3)$]  $g \in \mathcal{C}^{0,2,3}\left([0,T]\times \bar{G} \times \mathbb{R}; \mathbb{R}\right)$.
						\end{itemize}
						
						\begin{itemize}
							\item [$(\mathcal{H}.4)$]  $H : \bar{G} \to \mathbb{R}$ satisfies for some $p\geq 2$,  $\lvert H(x)\rvert \leq C\left(1+\left|x\right|^p\right).$
						\end{itemize}
						%
						First, as in \cite{elhachemyjiea} (see Proposition 6.2) we shall give an auxiliary estimations.
						\begin{proposition}
							There exist a non-negative constant $C$ such that for all $t,t' \in [0,T] $ and $x,x' \in \bar{G}$,
							\begin{equation}\label{es1}
								\mathbb{E} \left[\underset{s \in [0,T]}{\sup}e^{\mu A_s} \left|X^{t,x}_s\right|^2\right] \leq C \quad\text{and} \quad\mathbb{E} \left[ \underset{s \in [0,T]}{\sup} e^{2\mu A_s}\left|X^{t,x}_s -X^{t',x'}_s\right|^4 \right] \leq C.
							\end{equation}
						\end{proposition}
					\begin{remark}
						Since $[0,T]\times \bar{G}$ is bounded, $X^{t,x}_s\in \bar{G}, \forall s\in [0,T]$ and the functions $f,g,h,H$ are continuous, then there exists a constant $C$ such that 
						$$\lvert f(s,X^{t,x}_s,0,0)\rvert + \lvert h(s,X^{t,x}_s,0)\rvert + \lvert g(s,X^{t,x}_s,0)\rvert+\lvert H(X^{t,x}_s)\rvert  \leq C,$$
						and
						\begin{equation}\label{est1}
							\mathbb{E} \left[\Phi_T^{\theta,\mu} \left|H(X^{t,x}_T)\right|^2\right]\leq C\mathbb{E} \left[\Phi_T^{\theta,\mu} \left(1+\left|X^{t,x}_T\right|^2\right)\right]<\infty,
						\end{equation}
							and
							\begin{equation}\label{est2}
								\mathbb{E}\int_{0}^{T}\Phi_s^{\theta,\mu}\frac{\left|f(s,X^{t,x}_s,0,0)\right|^2}{a^2_s}ds+ \mathbb{E}\int_{0}^{T}\Phi_s^{\theta,\mu}\left|h(s,X^{t,x}_s,0)\right|^2 d\kappa^{t,x}_s + \mathbb{E}\int_{0}^{T}\Phi_s^{\theta,\mu}\left|g(s,X^{t,x}_s,0)\right|^2 ds
								<\infty.
							\end{equation}	
						\end{remark}
						
						Therefore, as a consequence of Theorem \ref{thm} below we have the following result.
						
						\begin{corollary}
							For all $(t,x)\in [0,T]\times \bar{G}$, there exists a pair of processes $(Y^{t,x}_s,Z^{t,x}_s)_{s \in [t,T]}$ solution to the Markovian GBDSDE-NL:
							\begin{multline}\label{14}
								Y^{t,x}_s= H(X^{t,x}_T) + \int_{s}^{T} f(r,X^{t,x}_r, Y^{t,x}_r,Z^{t,x}_r)dr + \int_{s}^{T} h(r,X^{t,x}_r, Y^{t,x}_r)d\kappa^{t,x}_r + \int_{s}^{T} g(r,X^{t,x}_r, Y^{t,x}_r)\overleftarrow{dB}_r \\
								-\sum_{i=1}^{d}\int_{s}^{T}Z^{t,x(i)}_r dH^{(i)}_r,~~~~t\leq s \leq T
							\end{multline}
							and there exists $C>0$ such that
							\begin{equation}
								\mathbb{E}\left[\sup_{t\leq s\leq T}\Phi_s^{\theta,\mu}|Y^{t,x}_s|^2\right] +	\mathbb{E}\int_t^T \Phi_s^{\theta,\mu} |Y^{t,x}_s|^2 d\mathcal{Q}_s
								+\mathbb{E}\int_t^T \Phi_s^{\theta,\mu} \left\|Z^{t,x}_s\gamma_s \right\|_{\ell^2}^2 ds  \leq C\left[1+\mathbb{E}\left(\underset{t\leq s \leq T}{\sup} \Phi_s^{\theta,\mu}|X^{t,x}_s|^2\right)\right]<\infty.
							\end{equation}
						\end{corollary}
						For the remainder of this paper, we make the following assumption. Although this condition can be rigorously proven, we state it as an assumption for brevity, as the proof is lengthy and not central to our main results.
						\begin{itemize}
							\item [$(\mathcal{H}.5)$] For any $p\geq 2$,
							\begin{equation*}
								\mathbb{E}\left[\sup_{0\leq s\leq T}\left(\Phi_s^{\theta,\mu}\right)^{p/2}|Y^{t,x}_s|^p\right] +	\mathbb{E}\left[\left(\int_0^T \Phi_s^{\theta,\mu} |Y^{t,x}_s|^2 d\mathcal{Q}_s\right)^{p/2}\right] +\mathbb{E}\left[\left(\int_0^T \Phi_s^{\theta,\mu}\left\|Z^{t,x}_s\gamma_s\right\|^2ds\right)^{p/2}\right] <\infty.
							\end{equation*}
						\end{itemize}
						Using the same arguments as in \cite{pardoux2014stochastic}, we get the following continuity result.
						\begin{proposition}
							Let $(Y_{s}^{t,x}, Z_{s}^{t,x})$ be a solution to \eqref{14}. Then, the random field $(s,t,x) \rightarrow Y_{s}^{t, x}, \ (s,t,x) \in[0,T] \times[0,T] \times \overline{\mathrm{G}}$ is almost surely continuous.
						\end{proposition}

						\subsection{Stochastic viscosity solutions for SIPDE-NBC}\label{sec4.3}
						In this part, we introduce the notion of stochastic viscosity solutions to  SIPDE with nonlinear Neumann boundary conditions.
							We consider the stochastic IPDE with nonlinear Neumann boundary condition of the form:
							\begin{equation} \label{360}
								\left\{
								\begin{array}{lll}
									\displaystyle d u(t,x)+ \left[\mathcal{L}_tu(t,x)+f\left(t,x,u(t,x),\left(u^1_k(t,x)\right)^{\infty}_{k=1}\right)\right]dt+ g\left(t,x,u(t,x)\right)dB_t =0,& (t,x)\in [0,T[ \times G\\\\
									\displaystyle u(T,x)=H(x), & x\in \bar{G} \\\\
									\displaystyle \frac{\partial u}{\partial n}(t,x)+h(t,x,u(t,x))=0,& x\in \partial G,
								\end{array}
								\right.
							\end{equation}
							where 
							\begin{align*}
								&\mathcal{L}_t u(t,x) = b_t .\sigma(x). D_x u(t,x)  + \int_{\mathbb{R}} \left[u(t,x+\sigma(x)z)-u(t,x)-D_x u(t,x)\sigma(x)z \right] F_t(dz)\\
								& u^1_k(t,x) = \int_{\mathbb{R}} \left[u(t,x+\sigma(x)z)-u(t,x)-D_x u(t,x) z \right]p_k(z)F_t(dz)
							\end{align*}
							and, for every $x\in \partial G$,
							$$ \displaystyle{\frac{\partial \varphi}{\partial n} = \sum_{i=1}^{l} \frac{\partial \Psi}{\partial x_i} (x) \frac{\partial \varphi}{\partial x_i}(t,x).}$$
							We now define the notion of stochastic viscosity solution for the SIPDE $(f,g,h)$. We are inspired by the work {of Buckdahn and Ma \cite{buckdahn2001stochastic}} and we refer to their paper for a lucid discussion on this topic. We use some of their notation and follow the lines of their proofs to obtain our main result.
							
							Indeed, we will use the stochastic flow $\hat{\chi}(t,x,y) \in C(\mathbb{F}^B,[0,T] \times \mathbb{R}^n \times \mathbb{R})$, defined as the unique solution of the SDE which, in Stratonovich form, reads as follows:
							\begin{equation*}
								\hat{\chi}(t,x,y)= y + \int_0^t  g(s,x,\hat{\chi}(s,x,y)) \circ dB_s, \quad t \geq 0.
							\end{equation*}
							Under the assumption $(\mathcal{H}.3)$ the mapping $y \mapsto \hat{\chi}(t,x,y)$ defines a diffeomorphism for all $(t,x)$, $\mathbb{P}$-almost surely (see \cite{protter2005stochastic}). Denote the $y$-inverse of $\hat{\chi}(t,x,y)$ by $\tilde{\pi}(t,x,y)$. Then, since $\tilde{\pi}(t,x,\hat{\chi}(t,x,y)) = y$, one can show that (cf. Buckdahn and Ma, 2001)
							$$
							\tilde{\pi}(t,x,y) = y - \int_0^t  D_y \tilde{\pi}(s,x,y) g(s,x,y) \circ dB_s,
							$$
							where the stochastic integrals have to be interpreted in Stratonovich sense.
							
							Now, let us introduce the process $\chi \in C(\mathbb{F}^B,[0,T] \times \mathbb{R}^n \times \mathbb{R})$ as the solution to the equation
							\begin{equation} \label{eq chi}
								\chi(t,x,y) = y + \int_t^T  g(s,x,\chi(s,x,y)) \circ dB_s, \quad 0 \leq t \leq T.
							\end{equation}
							
							We note that due to the direction of the Itô integral, equation \eqref{eq chi} should be viewed as going from $T$ to $t$ (i.e., $y$ should be understood as the initial value). Then $y \mapsto \chi(s,x,y)$ will have the same regularity properties as those of $y \mapsto \hat{\chi}(s,x,y)$ for all $(s,x) \in [t,T] \times \mathbb{R}^n$, $\mathbb{P}$-almost surely. Hence if we denote by $\pi$ its $y$-inverse, we obtain
							$$
							\pi(t,x,y) = y - \int_t^T  D_y \pi(s,x,y) g(s,x,y) \circ dB_s.
							$$

							To simplify the notation we write:
							$$
							\mathcal{L}^t_{f,g}(\varphi(t, x))=-\mathcal{L}_t \varphi(t, x)-f\left(t,x,\varphi(t,x),\left(\varphi^1_k(t,x)\right)^{\infty}_{k=1}\right)+\frac{1}{2} g. D_{y}g(t,x,\varphi(t,x)).
							$$
							
							We now introduce the notion of a stochastic viscosity solution of \eqref{360} as follows.
							\begin{definition}\label{def11} \begin{itemize}
									\item[(1).] A random field $u \in \mathcal{C}\left(\mathbf{F}^{B},[0, T] \times \overline{\mathrm{G}}\right)$ is called a stochastic viscosity subsolution (resp. supersolution) of \eqref{360} if $u(T,x) \leq H(x)$ (resp. $ \geq $), for all $x \in \overline{\mathrm{G}}$, and if for any stopping time $\tau \in \mathcal{M}_{0,T}^{B}$, any state variable $\xi \in L^{0}\left(\mathcal{F}_{\tau}^{B}, G\right)$, and any random field $\varphi \in \mathcal{C}^{1,2}\left(\mathcal{F}_{\tau}^{B},[0, T] \times G\right)$, with the property that for $\mathbb{P}^2$-almost all $\omega_2 \in\{0<\tau<T\}$ the inequality
									$$
									u(t, \omega_2, x)-\chi(t, \omega_2, x, \varphi(t, x)) \leq 0=u(\tau(\omega_2), \xi(\omega_2))-\chi(\tau(\omega_2), \xi(\omega_2), \varphi(\tau(\omega_2), \xi(\omega_2))) \ (\text{resp.} \ \geq)
									$$
									is fulfilled for all $(t, x)$ in some neighbourhood  of $(\tau(\omega_2), \xi(\omega_2)$), the following conditions are satisfied:
									\begin{enumerate}
										\item [(a)] On the event $\{0<\tau<T\}$ the inequality
										\begin{equation}
											\mathcal{L}^t_{f,g}(\psi(\tau, \xi))-D_{y} \chi(\tau, \xi, \psi(\tau, \xi)) D_{t} \varphi(\tau, \xi) \leq 0. \ \ (\text{resp.} \ \geq) \label{190}
										\end{equation}
										holds $\mathbb{P}^2$-almost surely, where $\psi(t, x) \triangleq \chi(t,x,\varphi(t,x))$.\\
										\item[(b)]  On the event $\{0<\tau<T\} \cap\{\xi \in \partial \mathrm{G}\}$ the inequality
										\begin{equation}
											\min \left[\mathcal{L}^t_{f,g}(\psi(\tau, \xi))-D_{y} \chi(\tau, \xi, \varphi(\tau, \xi)) D_{t} \varphi(\tau, \xi) -\frac{\partial \psi}{\partial n}(\tau, \xi)-h(\tau, \xi, \psi(\tau, \xi))\right] \leq 0. \ \ (\text{resp.} \ \geq) \label{20}
										\end{equation}
										holds $\mathbb{P}$-almost surely with $\psi(t, x) \triangleq \chi(t,x,\varphi(t,x))$.
									\end{enumerate}
									\item[(2)] A random field $u \in \mathcal{C}\left(\mathbf{F}^{B},[0,T] \times \overline{\mathrm{G}}\right)$ is called a stochastic viscosity solution of \eqref{360} if it is both a stochastic viscosity subsolution and a supersolution.
								\end{itemize}
							\end{definition}

							\begin{remark}
								Observe that if $f, h$ are deterministic and $g \equiv 0$ then Definition \ref{def11} coincides with the deterministic case (see e.g. \cite{elhachemyjiea}). \end{remark}
							
							Now, let us recall a notion of random viscosity solution which will be a bridge linking the stochastic viscosity solution and its deterministic counterpart.
							
							\begin{definition}
								A random field $u \in \mathcal{C}\left(\mathbf{F}^{B},[0,T] \times \mathbb{R}^{l}\right)$ is called an $\omega$-wise viscosity solution if for $\mathbb{P}$-almost all $\omega \in \Omega, u(\omega, \cdot, \cdot)$ is a (deterministic) viscosity solution of the $\operatorname{SIPDE}(f, 0, h)$.
							\end{definition}
							Let us introduce the following lemma adapted from \cite{nualart2001backward}.
							
							\begin{lemma}
								Let $\bar{h}:\Omega\times[0,T]\times\mathbb{R}\to\mathbb{R}$ be a $\mathcal{P}\otimes\mathcal{B}_{\mathbb{R}}$-measurable function such that	
								\[
								|\bar{h}(s,y)| \leq a_s(y^2 \wedge |y|) \quad \text{a.s.},
								\]
								where $(a_s)_{s\geq0}$ is a positive predictable process with $\mathbb{E}\left[\int_0^T a_s^2 ds\right] < \infty$. Then	
								\[
								\sum_{t \leq s \leq T} \bar{h}(s,\Delta L_s) = \sum_{i=1}^{d} \int_t^T \langle \bar{h}(s,\cdot), p_i\rangle_{L^2(\nu)} dH_s^{(i)} + \int_t^T \int_{\mathbb{R}} \bar{h}(s,y) F_s(dy)ds,
								\]
								where
								\[
								\langle \bar{h}(s,\cdot), p_i\rangle_{L^2(\nu)} = \int_{\mathbb{R}} \bar{h}(s,y)p_i(y)F_s(dy)ds.
								\]
							\end{lemma}
							
							Next we introduce the Doss-Sussman transformation. It enables us to convert an SIPDE of the form $(f, g, h)$ to an ordinary differential equation of the form $(\widetilde{f}, 0, \widetilde{h})$, where $\widetilde{f}$ and $\widetilde{h}$ are certain well-defined random fields, which are defined in terms of $f, g, h$.

							\begin{remark}\label{rmk}
								Let us recall that under assumption $(\mathcal{H}.3)$ the random field $\chi$ belongs to $C^{0,2,2}\left(\mathbf{F}^{B}\right.$, $[0, T] \times \mathbb{R}^{n} \times \mathbb{R}$ ), and hence that the same is true for $\pi$. Then, considering the transformation $\psi(t, x)=\chi(t, x, \varphi(t, x))$, we obtain
								\begin{align*}
									D_{x} \psi&=D_{x} \chi+D_{y} \chi D_{x} \varphi \\
									D_{x x} \psi&=D_{x x} \chi+2\left(D_{x y} \chi\right)\left(D_{x} \varphi\right)^{*}+\left(D_{y y} \chi\right)\left(D_{x} \varphi\right)\left(D_{x} \varphi\right)^{*}+\left(D_{y} \chi\right)\left(D_{x x} \varphi\right).
								\end{align*}
								Moreover, since for all $(t, x, y) \in[0, T] \times \mathbb{R}^{n} \times \mathbb{R}$ the equality $\pi(t, x, \chi(t, x, y))=y$ holds $\mathbb{P}$ almost surely, we also have
								$$
								\begin{aligned}
									D_{x} \pi+D_{y} \pi D_{x} \chi & =0, \\
									D_{y} \pi D_{y} \chi & =1, \\
									D_{x x} \pi+2\left(D_{x y} \pi\right)\left(D_{x} \chi\right)^{*}+\left(D_{y y} \pi\right)\left(D_{x} \chi\right)\left(D_{x} \chi\right)^{*}+\left(D_{y} \pi\right)\left(D_{x x} \chi\right) & =0\\
									\left(D_{x y} \pi\right)\left(D_{y} \chi\right)+\left(D_{y y} \pi\right)\left(D_{x} \chi\right)\left(D_{y} \chi\right)+\left(D_{y} \pi\right)\left(D_{x y} \chi\right) & =0, \\
									\left(D_{y y} \pi\right)\left(D_{y} \chi\right)^{2}+\left(D_{y} \pi\right)\left(D_{y y} \chi\right) & =0,
								\end{aligned}
								$$
								where all the derivatives of the random field $\pi(\cdot, \cdot, \cdot)$ are evaluated at $(t, x, \chi(t, x, y)$), and all those of $\chi(\cdot, \cdot, \cdot)$ are evaluated at $(t, x, y)$.
							\end{remark} 
							\begin{proposition}\label{propsition Doss}
								Assume $(\mathcal{H}.1)-(\mathcal{H}.3)$ hold. A random field $u$ is a stochastic viscosity solution to the $\operatorname{SIPDE}(f, g, h)$ if and only if $v(\cdot, \cdot)=\varepsilon(\cdot, \cdot, u(\cdot, \cdot))$ is a stochastic viscosity solution to the IPDE $(\widetilde{f}, 0, \widetilde{h})$, where $(\widetilde{f}, \widetilde{h})$ are two coefficients that will be made precise later (see \eqref{form f} and \eqref{form h} below).
							\end{proposition} 
							
							\begin{proof}
								We shall only argue for the stochastic subsolution case, as the supersolution part is similar. Therefore, in the present proof we assume that $u \in \mathcal{C}\left(\mathbf{F}^{B},[0, T] \times \overline{\mathrm{G}}\right)$ is a stochastic viscosity subsolution of the $\operatorname{SIPDE}(f, g, h)$. It then follows that $v(\cdot, \cdot)=\\pi(\cdot, \cdot, u(\cdot, \cdot))$ belongs to $\mathcal{C}\left(\mathbf{F}^{B},[0, T] \times \overline{\mathrm{G}}\right)$. In order to show that $v$ is a stochastic viscosity subsolution of the $\operatorname{SPDE}(\widetilde{f}, 0, \widetilde{h})$, we let $\tau \in \mathcal{M}_{0, T}^{B}, \xi \in L^{0}\left(\mathcal{F}_{\tau}^{B},[0, T] \times \mathrm{G}\right)$ and $\varphi \in \mathcal{C}^{1,2}\left(\mathcal{F}_{\tau}^{B},[0, T] \times \mathbb{R}^{n}\right)$ be such that for $\mathbb{P}$-almost all $\omega \in\{0<\tau<T\}$ the inequality
								$$
								v(t, x)-\varphi(t, x) \leq 0=v(\tau(\omega), \xi(\omega))-\varphi(\tau(\omega), \xi(\omega))
								$$
								holds for all $(t, x)$ in some neighbourhood $\mathcal{V}(\omega, \tau(\omega), \xi(\omega))$ of $(\tau(\omega), \xi(\omega))$. Next we put $\psi(t, x)=\chi(t, x, \varphi(t, x))$. Since the mapping $y \mapsto \chi(t, x, y)$ is strictly increasing, for all $(t, x) \in$ $\mathcal{V}(\tau, \xi)$ we have that
								$$
								\begin{aligned}
									u(t, x)-\psi(t, x) & =\chi(t, x, v(t, x))-\chi(t, x, \varphi(t, x)) \\
									& \leq 0=\chi(\tau, \xi, v(\tau, \xi))-\chi(\tau, \xi, \varphi(\tau, \xi)) \\
									& =u(\tau, \xi)-\psi(\tau, \xi)
								\end{aligned}
								$$
								holds $\mathbb{P}$-almost surely on $\{0<\tau<T\}$. Moreover, since $u$ is a stochastic viscosity subsolution of the $\operatorname{SIPDE}(f, g, h)$, the inequality
								\begin{equation}\label{opfg}
									\mathcal{L}^t_{f,g}(\psi(\tau, \xi))-D_{y} \chi(\tau, \xi, \varphi(\tau, \xi)) D_{t} \varphi(\tau, \xi) \leq 0, 
								\end{equation}
								holds $\mathbb{P}$-almost everything on the event $\{0<\tau<T\}$. On the other hand, from Remark \ref{rmk} it follows that
								
								\begin{align*}
									\mathcal{L}_t \psi(t,x) &= b_t .\sigma(x). D_x \psi(t,x)  + \int_{\mathbb{R}} \left[\psi(t,x+\sigma(x)z)-\psi(t,x)-D_x \psi(t,x)\sigma(x)z \right] F_t(dz)\\
									&=\mathcal{L}_t^x \chi(t,x,\varphi(t,x)) + D_y \chi \cdot \mathcal{L}_t \varphi(t,x)  
									\\	&+ \int_{\mathbb{R}} \Big[ \chi(t,x+\sigma(x)z, \varphi(t,x+\sigma(x)z)) - \chi(t,x+\sigma(x)z, \varphi(t,x))
									- D_y \chi \cdot D_x \varphi \cdot \sigma(x) z \Big] F_t(dz).
								\end{align*}
								%
								
								Then, if we define the random field $\widetilde{f}$ by
								
								\begin{align}\label{form f}
									\tilde{f}(t,x,y,z) &= \frac{1}{D_y \chi} \Bigg[
									f\bigg(t,x,\chi,\left(\tilde{\psi}^1_k\right)_{k=1}^\infty\bigg) 
									- \frac{1}{2} g D_y g 
									+ \mathcal{L}_t^x \chi \nonumber
									\\&+ \int_{\mathbb{R}} 
									\Bigl[
									\chi\bigl(t,x+\sigma(x)r,\, \varphi(t,x+\sigma(x)r)\bigr) 
									- \chi
									- \bigl(D_x\chi + D_y\chi \cdot z\bigr)\sigma(x)r
									\Bigr] 
									F_t(dr)
									\Bigg]
								\end{align}
								where \(\chi = \chi(t,x,y)\) and  \(\tilde{\psi}^1_k\) is defined as follows
								\[
								\tilde{\psi}^1_k(t,x) = \int_{\mathbb{R}} \left[
								\chi(t,x+\sigma r, \varphi(t,x+\sigma r)) 
								- \chi 	- (D_x \chi + D_y \chi D_x \varphi) \cdot r
								\right] p_k(r) F_t(dr)
								\]
								we obtain
								$$
								D_{y} \pi(t, x, \psi(t, x)) \mathcal{L}^t_{f, g}(\psi(t, x))=\mathcal{L}^t_{\tilde{f}, 0}(\varphi(t, x)).
								$$
								Consequently, \eqref{opfg} becomes
								\begin{equation}\label{inf}
									\mathcal{L}^t_{\tilde{f}, 0}(\varphi(\tau, \xi))-D_{t} \varphi(\tau, \xi) \leq 0,
								\end{equation}
								and hence the Doss-Sussman transformation converts an SIPDE of the form $(f, g, h)$ to one of the form $(f, 0, h)$, provided a similar transformation for the random field $h$ also works for the inequality in \eqref{190}. This establishes part (a) of Definition \ref{def11}.\\
								In order to establish part (b) in Definition \ref{def11} we notice that for all $(t, x) \in[0, T] \times \partial \mathrm{G}$, the following string of equalities holds:
								$$
								\begin{aligned}
									\frac{\partial \psi}{\partial n}(t, x) & =\left\langle D_{x} \psi(t, x), \nabla \Psi(x)\right\rangle \\
									& =\left\langle D_{x} \chi(t, x, \varphi(t, x)), \nabla \Psi(x)\right\rangle+D_{y} \chi(t, x, \varphi(t, x))\left\langle D_{x} \varphi(t, x), \nabla \psi(x)\right\rangle \\
									& =\left\langle D_{x} \chi(t, x, \varphi(t, x)), \nabla \Psi(x)\right\rangle+D_{y} \chi(t, x, \varphi(t, x)) \frac{\partial \varphi}{\partial n}(t, x) .
								\end{aligned}
								$$
								Hence,
								$$
								\begin{aligned}
									\frac{\partial \psi}{\partial n}(\tau, \xi)+h(\tau, \xi, \psi(\tau, \xi))= & D_{y} \chi(\tau, \xi, \varphi(\tau, \xi)) \frac{\partial \varphi}{\partial n}(\tau, \xi)+\left\langle D_{x} \chi(\tau, \xi, \varphi(\tau, \xi)), \nabla \Psi(x)\right\rangle +h(\tau, \xi, \eta(\tau, \xi, \varphi(\tau, \xi))) \\
									= & D_{y} \chi(\tau, \xi, \varphi(\tau, \xi))\left(\frac{\partial \varphi}{\partial n}(\tau, \xi)+\widetilde{h}(\tau, \xi, \varphi(\tau, \xi))\right)
								\end{aligned}
								$$
								where
								\begin{equation}\label{form h}
									\widetilde{h}(t, x, y)=\frac{1}{D_{y} \chi(t, x, y)}\left(h(t, x, \chi(t, x, y))+\left\langle D_{x} \chi(t, x, y), \nabla \Psi(x)\right\rangle\right).
								\end{equation}
								Since $D_{y} \chi(t, x, y)>0$, we obtain, $\mathbb{P}$-almost surely on the event $\{0<\tau<T\} \cap\{\xi \in \partial \mathrm{G}\}$, the inequality
								\begin{equation}\label{inh}
									\min \left[\mathcal{L}^t_{\tilde{f}, 0}(\varphi(\tau, \xi))-D_{t} \varphi(\tau, \xi),-\frac{\partial \varphi}{\partial n}(\tau, \xi)-\widetilde{h}(\tau, \xi, \psi(\tau, \xi))\right] \leq 0.
								\end{equation}
								Combining  \eqref{inf} and  \eqref{inh}, we obtain that the random field $v$ is a stochastic viscosity subsolution of the $\operatorname{SIPDE}(\widetilde{f}, 0, \widetilde{h})$, which concludes the proof of Proposition \ref{propsition Doss}.
							\end{proof}
							
							The main objective of this section is to show how a SIPDE with coefficients $(f, g, h)$ is related to equation \eqref{basic equation} introduced in Section 1.

							To this end, we need the following result which is proved in Buckdahn and Ma [2]:
							\begin{proposition}
								Assume $\left(\mathcal{H}.{5}\right)$ holds. Let $\chi$ be the unique solution to $\operatorname{SDE}$ \eqref{eq chi} and $\pi$ be the $y$-inverse of $\chi$. Then, there exists a constant $C>0$, depending only on the bound of $g$ and its partial derivatives, such that for $\zeta=\chi$, and $\zeta=\pi$, the following inequalities hold $\mathbb{P}$-almost surely for all $(t, x, y) \in[0, T] \times \mathbb{R}^{l} \times \mathbb{R}$:
								$$
								\begin{aligned}
									& |\zeta(t, x, y)| \leq|y|+C\left|B_{t}\right| \\
									& \left|D_{x} \zeta\right|,\left|D_{y} \zeta\right|,\left|D_{x x} \zeta\right|,\left|D_{x y} \zeta\right|,\left|D_{y y} \zeta\right| \leq C \exp \left\{C\left|B_{t}\right|\right\}
								\end{aligned}
								$$
								Here all the derivatives are evaluated at $(t, x, y)$.	
							\end{proposition} 
							Next, for $t \in[0, T]$ and $x \in \bar{G}$, let us define the following processes:
							$$
							\begin{aligned}
								& U_{s}^{t, x}=\chi\left(s, X_{s}^{t, x}, Y_{s}^{t, x}\right), \quad 0 \leq t \leq s \leq T \\
								&{ V_{s}^{t, x}=\int_{\mathbb{R}} 
									\Bigl[
									\chi\bigl(t,X_{s}^{t, x}+\sigma(X_{s}^{t, x})r,\, Y_{s}^{t, x}\bigr) 
									- \chi
									- \bigl(D_x\chi + D_y\chi \cdot Z^{t,x}_s\bigr)\sigma(X_{s}^{t, x})r
									\Bigr] 
									F_s(dr)}
							\end{aligned}
							$$
							Then, from Proposition 4.5, we obtain
							$$
							\left(\left(U_{s}^{t, x}, V_{s}^{t, x}\right),(s, x) \in[0, T] \times \overline{\mathrm{G}}\right) \in \mathcal{S}^{2}(F ;[0, T] ; \mathbb{R}) \times \mathcal{M}^{2}\left(F ;[0, T] ; \mathbb{R}^{d}\right)
							$$
							
							\begin{theorem}
								For each $(t, x) \in[0, T] \times \overline{\mathrm{G}}$, the process $\left(U_{s}^{t, x}, V_{s}^{t, x}, t \leq s \leq T\right)$ is the unique solution to the following generalized BSDE:
								\begin{equation}
									U_{s}^{t, x}=  H\left(X_{T}^{t, x}\right)+\int_{s}^{T} \tilde{f}\left(t, X_{r}^{t, x}, U_{r}^{t, x}, V_{r}^{t, x}\right) \mathrm{d} r 
									+\int_{s}^{T} \tilde{h}\left(r, X_{r}^{t, x}, U_{r}^{t, x}\right) \mathrm{d} k_{r}^{t, x}-\sum_{i=1}^{d}\int_{s}^{T}V^{t,x(i)}_r dH^{(i)}_r, \label{2555}
								\end{equation}
								where $\tilde{f}$ and $\tilde{h}$ are given by (24) and (26).	
							\end{theorem}
							Here, the argument follows the same procedure as in \cite{boufoussi} (Theorem 3.4). Therefore, we
							omit the proof.

							Now, define for each $(t,x) \in[0, T] \times \overline{\mathrm{G}}$ the random fields $u$ and $v$ by $u(t, x)=Y_{t}$ and $v(t, x)=U_{t}$, where $(Y, Z)$ and $(U, V)$ are the solutions to the equations \eqref{14}  and \eqref{2555}, respectively. Then, we have
							\begin{equation*}
								u(t, \omega, x)=\eta(\omega, t, x, v(t, \omega, x)), \quad v(t, \omega, x)=\varepsilon(\omega, t, x, u(t, \omega, x)) \tag{40}
							\end{equation*}
							\begin{theorem}
								The random field $u$ is a stochastic viscosity solution to the $\operatorname{SIPDE}(f,g,h)$.	
							\end{theorem}
							The proof follows the same arguments in \cite{boufoussi}.

							\subsection{The obstacle problem of SIPDE-NBC}
							First, let us consider the following assumption on the barrier.
							\begin{itemize}
								\item [$(\mathcal{H}.5)$]  $\mathcal{S} :[0,T]\times \bar{G}\to \mathbb{R}$  satisfies for all $s\in [0,T]$ and $p\geq 2$:
								\begin{enumerate}
									\item[(i)] $ \lvert \mathcal{S}(s,x)\rvert \leq C\left(1+\left|x\right|^p\right),$
									\item[(ii)] $\mathcal{S}\in C^{1,2} ~~\text{such that} \ \mathcal{S}(T,x)\leq H(x),~~\forall x\in \bar{G}.$
								\end{enumerate}
							\end{itemize}
							
							As in Remark \ref{rmk}, we have
							$$ \lvert \mathcal{S}(s,X^{t,x}_s)\rvert  \leq C \ \ \text{and} \ \
							\mathbb{E} \left[\sup_{0\leq s \leq T} \Phi_s^{\theta,\mu} \left|\mathcal{S}(s,X^{t,x}_s)\right|^2\right] \leq C\mathbb{E} \left[\Phi_s^{\theta,\mu} \left(1+\left|X^{t,x}_s\right|^2\right)\right]<\infty.$$

							Therefore, by Theorem \ref{thm} below we have the following result.
							
							\begin{corollary}
								For all $(t,x)\in [0,T]\times \bar{G}$, there exists a triplet of processes $(Y^{t,x}_s,Z^{t,x}_s,K^{t,x}_s)_{s \in [t,T]}$ solution to the Markovian Reflected GBDSDE-NL:
								\begin{multline}\label{14}
									Y^{t,x}_s= H(X^{t,x}_T) + \int_{s}^{T} f(r,X^{t,x}_r, Y^{t,x}_r,Z^{t,x}_r)dr + \int_{s}^{T} h(r,X^{t,x}_r, Y^{t,x}_r)d\kappa^{t,x}_r + \int_{s}^{T} g(r,X^{t,x}_r, Y^{t,x}_r)\overleftarrow{dB}_r \\
									+K^{t,x}_T -K^{t,x}_s	-\sum_{i=1}^{d}\int_{s}^{T}Z^{t,x(i)}_r dH^{(i)}_r,~~~~t\leq s \leq T
								\end{multline}
								and there exists $C>0$ such that
								\begin{multline}
									\mathbb{E}\left[\sup_{t\leq s\leq T}\Phi_s^{\theta,\mu}|Y^{t,x}_s|^2\right] +	\mathbb{E}\int_t^T \Phi_s^{\theta,\mu} |Y^{t,x}_s|^2 d\mathcal{Q}_s
									+\mathbb{E}\int_t^T \Phi_s^{\theta,\mu} \left\|Z^{t,x}_s\gamma_s \right\|_{\ell^2}^2 ds + \mathbb{E} \left|K^{t,x}_T-K^{t,x}_t \right|^2\\ \leq C\left[1+\mathbb{E}\left(\underset{t\leq s \leq T}{\sup} \Phi_s^{\theta,\mu}|X^{t,x}_s|^2\right)\right]<\infty.
								\end{multline}
							\end{corollary}
							Using the same arguments as in \cite{pardoux2014stochastic}, we get the following continuity result.
							\begin{proposition} 
								Let $(Y_{s}^{t,x}, Z_{s}^{t,x})$ be a solution to \eqref{14}. Then, the random field $(s,t,x) \rightarrow Y_{s}^{t, x}, \ (s,t,x) \in[0,T] \times[0,T] \times \overline{\mathrm{G}}$ is almost surely continuous.
							\end{proposition}
							With the help of the result obtained in Section \ref{sec4.3} and by the same method used in  \cite{aman2012reflected} (Theorem 4.2, pp.1168), we establish the existence of a stochastic viscosity solution to the following obstacle problem of SIPDE-NBC:
							\begin{equation} \label{36}
								\left\{
								\begin{array}{lll}
									\displaystyle{  \big(u(t,x)-\mathcal{S}(t,x)\big)
										\wedge\left( - \frac{\partial u}{\partial t}(t,x)- \mathcal{L}_tu(t,x)-f\left(t,x,u(t,x),\left(u^1_k(t,x)\right)^{\infty}_{k=1}\right)\right)}\\\\
									\hspace{7cm}\displaystyle+ g\left(t,x,u(t,x)\right)dB_t =0,& (t,x)\in [0,T[ \times G\\\\
									\displaystyle u(T,x)=H(x), & x\in \bar{G} \\\\
									\displaystyle \frac{\partial u}{\partial n}(t,x)+h(t,x,u(t,x))=0,& x\in \partial G,
								\end{array}
								\right.
							\end{equation}

							\appendix
							\section{Special generalized backward doubly SDEs with Lipschitz drivers}
							The existence and uniqueness results for generalized BDSDEs with jumps and a non-homogeneous Lévy process (GBDSDE-NL) are provided in this section for the specific case where the coefficients $f$ and $g$ are Lipschitz and depend only on $y$, while the driver $g$ does not depend on the variables $(y, z)$. We first consider the case of standard generalized BDSDE-NL without reflection. Then, using these results, we extend the analysis to the reflected case with a continuous obstacle $S$.
							
							We denote $\mathfrak{f}(t,y) := f(t, y, z)$ and $\mathfrak{g}(t) := g(t, y, z)$ for any $(t, y, z) \in [0,T] \times \mathbb{R} \times \ell^2$. Toward the end of this section, the notations from the previously mentioned assumption \textbf{(H-M)} are used, along with the same satisfied properties.
							\subsection{Classical generalized backward doubly SDEs}
							Consider the following  GBDSDE-NL:
							\begin{equation}\label{Classical basic}
								Y_t=\xi+\int_{t}^{T}\mathfrak{f}(s,Y_s)ds+\int_{t}^{T}h(s,Y_s)d\kappa_s+\int_{t}^{T}\mathfrak{g}(s)\overleftarrow{dB}_s-\sum_{k=1}^{d}\int_{t}^{T}Z^{(k)}_s dH^{(k)}_s,\quad t \in [0,T].
							\end{equation}
							\begin{theorem}[Existence and uniqueness]\label{Thm A1}
								Assume that 
								\begin{description}
									\item[(H-LP)]
									\begin{itemize}
										\item There exists a constant $\mathfrak{b} > 0$ such that, for any $t \in [0,T]$ and $y, y' \in \mathbb{R}$, we have
										$$
										\left|\mathfrak{f}(t, y) - \mathfrak{f}(t, y')\right| + \left|h(t, y) - h(t, y')\right| \leq \mathfrak{b} \left|y - y'\right|.
										$$
										
										\item For any $t \in [0,T]$, we have $\left|\mathfrak{f}(t, 0)\right| \leq \varphi_t$ and $\left|h(t, 0)\right| \leq \psi_t$.
										
										\item For all $\mu > 0$:
										\begin{itemize}
											\item $\mathbb{E}\left(e^{\mu \kappa_T} |\xi|^2 \right) < + \infty.$
											
											\item $\mathbb{E} \left( \int_0^T e^{\mu \kappa_t} \left( \left|\frac{\varphi_t}{a_t}\right|^2  +  \left| \mathfrak{g}(t) \right|^2 \right) dt \right)
											+ \mathbb{E} \left( \int_0^T e^{\mu \kappa_t} \left|\psi_t\right|^2 d\kappa_t \right) < + \infty$.
										\end{itemize}
									\end{itemize}
								\end{description}
								Then the BDSDE-NL (\ref{Classical basic}) admits a unique solution $(Y ,Z) \in \left(\mathcal{S}^2_\mu \cap \mathcal{H}^{2,A}_\mu\right) \times \mathcal{H}^{2,\ell^2}_\mu$.
							\end{theorem}
							\begin{proof}
								By adopting a similar approach to that used in \cite[Theorem 3.2]{hu2009stochastic}, along with the predictable representation theorem for non-homogeneous Lévy processes provided in \cite[Theorem 2.4]{jamali2019predictable} (see also \cite[Section 2.3. Example 3]{jamali2019predictable}), we can establish the existence of a pair $(Y, Z) \in \mathcal{S}^2 \times \mathcal{H}^{2,\ell^2}$ that satisfies the BDSDE-NL (\ref{Classical basic}). The integrability condition is verified by following the same computations as in Proposition \ref{Propo 1}, where it suffices to prove the result in the case where $\kappa_T$ is a bounded random variable, and subsequently apply Fatou's Lemma (see, for instance, \cite[Proposition 1]{elmansourielotmani2024} for a similar argument). Finally, the uniqueness result can be deduced in a similar way, following the argument in Proposition \ref{Propo 2}.
							\end{proof}
							\begin{theorem}[Comparison theorem]\label{Thm A2}
								Assume that condition \textbf{(H-LP)} holds. Let $(Y^i, Z^i)$ be a solution to the BDSDE-NL (\ref{Classical basic}) associated with the parameters $(\xi^i, \mathfrak{f}^i, h, \kappa, \mathfrak{g})$ for $i = 1, 2$. If $\xi^1 \leq \xi^2$ and $\mathfrak{f}^1(t, y) \leq \mathfrak{f}^2(t, y)$ for all $(t, y) \in [0,T] \times \mathbb{R}$, then $Y^1_t \leq Y^2_t$ a.s. for all $t \in [0,T]$.
							\end{theorem}
							
							\begin{proof}
								The proof follows the same reasoning as in the proof of \cite[Lemma 3.1]{aman2012reflected}. Hence, we omit the details.
							\end{proof}
							\subsection{Special generalized reflected backward doubly SDEs}
							We consider the following version of the RGBDSDE-NL (\ref{basic equation}):
							\begin{equation}
								\left\{
								\begin{split}
									\text{(i)}&~Y_t=\xi+\int_{t}^{T}\mathfrak{f}(s,Y_s)ds+\int_{t}^{T}h(s,Y_s)d\kappa_s+\int_{t}^{T}\mathfrak{g}(s)\overleftarrow{dB}_s+\left(K_T-K_t\right)-\sum_{k=1}^{d}\int_{t}^{T}Z^{(k)}_s dH^{(k)}_s,\\
									\text{(ii)}&~Y_t \geq S_t, ~t \in [0,T] ~\mbox{ and }~\int_{0}^{T} \left(Y_{s}-S_{s}\right)dK_s=0.
								\end{split}
								\right.
								\label{basic equation LP}
							\end{equation}
							\begin{theorem}\label{Thm A3}
								Assume that the data $\left(\xi, \mathfrak{f}, h, \kappa, \mathfrak{g}\right)$ satisfy assumption \textbf{(H-LP)} and that the obstacle $S$ satisfies $\mathbb{E}\sup_{t \in [0,T]} \left| e^{\mu} S^{+}_{t} \right|^2 < +\infty$ for any $\mu > 0$. Then the RGBDSDE-NL (\ref{basic equation LP}) admits a unique solution $\left(Y, Z, K\right) \in \left(\mathcal{S}^2_\mu \cap \mathcal{H}^{2,A}_\mu\right) \times \mathcal{H}^{2,\ell^2}_\mu \times \mathcal{A}^2$.
							\end{theorem}
							
							\begin{proof}
								The proof is quite standard, using Theorems \ref{Thm A1} and \ref{Thm A2}, and following a similar argument as in \cite[Proposition 3.2]{aman2012reflected}. The integrability property follows from Proposition \ref{Propo 1}, in a manner similar to that established in Theorem \ref{Thm A1}.
							\end{proof}
							

							\section*{Disclosure statement}
							No potential conflict of interest was reported by the authors.

\end{document}